\documentclass[preprint,12pt]{elsarticle}
\usepackage[margin=1in]{geometry}

\usepackage{graphicx}
\usepackage{amssymb}
\usepackage{enumitem}

\usepackage{lineno}
\usepackage{amsmath}
\usepackage{amssymb}
\usepackage{amsfonts}
\usepackage[section]{placeins}
\usepackage{float}
\usepackage{graphics} 
\usepackage{epsfig}

\newtheorem{corollary}{Corollary}
\newtheorem{lemma}{Lemma}
\newtheorem{definition}{Definition}

\newtheorem{claim}{Claim}
\usepackage{mwe}
\usepackage{subfig}
\usepackage{color}
\usepackage{hyperref}
\usepackage{comment}
\usepackage{multirow}
\usepackage{array}
\usepackage[dvipsnames]{xcolor}

\allowdisplaybreaks

\makeatletter
\DeclareRobustCommand{\qed}{%
  \ifmmode % if math mode, assume display: omit penalty etc.
  \else \leavevmode\unskip\penalty9999 
  \hbox{}\nobreak\hfill
  \fi
  \quad\hbox{\qedsymbol}}
\newcommand{\openbox}{\leavevmode
  \hbox to.77778em{%
  \hfil\vrule
  \vbox to.675em{\hrule width.6em\vfil\hrule}%
  \vrule\hfil}}
\newcommand{\qedsymbol}{\openbox}
\newenvironment{proof}[1][\quad\proofname]{\par
  \normalfont
  \topsep6\p@\@plus6\p@ \trivlist
  \item[\hskip\labelsep\itshape
    #1.]\ignorespaces
}{%
  \qed\endtrivlist
}
\newcommand{\proofname}{Proof}
\makeatother

\usepackage{algorithm}
\usepackage{algpseudocode}

\algtext*{EndFor}% Remove "end for" text
\algtext*{EndIf}% Remove "end if" text
\algtext*{EndFunction}% Remove "end function" text
\algtext*{EndWhile}% Remove "end while" text

\newcommand{\sfunction}[1]{\textsf{\textsc{#1}}}
\algrenewcommand\algorithmicforall{\textbf{foreach}}
\algrenewcommand\algorithmicindent{.8em}

\journal{}

\begin{document}

\begin{frontmatter}


\title{Shareability Network Based Decomposition Approach for Solving Large-scale Single School Routing Problems}



\author[mitcee]{Xiaotong Guo\corref{cor}}
\ead{xtguo@mit.edu}
\author[cornellcee]{Samitha Samaranayake}
\ead{samitha@cornell.edu}

\address[mitcee]{Department of Civil and Environmental Engineering, Massachusetts Institute of Technology, Cambridge, MA 02139, USA}
\address[cornellcee]{School of Civil and Environmental Engineering, Cornell University, Ithaca, NY 14853, USA}
\cortext[cor]{Corresponding author}

\begin{abstract}
% We consider the classic School Bus Routing Problem (SBRP) with a generalization for allowing alternate modes, where students are either picked up by a fleet of school buses or transported by an alternate transportation mode, subject to a set of constraints.
We consider the Single School Routing Problem (SSRP) where students from a single school are picked up by a fleet of school buses, subject to a set of constraints.
The constraints that are typically imposed for school buses are bus capacity, a maximum student walking distance to a pickup point, and a maximum commute time for each student.
This is a special case of the Vehicle Routing Problem (VRP) with a common destination. 
We propose a decomposition approach for solving this problem based on the existing notion of a shareability network, which has been used recently in the context of dynamic ridepooling problems. 
Moreover, we come up with a simplified formulation for solving the SSRP by introducing the connection between the SSRP and the weighted set covering problem (WSCP).
To scale this method to large-scale problem instances, we propose i) a node compression method for the shareability network based decomposition approach, and ii) heuristic-based edge pruning techniques that perform well in practice. 
We show that the compressed problem leads to an Integer Linear Program (ILP) of reduced dimensionality that can be solved efficiently using off-the-shelf ILP solvers. Numerical experiments on the synthetic Boston Public School (BPS) instances are conducted to evaluate the performance of our approach.
Meanwhile, our proposed SSRP formulation allows a natural extension for introducing alternate transportation modes to students, which effectively reduces the number of buses needed for each school and leads to a 15\% cost reduction on average.
Moreover, two state-of-art large-scale SSRP solving techniques are compared with our proposed approaches on benchmark networks and our methods outperform both techniques under a single school setting.
%even without considering alternate modes.

\end{abstract}


\end{frontmatter}

%\linenumbers

\section{Introduction}
\label{sec:intro}

According to the American School Bus Council, nearly 480,000 school buses transported 25 million students to and from school and school-related activities every school day in 2010~\cite{ASBC}.  Meanwhile, based on a report from the Nation Center for Education Statistics, 23 billion dollars were spent on public school transportation during the academic year 2013-2014, which is nearly 5 percent of the total expenditures for public schools \cite{DIGEST}. Every dollar spent on transporting students is a dollar lost for direct spending to improve the education of students. Therefore, an efficient and economical operation of school bus systems is of significant importance to school districts that are trying to make the most of their limited education budgets.

The major costs associated with operating a school bus service are the capital and operational costs of the buses and wages of drivers. 
Thus, an efficient solution will serve students by using the fewest buses possible\footnote{minimizing the total distance traveled is a secondary objective.}. 
Under the single school setting, this needs to be done subject to getting everyone to the school on time and not making some students spend a very long time sitting on a bus (e.g. one hour maximum riding time in Boston). 
This leads to the so-called Single School Routing Problem (SSRP). 

The SSRP is a generalization of the metric Traveling Salesman Problem (TSP) and a special case of the Vehicle Routing Problem (VRP), both of which are NP-hard problems~\cite{VRP}. 
While the metric TSP has a number of good approximation techniques for obtaining provable guarantees on the solution accuracy, the VRP and SSRP problems are harder to approximate and typically solved using heuristic techniques.  
Therefore, state-of-the-art methods for solving SSRP can only solve small-scale problems optimally. To solve the SSRP at scale, the problem is typically formulated as an Integer Linear Program (ILP) and solved using different heuristics techniques~\cite{PARK2012204,RIERALEDESMA2012391,SCHITTEKAT2013518}. 
One limitation of these approaches is that they lead to high dimensional ILP problems that have extremely large decision spaces, and are hard to solve well at-scale even with very sophisticated heuristic techniques.

This paper proposes a new approach for solving the SSRP using a decomposition techniques based on the notion of a shareability network~\cite{Santi13290}. Compared to classical SSRP approaches, our decomposition leads to a much simpler ILP formulation that can be solved more efficiently at scale. Our approach utilizes the following steps:

\begin{itemize}
\item Decoupling the bus routing and student matching problems via the construction of a shareability network and a student-trip assignment graph.
\item Using a node compression technique for the shareability network by assigning students to bus stops subject to maximum walking constraints. 
\item Using a set of heuristic-based edge pruning techniques for the shareability network to delete edges and compress the feasible bus routing set.
\end{itemize}

Steps described above lead to a much simpler ILP. 
For extreme large-scale problems, node compression and edge pruning techniques for the shareability network can be combined with the traditional large-scale ILP heuristics to obtain solutions more efficiently (column generation for instance).

Furthermore, our approach also naturally allows for incorporating alternate transportation modes in the SSRP. For example, the model can assign some students to an external travel mode (e.g., school district contracts with private transportation providers or public transportation systems), which leads to a more efficient school bus schedule and in particular can reduce the number of buses needed. Alternate modes have been utilized by the Boston Public Schools (BPS) when designing school bus schedules, where students in 6-th grade and above have options to receive a discounted MBTA\footnote{Massachusetts Bay Transportation Authority (MBTA) is the transit agency operating public transportation services in Greater Boston area.} card and take public transportation to schools~\cite{BPS_principle}.\\

\noindent The contributions of this article can be summarized as follows:
\begin{enumerate}
\item Modeling the SSRP using the shareability network framework (used in high-capacity ridepooling context), defining the corresponding student-trip graph and formulating the corresponding ILP problem.
\item Connecting SSRP with the weighted set covering problem (WSCP) and simplifying the ILP for solving SSRPs.
\item Showing that techniques used in high-capacity ridepooling with the shareability network can not be applied to the SSRP directly, due to the density of the resulting shareability network, and developing network compression techniques to improve the tractability of the problem. 
\item Displaying numerical results to validate performances of our approach in solving large-scale SSRP problems efficiently. 
Conducting benchmark comparisons against two different state-of-the-art approaches for solving SSRP problems and showing the relative performance of our proposed methods.
\item Generalizing the standard SSRP to a SSRP with alternate modes and showing how our approach naturally extends to this setting. System-wide savings have been found when introducing alternate modes in SSRPs.
\end{enumerate}

The remainder of the article is organized as follows. 
Section \ref{sec:liter} reviews the related literature. 
Section \ref{sec:problem} provides basic definitions for the SSRP and the generalization to multiple modes. 
The model formulation for our decomposition approach via a compressed shareability network is shown in Section \ref{sec:method}. 
Section \ref{sec:experiments} describes numerical experiments, benchmark comparisons and sensitivity analyses for our proposed approach.
Finally, Section \ref{sec:disc} recaps the main ideas of this work and lists limitations and future research directions.

\section{Literature Review}
\label{sec:liter}

The SSRP is a sub-problem of the classic School Bus Routing Problem (SBRP), which has been studied since 1969 when Newton and Thomas first proposed a method to generate school bus routes and schedules~\cite{NEWTON196975}. 
\citet{PARK2010311} did a broad review of the SBRP prior to 2010.
\citet{ELLEGOOD2020102056} conducted a comprehensive review of the SBRP during the past decade and pointed out contemporary trends and research directions.
The SBRP is decomposed into five sub-problems including bus stop selection, bus route generation, bus route scheduling, school bell time adjustment and strategic transportation policy.
The primary focus of this article is on the bus route generation aspects of the SBRP under a single school context\footnote{We also consider bus stop selection, but this is not the primary focus of our work.}, which we will refer to as the SSRP in this paper.

% A comprehensive review of the SBRP can be found in Park and Kim~\cite{PARK2010311}, where the SBRP is decomposed into five steps including data preparation, bus stop selection, bus route generation, school bell time adjustment and route scheduling. This article focuses on solving the bus stop selection and bus route generation aspects of the SBRP, which is we refer to as the SBRP (a slight abuse of notation that is common in the literature). 

Different problem settings for the SBRP have been considered throughout the recent literature.
Problem settings of the SBRP depend on the number of schools (single school or multiple school), the service environment (urban, rural or both), the bus fleet (homogeneous or heterogeneous), objectives and constraints.
The typical objective of the SBRP is cost minimization and the common constraints considered are bus capacity, time windows and maximum ride time. 
The detailed problem setting in this paper will be presented in the next section.

For the single school SBRP, \citet{Bekta2007} proposed an ILP model based on the open vehicle routing problem (OVRP), in which vehicles do not return to the depot after serving the last demand. They solve the real-world single school SBRP for transporting students of an elementary school in central Ankara, Turkey.
They considered a capacity constraint for vehicles and a maximum travel distance constraint for students, and an objective of minimizing the bus operating cost. 
This paper provided a basic mathematical formulation of the SBRP. 
\citet{Sghaier_Guedria_Mraihi} proposed several modified genetic algorithms to solve the single school SBRP with capacity and maximum travel distance constraints under an urban setting.  
Performances of different algorithms were tested on a simulated instance with 519 students and 30 stops.

In a multiple school SBRP setting, mixed load routing, where one school bus can transport students from multiple schools, is an option that is sometimes considered.
\citet{Ellegood_Campbell_North_2015} utilized a continuous approximation model to evaluate the condition under which the mixed load routing was beneficial to schools.
They showed that the mixed load routing strategy was most beneficial for large school districts where schools were close to each other and a large percentage of bus stops were shared by students from different schools.
\citet{PARK2012204} also developed a mixed load algorithm for the multi-school SBRP. 
The problem was modeled using an ILP and solved by a post-processing algorithm applied to a single-school load solution. 
The proposed algorithm was an improvement on the mixed load algorithm provided by \citet{BRACA1997}, which addressed the New York City school bus routing problem. 

% Shafahi et al.~\cite{ALI2017} proposed a new formulation of the SBRP with a homogeneous fleet that maximized trip compatibility (two trips are compatible if they can be served by the same bus) while minimizing the total travel time, and generated eight mid-size data sets to test the performance of the model. 

The literature on solving large-scale SBRPs are dominated by heuristic approaches. 
\citet{RIERALEDESMA2012391} solved the large-scale SBRP by modeling it as an instance of the multi-vehicle traveling purchaser problem, which is a generalization of the VRP. 
The LP-relaxation method was used to efficiently solve the high dimensional ILP and a heuristic algorithm was proposed to round the fractional results. This approach was tested by using synthetic data and shown to solve instances with up to 125 students. 
\citet{SCHITTEKAT2013518} proposed a sophisticated ILP considering both bus stop selection and bus route generation simultaneously and used a metaheuristic approach to solve the problem. 
The metaheuristic approach contains two steps: i) a route construction phase that uses a greedy randomized adaptive search procedure to compute sub-optimal initial solutions, and ii) an improvement phase where a variable neighborhood descent method is used to ensure a local optimum in all neighborhoods.
The method could produce high quality solutions within one hour for problems of up to 80 stops and 800 students. Generated instances from this article were used as one of the benchmarks for testing our proposed approaches.  

More recently, \citet{Shafahi_Wang_Haghani_2018} solved the multi-school SBRP by utilizing a Minimum Cost Matching with Post Improvement (MCMPI) algorithm, which was a two-stage metaheuristic approach including a cost-minimizing trip generation algorithm and a post-improvement simulated annealing algorithm.
\citet{Miranda_2018} utilized an iterated local search approach to solve the multi-load SBRP, which extended the mixed load setting by allowing students to be picked up and dropped off simultaneously. 
\citet{Sales_Melo_Bonates_Prata_2018} proposed a memetic algorithm (a type of genetic algorithm) to solve a heterogeneous fleet SBRP, where stop generation, route generation and stop selection problems were considered.
\citet{Mokhtari_Ghezavati_2018} designed a bi-objective ant colony optimization algorithm to solve the SBRP with mixed loads. 
Their proposed algorithm minimized both the number of buses and the average travel time of students.

In one of the most high profile recent SBRP results, \citet{MIT_SBRP} proposed an optimization model for the School Time Selection Problem (STSP), which is a generalization of the school bus routing problem that includes reusing the same bus fleet over multiple rounds of trips (e.g. for a school with a 7:30 am start time followed by one with a 8:30 am start time). A state-of-the-art school bus routing algorithm, named BiRD (Bi-objective routing decomposition), was proposed.
The BiRD algorithm consists of generating single-school bus routes as sub-problems and combining sub-problems via mixed-integer optimization to identify a trip-by-trip itinerary for each bus in the fleet.
The implementation of their approach was claimed to lead to a $\$$5 million annual cost saving for Boston Public Schools.
The single school bus routing component of BiRD algorithm will serve as the second benchmark to test our proposed methods in the experiments section.

In summary, most of recent papers used ILP as a basic approach and concentrated on proposing heuristic techniques to improve efficiency for solving the ILP.
To improve the efficiency and accuracy of current approaches, this work proposes a shareability network based decomposition approach to solve large-scale SBRPs under a single school context and conducts experiments based on synthetic instances derived from real-world dataset.

Our approach for modeling the SSRP is an extension of techniques used for  dynamic high-capacity ridepooling problems~\cite{Alonso-Mora462}, which is a special case of the dynamic capacitated VRP with time windows.
In the ridepooling context, \citet{Alonso-Mora462} proposed a decomposition approach via the shareability network to got a lower dimensional ILP that was computationally tractable, and this approach is extended to the SSRP in this paper. 
The notion of the shareability network, first described by \citet{Santi13290}, is utilized to efficiently compute optimal sharing strategies on a large dataset. There have been many follow up works to~\cite{Alonso-Mora462} that apply the idea of a shareability network to ride-pooling problems~\cite{Kucharski2020, Tafreshian2020, Luo2021, Syed2019, Simonetto2019}. For more details aboutride-pooling problems, we also refer readers to the survey paper by \citet{Wang_Yang_2019}.

% More recently, \citet{Kucharski2020} leveraged a directed shareability network with an efficient graph search algorithm to match trips into attractive shared rides for an offline ride-pooling problem. They proposed a utility-based formulation which incorporated fares, delays and on-board discomfort perspectives in order to find attractive shared rides. 

% Global pandemic caused by the COVID-19 has a significant impact on public school operations, including school bus schedules. Sharing the same trip across multiple students could increase the risk of spreading the virus. Although the viral transmission is not the main focus of this paper, how the virus spreading in a ride-pooling network and corresponding control strategies can be found in \citet{Kucharski2021}, which is a similar problem to students sharing in the SBRP context.

A preliminary version of this work was presented at the 2018 Intelligent Transportation Systems Conference~\cite{Xiaotong}. In this extended article, we introduce the following new contributions: i) More advanced network compression techniques. In particular, we propose an algorithm to compensate for the optimality loss induced by the edge pruning technique; ii) A generalization that allows for solving the SSRP with alternate modes; iii) The connection between the weighted set cover problem (WSCP) and SSRP; and iv) A significantly extended section on synthetic BPS experiments and comparisons with two state-of-the-art algorithms to demonstrate performances of our proposed methods.

\section{Problem Formulation}
\label{sec:problem}

\begin{table}[p]
\centering
\caption{Notations used in this paper}
\label{tab:var}
\begin{tabular}{ll}
\hline
\underline{Problem Formulation}\\
$G_r=(V_r,E_r)$ & Road network $G_r$ with vertex set $V_r$ representing locations, and edge \\ 
& set  $E_r$ representing road segments\\
$d_{ij}, t_{ij}$ & Shortest path distance and travel time between vertices $i,j \in V_r$\\
$S$ & Set of students \\
$B$ & Set of vehicles (buses) \\
$M$ & Set of potential bus stops \\
$D$ & Set of students residence locations \\
$A$ & Set of alternate transportation modes \\
$C$ & Capacity of school buses \\
$t^{max}$ & Maximum riding time on buses for students \\
$t^{delay}_m$ & Delay time for picking up students at bus stop $m$ \\ 
$v_0,v_d$ & Bus depot vertex and school vertex \\
$d_s^{max}$ & Maximum walking distance for student $s$ \\
$N_s$ & Set of reachable bus stops for student $s$ \\
$N$ & Union of bus depot location, school location, potential bus stops \\& locations and student residence locations \\
$\alpha_{C_1}^b$ & Cost for operating (owning or leasing) a bus per day including labor\\
$\alpha_{C_2}^b$ & Cost for operating a bus per mile \\
$\alpha_{C}^a$ & Cost for taking an alternate mode $a \in A$ per mile\\
$x_{ijk}$ & Binary variable for whether bus $k$ travel from vertex $i$ to $j$ \\
$y_{ik}$ & Binary variable for whether bus $k$ visits vertex $i$ \\
$z_{isk}$ & Binary variable for whether bus $k$ picks up student $s$ at vertex $i$\\
$u_{sa}$ & Binary variable for whether student $s$ takes mode $a$ to school\\
\underline{Problem Decomposition}\\
$G=(V,E)$ & Shareability network with vertex set $V$ representing requests, and \\ & edge set $E$ representing shareability between requests\\
$G_{ST}=(V_{ST},E_{ST})$ & Student-Trip graph with vertex set $V_{ST}$ representing the union of \\ & students and feasible trips, and edge set $E_{ST}$ representing whether \\& students are involved in trips \\
$\tau \in T,\tau_b \in T_b,\tau_a^s \in T_a$ & Set of feasible trips, feasible bus trips and feasible trips for an \\& alternate mode $a$ \\
$S(\tau)$ & Set of students who participate in the trip $\tau$ \\
$C_{\tau}$ & Travel distance for any feasible trip $\tau \in T$ \\
$L(\tau)$ & Number of students in a feasible trip $\tau \in T$ \\
$x_{s\tau}$ & Binary variable for whether student $s$ is assigned to trip $\tau$  \\
$y_{\tau}$ & Binary variable for whether trip $\tau$ is chosen in the optimal trip set\\
\underline{Network Compression}\\
$x_m$ & Binary variable for whether picking the potential bus stop $m \in M$  \\
$\Bar{t}_{ij}$ & Adjusted travel time between any two nodes $i$ and $j$ \\
$\delta_{ij}$ & Detour factor for any pair of nodes $i$ and $j$\\
$n(m)$ & Number of students at any bus stop $m \in M$ \\
$\beta$ & Heuristic parameter for the edge pruning technique\\
$\gamma$ & Heuristic parameter for the $\gamma$-quasi-clique process in edge pruning\\
\hline
\end{tabular}
\end{table}

We first provide a formal definition of the Single School Routing Problem (SSRP). 
The problem description will be consistent throughout the paper and the notations used are listed in Table \ref{tab:var}.

Let $G_r(V_r,E_r)$ denote the road network. For any pair of nodes $i, j \in V_r$, $d_{ij}$ represents the shortest path distance between $i,j$, and $t_{ij}$ indicates the corresponding travel time.
Consider a set of students $S$ who need to be transported to a single destination (school) $v_d \in V_r$ with a homogeneous fleet of school buses $B$, in which each school bus has capacity $C$ students. 
We define that each student $s \in S$ is located at some location\footnote{The graph can be augment to model pickups (bus stops) between vertices if needed.} $v_s \in V_r$ and the set $D$ indicates students pickup locations, $D \subseteq V_r$. 
Moreover, we let $M$ represent the set of potential bus stops, where $M \subseteq V_r$. The delay time for picking up students at a bus stop $m \in M$ is denoted by $t^{delay}_m$. Finally, to model students' travel without school buses, we define a set of alternate modes $A$, each student $s \in S$ can either take a school bus or an alternate transportation mode $a \in A$ to the school.

Let $\alpha_{C_1}^b$ be the cost for leasing (or amortized capital cost of owning) a bus per day including the labor cost for drivers\footnote{The driver labor is a fixed cost that is independent of distance traveled per bus and is typically a dominating expense.}, $\alpha_{C_2}^b$ be the operating cost per bus per mile, and $\alpha_{C}^a$ be the cost for taking alternate mode $a$ per mile\footnote{For the simplicity, we assume the cost for the alternate mode $a$ has a linear relationship with the trip distance. This can easily be replaced by a more complex cost function.}.
The cost for taking alternate mode $a$ can be student-specific, where $\alpha_{C}^a = \infty$ indicates that students can only take buses to school.
The objective of this problem is to minimize the total cost for the school bus schedules of a single school.
Figure \ref{problem} illustrates an instance of the SSRP. 

\begin{figure*}[!h]
\centering
\includegraphics[scale=0.65]{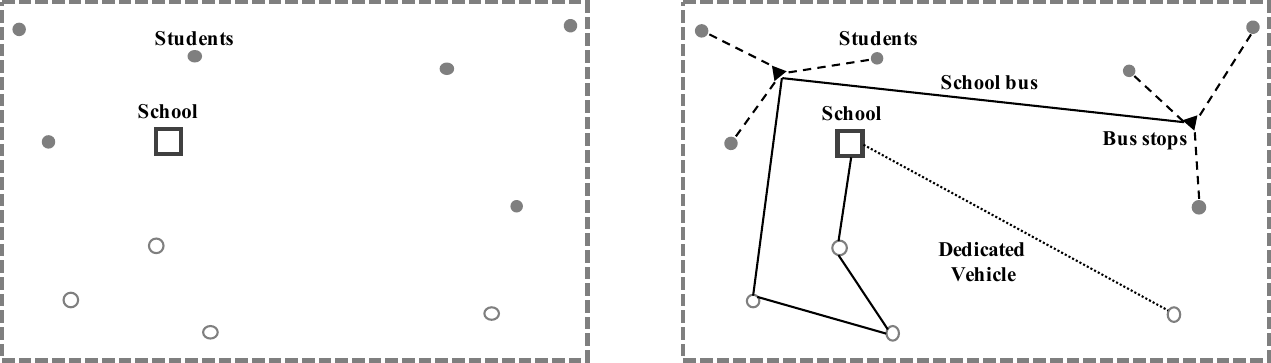}
\caption{Instance of a feasible solution for the SSRP considering the dedicated vehicle with bus capacity $C = 9$. Left figure gives the input for SSRP and right figure gives the results of the problem, which contains the bus routes and bus stops for each student. Circles represent students with door-to-door pickups. }
\label{problem}
\end{figure*}

We enforce the following constraints in the SSRP formulation: 
\begin{enumerate}
\item The maximum riding time any student $s$ can be on the school bus is $t^{\text{max}}$.
\item Each student $s \in S$ has a maximum walking distance $d_s^{\text{max}}$ from their residence to the assigned bus stop. This distance can be student specific and equal to zero if the students need door-to-door pickups. We let $N_s$ represent the set of reachable stops for a student $s$, i.e., $N_s = \{ m \in M \mid d_{v_sm} \leq d_s^{max}\}$.
\item All school buses start at a single pre-specified location $v_0 \in V_r$ and end at the school $v_d$. 
\end{enumerate}

We let $N = D \cup M \cup\{ v_0, v_d\} $ denote the set of pickup locations combined with potential bus stops, bus depot and school location. 
The decision variables for this problem are $x_{ijk}$, $y_{ik}$, $z_{isk}$ and $u_{sa}$, where $x_{ijk} = 1$ if bus $k$ travels from vertex $i$ to $j$ through the shortest path, $y_{ik} = 1$ if bus $k$ visits vertex $i$, $z_{isk} = 1$ if student $s$ is picked up by bus $k$ at vertex $i$ and $u_{sa} = 1$ if student $s$ takes an alternate mode $a$ to the destination. 
Assuming each bus stop or student home address can be visited by at most one bus, the ILP formulation for the SSRP considering alternate modes can be formulated as follows: 

\begin{align}
\min  \quad
& \alpha_{C_1}^b \cdot K +  \alpha_{C_2}^b \cdot \sum_{i \in N}\sum_{j \in N} d_{ij} \sum_{k \in B} x_{ijk} + \sum_{a \in A} \alpha_{C}^a \cdot \sum_{s \in S} u_{sa} d_{v_sv_d}\\
\text{s.t.} \quad
& \sum_{j \in N } x_{ijk} = \sum_{j \in N } x_{jik}  = y_{ik} \quad \forall i \in N \setminus \{ v_0,v_d\}, \forall k \in B \\
& \sum_{j \in N \setminus \{ v_0, v_d\}} x_{v_0jk} = \sum_{i \in N \setminus \{ v_0, v_d\}} x_{iv_dk} \quad \forall k \in B \\
& \sum_{i,j \in Q} x_{ijk} \leq |Q| - 1 \quad \forall Q \subseteq N, \forall k \in B \\
& \sum_{i \in N \setminus \{ v_0, v_d\}} \sum_{j \in N \setminus \{ v_0 \}} (t_{ij} + t^{delay}_{i}) \cdot x_{ijk} \leq t^{max} \quad   \forall k \in B\\
& \sum_{i \in N \setminus \{v_0, v_d\} } \sum_{s \in S} z_{isk} \leq C \quad \forall k \in B\\
& z_{isk} \leq y_{ik} \quad \forall i \in N, \forall s \in S, \forall k \in B \\
& \sum_{a \in A}u_{sa} + \sum_{i \in N_s \cup \{ v_s \}} \sum_{k \in B} z_{isk}  = 1 \quad \forall s \in S \\
& \sum_{k \in B} \sum_{i \in N \setminus \{v_d\}} x_{iv_dk} = K \leq |B| \\
& x_{ijk} \in \{0,1\} \quad \forall i,j \in N, \forall k \in B\\
& y_{ik} \in \{0,1\} \quad \forall i \in N, \forall k \in B \\
& z_{isk} \in \{ 0,1 \} \quad \forall i \in N, \forall s \in S, \forall k \in B \\
& u_{sa} \in \{0,1\} \quad \forall s \in S, \forall a \in A
\end{align}

The objective function (1) minimizes the overall school bus scheduling cost considering the number of buses, vehicle miles travel and alternate mode cost. 
Constraints (2) ensure that if bus $k$ visits pickup location $i$, then there will be a flow entering $i$ and a flow leaving $i$ for bus $k$. 
Constraints (3) impose that a bus entering the destination should also have left the depot.
Constraints (4) enforce sub-tour elimination, i.e. ensures a single connected route for bus $k$.
Constraints (5) consider the maximum travel time for each student by restricting the total travel time for each bus route starting from picking up the first student.
Constraints (6) enforce that the number of students in bus never exceed the capacity $C$. 
Constraints (7) ensure that student $s$ will not be picked up by bus $k$ at vertex $i$ unless bus $k$ visits vertex $i$.
Constraints (8) impose that student $s$ either takes an alternate mode or is picked up by a school bus.
Constraint (9) enumerates the number of non-idle buses and enforces the maximum number of available buses $|B|$.
Constraints (10) - (13) make sure that decision variables are binary.

This ILP formulation provides the optimal school bus schedules for a single school. 
% The single school setting reduces the problem complexity, and moreover, also corresponds to the real-world problem setting provided by BPS~\citep{BPS_data}. 
% Even though most school buses are operated by a school district, in an urban setting (e.g., Boston) there are a number of practical and legal reasons (e.g., jurisdiction over disputes between students from different schools) for not mixing students across multiple schools in a single bus.
While the problem can be formulated as an ILP, solving large-scale instances of this SSRP is intractable when using off-the-shelf ILP solvers. Therefore, solving SSRPs at scale in a computationally tractable manner requires utilizing some decomposition and heuristic methods. The following section describes our proposed new approach and the corresponding heuristics for solving large-scale SSRPs efficiently.

\section{Methodology}
\label{sec:method}

In this section, we propose a decomposition method based on the notional of a shareability network~\cite{Santi13290} and its application to ridepooling problems~\cite{Alonso-Mora462}, to solve the SSRP. Furthermore, we simplify the ILP for solving the SSRP by identifying its connection to the WSCP.
Finally, we improve the tractability of large-scale instances by introducing network compression techniques that effectively prune the shareability network, as the network can get intractably for large-scale offline problems like SSRPs. 
    
\subsection{Decomposition through the shareability network}

In order to reduce the complexity and dimensionality of the ILP for solving the SSRP, we propose a decomposition method via the shareability network, which consists of several steps leading to an assignment problem that yields a much-simplified ILP.

The shareability network~\cite{Santi13290} is an undirected graph $G_S = (V_S , E_S)$, where $V_S$ corresponds to the set of trips and each edge $(i,j) \in E_S$ indicates that trip $i$ can share a vehicle with trip $j$ under some compatibility constraints. 
The shareability network under the SSRP setting is constructed as follows.
The vertex set $V_S$ designates the set of student locations and each edge $(s_i, s_j) \in E_S$ reflects the fact that students $s_i$ and $s_j$ can share the same school bus (under a desired set of quality of service constraints).
For example, in our setting, students $s_i$ and $s_j$ can share the same bus if both students can be transported to the destination (school) $v_d$ within the maximum riding time $t^{max}$ using the same bus. 
Figure (\ref{share_instance}a) shows an instance of a shareability network for four students. 

\begin{figure}[h]
    \centering
    \subfloat[Shareability network]{{\includegraphics[width=5cm]{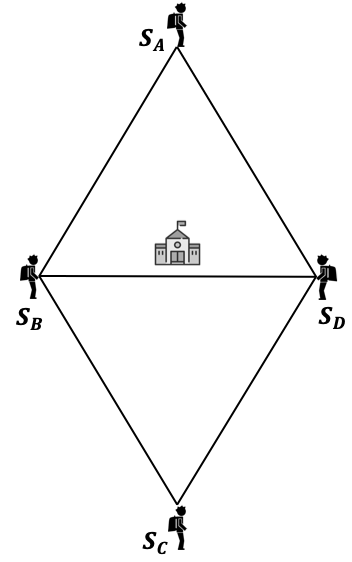} }}%
    \qquad
    \subfloat[Student-trip graph]{{\includegraphics[width=5cm]{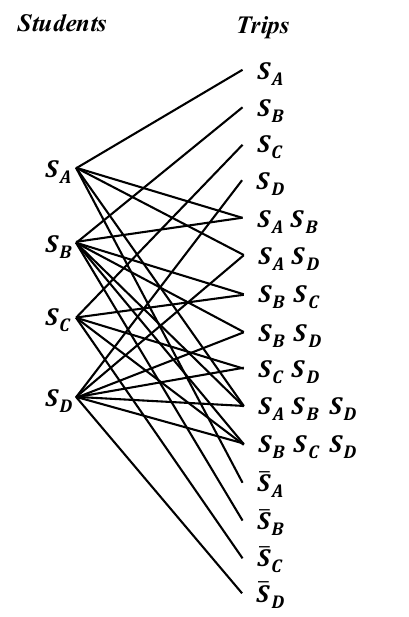} }}%
    \caption{Instance of the shareability network and ST-graph with 4 students. Each student can be assigned their personal bus, share the bus with others or take dedicated vehicles to school (represents by $\Bar{S}$ in the graph). In this instance, there are five feasible pairs of two students and two feasible pairs of three students.}%
    \label{share_instance}
\end{figure}

Next, we establish a bipartite graph $G_{ST} = (V_{ST}, E_{ST})$ where $V_{ST}$ contains a set of students and a set of all possible trips configurations (school bus or alternate modes assignment) based on the shareability network. 
This bipartite graph is referred as the student-trip graph (ST-graph).
The set of feasible trip configurations $T$ includes bus trips $T_b$ and trips $T_a$ for an alternate mode $a$, i.e., $T = \bigcup_{a \in A} T_a \cup T_b$.
Let $S(\tau_b)$ denote the set of students who participate in a feasible bus trip $\tau_b \in T_b$. 
A bus trip $\tau_b \in T_b$ is feasible if the total riding time for each student is less than or equal to the maximum allowed ($t_s \leq t^{max}, ~\forall s \in S(\tau_b)$) and the total number of students in the bus is smaller than its capacity ($|S(\tau_b)| \leq C$). 
For each student $s \in S$, $\tau_a^s \in T_a$ represents a non-school bus trip that student $s$ directly takes via alternate mode $a$ to the school.
The node set $V_{ST}$ is the union of the set of students and the set of feasible trips, i.e., $V_{ST} = S \cup T$, and there will be an edge $e(s,\tau_b) \in E_{ST}$ if $\tau_b \in T_b$, $s \in S(\tau_b)$, and an edge $e(s,\tau_a^s) \in E_{ST}$ for every $\tau_a^s \in T_a$. 
Figure (\ref{share_instance}b) shows an instance of ST-graph corresponding to the shareability network in Figure (\ref{share_instance}a). 

It is worth mentioning that alternate modes are introduced by extending the trip configurations $T$ while keeping the same problem structure as the standard SSRP.
The set of feasible bus trip $T_b$ is generated using the shareability network.
The following observation is typically made to efficiently compute the feasible bus trips in $T_b$ based on the shareability network $G$~\cite{Alonso-Mora462}.

\begin{lemma} ($Lemma\;\text{1}$ in \cite{Alonso-Mora462})
    \label{lemma}
    A trip configuration $\tau_b$ can only be feasible if the set of students $s$ in the trip configuration $\tau_b$ form a clique in the shareability network $G$ (i.e. $\forall s_i, s_j \in S(\tau_b), e(s_i,s_j) \; \text{exists}$). This is a necessary (but not sufficient) condition.\\
\end{lemma}

Given this observation, a potential trip configuration is not feasible if any pair of students in the trip configuration are not connected by edges in the shareability graph $G_S$. 
Thus, if a set of $n$ students ($s_1, \cdots, s_n$) can not form a feasible trip configuration, we know that any trip configuration that includes these students plus another student $s_{n+1}$ will certainly not be feasible. 
Therefore, as in~\cite{Alonso-Mora462}, we construct the set of feasible trip configurations $T_b$ by first considering trips that consist of one student, and progressively consider larger sets only when the smaller set is feasible. 
Algorithm \ref{alg1}\footnote{This extends the technique of enumerating feasible trips based on the shareability network, originally from~(\cite{Alonso-Mora462}), to the school bus routing problem.} describes the details for generating the feasible trip list $T_b$. 
The input function $\sfunction{PathTsp}(\cdot)$ is a black-box for solving the path traveling salesman problem (path-TSP). As this problem is NP-Hard~\cite{Laporte_1992}, we utilize an insertion heuristic based approach for solving the path-TSP problem. Details are provided in the \ref{apend:tsp}. We note that any other efficient Path-TSP heuristics can be substituted for this. 
% In general, this algorithm provides an efficient pruning mechanism that eliminates the consideration of infeasible bus trips.

\begin{algorithm}[p]
\caption{Generating the set of feasible bus trips. Input: the shareability network $G$, the set of students $S$, maximum riding time $t^{max}$, bus capacity $C$, path-TSP solver for any trip $\tau$ with optimal travel time $t^*$ as the output, i.e., $t^* = \sfunction{PathTsp}(\tau)$.}
\label{alg1}
\begin{algorithmic}[1]
\Function{BusTripGeneration}{$G = (V,E),S,t^{max},C, \sfunction{PathTsp}(\cdot)$}
    \State $T_b,T_b^1 \gets \emptyset$ 
    \For {$s \in S$}    \Comment{Generate the trip list with one student}
        \State $\tau \gets \{s\}$
        \State $T_b^1 \gets T_b^1 \cup \{\tau\}$
    \EndFor
    \State $T_b \gets T_b \cup T_b^1$
    \State $k \gets 2$  \Comment{Iterate from trips with two students}
    \While {true}
        \State $T_b^k \gets \emptyset$  \Comment{Initialize the trip list with $k$ students}
        \For {$\tau \in T_b^{k-1}$} 
            \For {$s \in S$ and $s \notin \tau$}
                \State $\tau' \gets \tau \cup \{s\}$ \Comment{Add one more student to the trip with $k-1$ students}
                \If {$\sfunction{CliqueCheck}(\tau,s,G(V,E)) = true$}
                    \If {$\sfunction{FeasibilityCheck}(\tau',t^{max},C,\sfunction{PathTsp}(\cdot)) = true$}
                        \State $T_b^k \gets T_b^k \cup \{\tau'\}$ \Comment{Add feasible trips with $k$ students into the list}
                    \EndIf
                \EndIf
            \EndFor
        \EndFor
        \If{$|T_b^k| = 0$} \Comment{Break when there are no feasible trips with $k$ students}
            \State \textbf{break}
        \EndIf
        \State $T_b \gets T_b \cup T_b^k$; $k \gets k + 1$
    \EndWhile
    \State \textbf{return} $T_b$
\EndFunction
\Function{CliqueCheck}{$\tau,s,G(V,E)$}
    \For{$s' \in S(\tau)$}
        \If {$e(s,s') \notin E$} 
            \State \textbf{return} false
        \EndIf
    \EndFor
    \State \textbf{return} true
\EndFunction
\Function{FeasibilityCheck}{$\tau,t^{max},C,\sfunction{PathTsp}(\cdot)$}
    \State $t^* \gets \sfunction{PathTsp}(\tau)$
    \If {$t^* \leq t^{max}$ and $|S(\tau)| + 1 \leq C$}
        \State \textbf{return} true
    \Else
        \State \textbf{return} false
    \EndIf
\EndFunction
\end{algorithmic}
\end{algorithm}

The last step of our approach is to compute the optimal student-trip assignment given the ST-graph $G_{ST}$, which is formalized as an ILP. The total travel cost $C_{\tau}$ for each trip $\tau \in T$ is calculated from the ST-graph $G_{ST}$, as this is given by the solution of the Path-TSP problem for any feasible trip configurations. 
This gives us all the information needed to formulate an assignment problem based on $G_{ST}$, which assigns all students to trips (if a feasible solution exists) while minimizing the overall school bus scheduling cost. Recall that the total cost consists of the number of buses, vehicle miles traveled and the cost for transporting student via alternate modes.
This student-trip assignment problem can be treated as a special case of the Weighted Set Covering Problem (WSCP).
In the following section, we will establish the connection between the SSRP and the WSCP, and give a simplified ILP formulation for solving the SSRP.

\subsection{Connection between SSRP and WSCP}

%While decomposing SBRP using the concept of a shareability network allows for simplifying the ILP, 
%In this section, we show that the connection between the WSCP and SBRP allows us to simplify the ILP.
%Firstly, we give a formal definition for the WSCP.

\begin{definition}[WSCP]
\label{def:WSCP}
Given a set of $n$ elements $\mathcal{U} =\{e_1,e_2,...,e_n \}$ and $m$ subsets of $\mathcal{U}$, $\mathcal{S}=\{S_1,S_2,...,S_m\}$ with a cost function $c : \mathcal{S} \longrightarrow \mathbb{R} ^{+}$, $c(S_j)$ that denotes the cost of subset $S_j$, the objective is to find a set $\mathcal{A} \subseteq \mathcal{S}$ such that:

\begin{enumerate}
    \item All elements in $\mathcal{U}$ are covered by the set $\mathcal{A}$, and
    \item The sum of costs of subsets in $\mathcal{A}$ is minimized.
\end{enumerate}
\end{definition}

Let $x_S$ be the binary variable for selecting subset $S \in \mathcal{S}$ in the solution $\mathcal{A}$. The WSCP can be formulated as the following ILP:

\begin{align}
{\text{min} } \quad
& \sum_{S\in\mathcal{S}} c(S) \cdot x_S \\
\text{s.t.} \quad
& \sum_{S:e \in S} x_S \geq 1 \quad \forall e \in \mathcal{U}\\
& x_S \in \{0,1\} \quad \forall S \in \mathcal{S}
\end{align}

In the special case of the WSCP where the set $\mathcal{A}$ is a collection of disjoint subsets in $\mathcal{S}$, i.e., 
$$\forall S_i, S_j \in \mathcal{A}, S_i \cap S_j = \emptyset, $$ the problem becomes the Weighted Set Partitioning Problem (WSPP). For the ILP above, constraints (15) become 
\begin{equation}
    \sum_{S:e \in S} x_S = 1 \quad \forall e \in \mathcal{U},
\end{equation}
which imply that each element in $\mathcal{U}$ will be covered by $\mathcal{A}$ exactly once.

In order to build the connection between the SSRP and the WSCP, we first show the correspondence between the SSRP and WSPP.
In the ST-graph of the SSRP, each student $s \in S$ can be treated as an element and the set of elements is $\mathcal{U} = S$.
Each trip configuration $\tau$ serves as a subset of $S$ with trip cost $C_{\tau}$.
The feasible trip configurations $T$ is the collection of subsets $\mathcal{S}$, and the SSRP is equivalent to the WSPP as we are finding a collection of subsets of $\mathcal{S}$ with the minimum cost.
Additionally, the SSRP is a special case of WSPP with two extra conditions on the feasible bus trip configuration list $T_b$.

\begin{claim}
The SSRP is a special case of the WSPP with the following conditions for the bus trip configuration list $T_b$:
\begin{itemize}
    \item \textbf{Downward closed: } $\forall \tau_b \in T_b, \{ \tau_b': \tau_b' \subseteq \tau_b \} \subseteq T_b$.
    \item \textbf{Monotonic cost function: } $\forall \tau_b' \subseteq \tau_b, c(\tau_b) \geq c(\tau_b')$.
\end{itemize}
\end{claim}

\begin{proof}
When generating the feasible bus trips $\tau_b \in T_b$, let $\tau_b \in T_b$ be a feasible bus trip and $T_b^{sub} =\{\tau_b' : S(\tau_b') \subseteq S(\tau_b) \}$ be the collection of sub-trips for $\tau_b$, which $S(\tau_b)$ represents the set of all students in the bus trip $\tau_b$. 
Because students in the trip $\tau_b$ should form a clique in the shareability network, students in the trip $\tau_b'$ also form a clique. 
Meanwhile, the trip cost for $\tau_b'$ is smaller than the cost for $\tau_b$.
Thus, $\forall \tau_b' \in T_b^{sub}$, $\tau_b' \in T_b$, and we have downward closed and monotonic cost function conditions for bus trip configuration list $T_b$.
\end{proof}

The following claim shows the relationship between the WSPP and the WSCP when the collection of subsets $\mathcal{S}$ is downward closed and has a monotonic cost function.

\begin{claim}
\label{clain:WSPP_WSCP}
The WSPP with the following conditions on $\mathcal{S}$ can be solved by the ILP (14) - (16) for the WSCP.
\begin{itemize}
    \item \textbf{Downward closed: } $\forall S \in \mathcal{S}, \{ S' : S' \subseteq S \} \subseteq \mathcal{S}$.
    \item \textbf{Monotonic cost function: } $\forall S' \subseteq S, c(S) \geq c(S')$.
\end{itemize}
\end{claim}

\begin{proof}
We prove this claim by the contradiction. 
Let $\mathcal{A}$ be the optimal solution for the WSCP, and suppose there exits two subsets $S_1, S_2 \in \mathcal{A}$, $S_1 \cap S_2 \neq \emptyset$.

Let $S' = S_1 \cap S_2$, $S'_1 = S_1 \setminus S'$ and $S'_2 = S_2 \setminus S'$. 
According to the downward closed condition, $S'_1, S'_2 \in \mathcal{S}$ since $S'_1 \subseteq S_1$ and $S'_2 \subseteq S_2$.
With the monotonic cost function, we have $c(S'_1) \leq c(S_1)$ and $c(S'_2) \leq c(S_2)$. 
We can reduce the total cost for $\mathcal{A}$ by replace either $S_1$ with $S'_1$ or $S_2$ with $S'_2$ in the optimal set $\mathcal{A}$ while still covering all elements.
Thus, the optimal set $\mathcal{A}$ should be a collection of disjoint subsets in $\mathcal{S}$, and the optimal set $\mathcal{A}$ is also optimal for the set partitioning problem with same $\mathcal{U}$ and $\mathcal{S}$.
\end{proof}

Then we give the simplified ILP for solving the SSRP which is generalized from ILP (14) - (16) for solving the WSCP.

\begin{corollary}
The SSRP can be solved by the following ILP:

\begin{align}
{\text{min} } \quad
& \sum_{\tau_b\in T_b} (\alpha_{C_1}^b + \alpha_{C_2}^b \cdot C_{\tau_b}) \cdot y_{\tau_b} + \sum_{a\in A}\sum_{\tau_a \in T_a} \alpha_C^a \cdot C_{\tau_a} \cdot y_{\tau_a} \\
\text{s.t.} \quad
& \sum_{\tau:s \in \tau} y_{\tau} \geq 1 \quad \forall s \in S\\
& y_{\tau} \in \{0,1\} \quad \forall \tau \in T
\end{align}

\end{corollary}

\begin{proof}
Combining Claim 1 and Claim 2, the SSRP can be transformed into a WSCP and solved by the corresponding ILP.
\end{proof}

By building the connection between the WSCP and the SSRP, we transform the ILP (1) - (13) to a much simplified ILP (18) - (20). The ILP (18) - (20) corresponds to a hyper-graph matching problem in the student-trip graph (as in \citet{Alonso-Mora462}), which is constructed based on the shareability network~\cite{Santi13290}.

However, this simplified ILP is still intractable when considering large-scale SSRP instances, since the size of the bus trip set $T_b$ increases exponentially in $|S|$. While in theory this happens in the ride-pooling setting as well, in practice the shareablity network for pooling is sparse due to the quality of service constraints of the system (e.g. low passenger waiting times and detour limitations)~\cite{Alonso-Mora462}. However, in the SSRP context, the shareability network is considerably denser because of looser quality of service constraints. For example, there is no waiting time constraints (students do not specify a pickup time) and the maximum travel detour can be large due to the only limitation being the total trip length $t^{max}$ (usually 1 hour).
The denser shareability network induces many large cliques and thus potential trip configurations to evaluate, which can become computationally challenging.
In order to address this issue, we propose some network compression techniques that induce sparsity in the SSRP shareability network, and thereby improve the computational tractability of the problem.

\subsection{Network compression techniques}

As mentioned above, our shareability network based approach is still intractable for large-scale instances as the size of the feasible bus trip configuration list $T_b$ can be very large (in the order of billions for real-world instances). 
Generating the feasible trip configuration $T$ is a time-consuming process (requires solving a Path-TSP for each candidate configuration). Furthermore, even if the trip configurations were known, solving the student-trip assignment ILP (18) - (20) with a large number of variables ($O(|T|)$) becomes a challenging task for off-the-shelf ILP solvers.

To address this computation bottleneck of the proposed decomposition approach, we develop network compression techniques that induce sparsity in the shareability network. The techniques we presented reduce the time it takes to compute the feasible trip configurations $T$, while retaining all (or most of the) useful information that is embedded in the network (i.e. retaining good trip configurations).
We present compression techniques from two perspectives that work by either compressing the nodes or pruning the edges of the shareability network. A sparse shareability network leads to a shorter feasible trip configuration list and makes even large-scale SSRPs tractable to solve.

\subsubsection{Node compression technique}
\label{subsubsection:NC}

For the node compression technique, we reduce the number of nodes in the shareability network by generating bus stops and allowing students to walk to bus stops within a maximum walking distance. 
The school buses will pick up students at bus stops instead of students' residence, and the shareability network will be constructed based on bus stops. Due to various reasons, some students might require door to door travel without walking to a bus stop. In these settings, we can create a bus stop at students' residence.

Given a pre-defined set of candidate bus stops $M$, we first formulate the following ILP to assign students to a minimum number of bus stops.
The binary decision variable $x_m$ (for $m \in M$) denotes whether bus stop $m$ is to be selected. 

% Let $N^{max}$ denote the maximum number of students each bus stop can hold.
% The binary decision variable $\mu_{ms}$ denotes whether student $s \in S$ is assigned to stop $m \in M$, and the binary decision variable $\nu_{m}$ indicates whether bus stop $m \in M$ is to be selected.

% \begin{align}
% {\text{min} } \quad
% & \sum_{m \in M} \nu_m \\
% \text{s.t.} \quad
% & \sum_{m \in N_s} \mu_{ms} = 1 && \forall s \in S\\
% & \sum_{s \in S} \mu_{ms} \leq N^{max} && \forall m \in M\\
% & \nu_m \geq \mu_{ms} && \forall s \in S,\;  \forall m \in M\\
% & \mu_{ms} = 0 && \forall s \in S,\;  \forall m \notin N_s\\
% & \nu_{m} \in \{ 0,1 \} && \forall m \in M \\
% & \mu_{ms} \in \{ 0,1 \} && \forall s \in S, \; \forall m \notin M
% \end{align}

% The objective function (21) minimizes the total number of bus stops that are needed.  
% Constraints (22) ensure that each student is assigned to exactly one bus stop within their maximum walking range. 
% Constraints (23) guarantee that at most $N^{max}$ number of students are assigned to each bus stop.
% Constraints (24) impose that a bus stop is selected if at least one student is assigned to it.
% Constraints (25) make sure that each student will not be assigned to bus stops outside the radius of their maximum walking distance.
% Constraints (26) and (27) impose that decision variables are binary.
% This ILP generates the minimum number of bus stops and assigns students to bus stops within their maximum walking distance. 

% The binary decision variable $x_m$ (for $m \in M$) denotes whether bus stop is to be selected. 

\begin{align}
{\text{min} } \quad
& \sum_{m \in M} x_m \\
\text{s.t.} \quad
& \sum_{m \in N_s} x_m \geq 1 \quad \forall s \in S\\
& x_{m} \in \{ 0,1 \} \quad \forall m \in M
\end{align}

The objective function (21) minimizes the total number of bus stops that are needed.  
Constraints (22) ensure that each student has at least one bus stop within their maximum walking distances. 
Constraints (23) impose that decision variables are binary.
This ILP generates the minimum number of bus stops needed to cover all students. 

Even if some students need door-to-door pickups ($d_s^{max} = 0$ and $N_s = \emptyset$), the above node compression technique can still be used by assigning the student to a bus stop and adding a penalty corresponding to the vehicle having to move from the stop to the student residence and back. 
This penalty is bounded by the distance of a round-trip between the assigned bus stop and door-to-door students' residence. The penalty will be incorporated in the function $\sfunction{PathTsp}(\cdot)$ when considering the feasibility of trips.
For cases with considerable number of students who need door-to-door pickups, this heuristic substantially reduces the computational complexity.
The impact of virtual walking distance for students with door-to-door pickups will be discussed in the experiments section.
Figure \ref{fig:node_compression} shows an example of the shareability network before and after applying the node compression technique with above heuristic for students with door-to-door pickups (this example uses a virtual walking distance of 0.5 miles, which is an independent parameter from the maximum walking distance).

\begin{figure*}[!h]
    \centering
    \subfloat[The original shareability network]{{\includegraphics[scale=0.235]{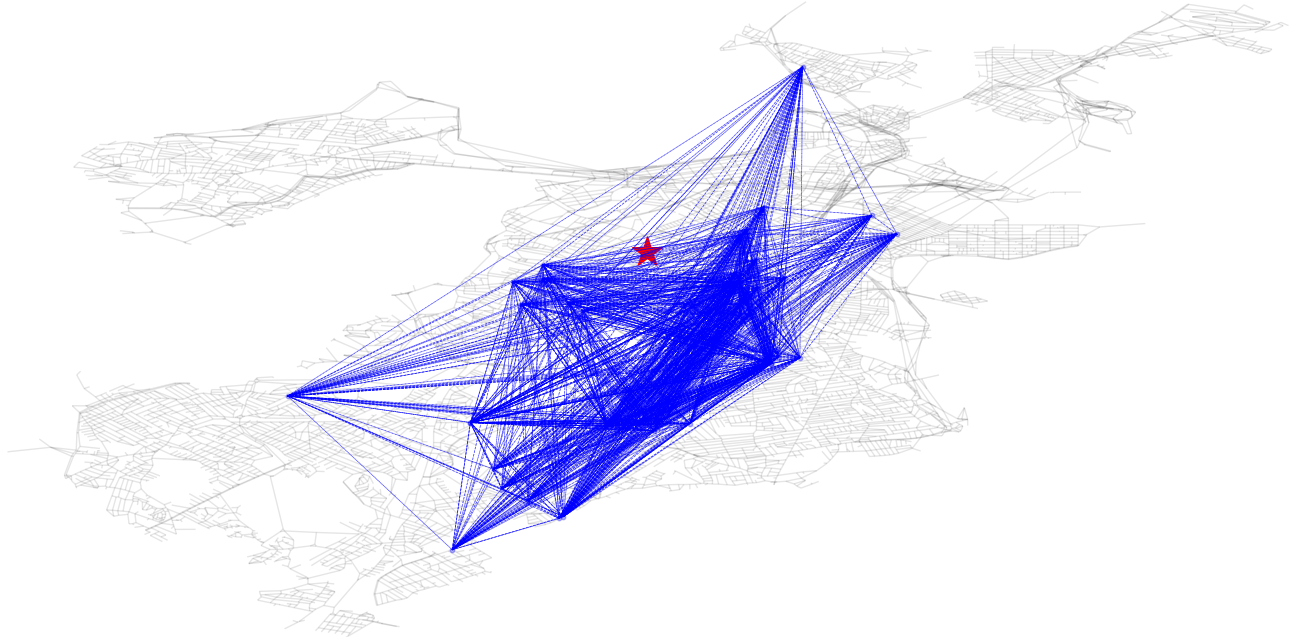}}}
    \qquad
    \subfloat[The shareability network after applying node compression technique]{{\includegraphics[scale=0.235]{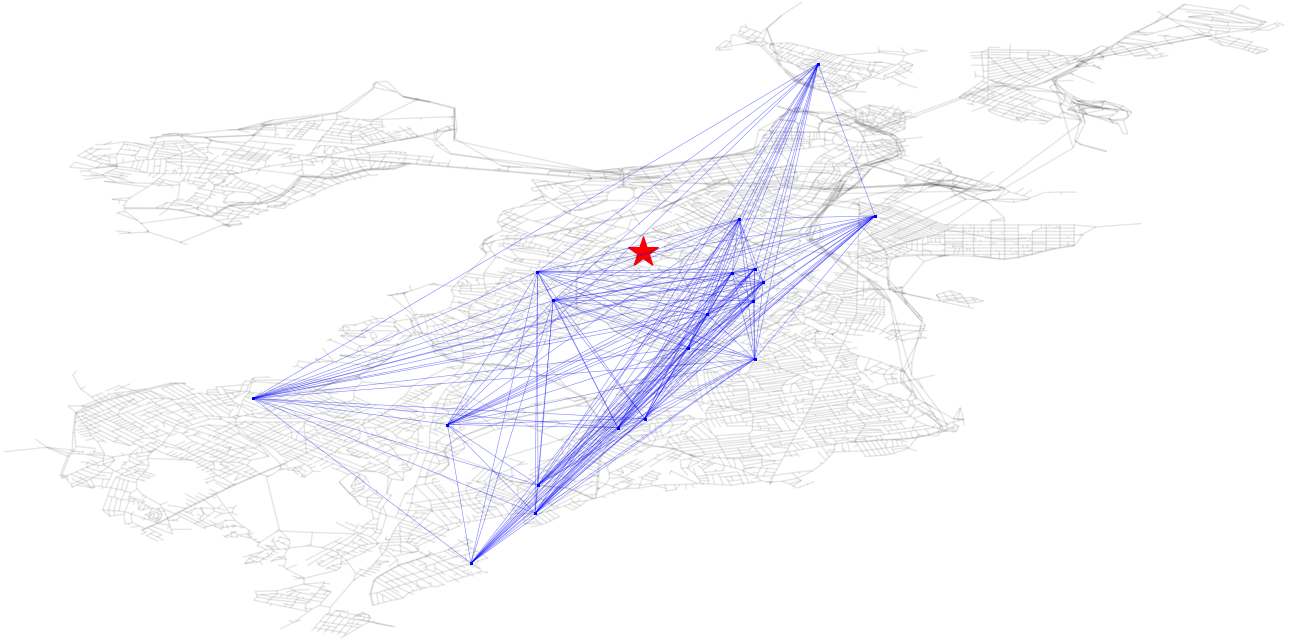} }}
    \caption{An instance of applying the node compression technique for the shareability network. These figures are generated using data corresponding to the Tommy Harper School from the synthetic BPS dataset~\cite{BPS_data}. The school has 51 students including 7 students with door-to-door pickups (considering a virtual walking distance 0.5 miles). The red star is the location for school and the blue graph is the shareability network. The number of nodes (bus stops) in the shareability network decreases from 51 to 19 after applying the node compression technique.}%
    \label{fig:node_compression}
\end{figure*}

It is worth mentioning that our node compression formulation might lead to overcrowded bus stops (with many students) in dense urban areas. 
The school bus schedules can be inefficient if a school bus that visits an overcrowded bus stop is not able to pick up students from any other stops. 
Implications of overcrowded bus stops will be discussed in detail in Section \ref{subsubsection:MH} and a post-processing approach will be proposed to address this issue.

The node compression technique reduces the maximum number of effective students for any trip $\tau_b \in T_b$, since each bus stop now aggregated multiple students as a single request with a larger capacity. 
Nonetheless, even with the corresponding reduction in the number of nodes in the shareability network, the shareability network might still not be sparse enough for the computational tractability of large-scale problem instances. 
Therefore, we also adopt a heuristic-based edge pruning technique that deletes edges which are unlikely to induce shared trips. 
This pruning technique can in theory leads to a sub-optimal solution because we will eliminate feasible sharing possibilities. Our aim is to generate a heuristic set of rules that only eliminate pairings that are very unlikely to occur in practice. 

\subsubsection{Edge pruning technique}

For the edge pruning technique, we reduce edges in the shareability network following some mechanisms.
The main idea behind pruning edges in the shareability network is to go beyond the technical feasibility of a pairing and introduce additional constraints that in some way embed the realism of a pair of nodes being shared. The travel time between two nodes is not the only factor corresponding to the likelihood of sharing trips, the relative position of the school and nodes also matters. For example, it is unlikely for a student 5 miles north of the school to share a ride with a student 5 miles south of the school, if there are enough students to fill a bus from north of the school. On the other hand, two students could share the same trip if they are on the same side of the school even if they are relatively far away from each other. 

To incorporate both factors, we define the adjusted travel time $\Bar{t}_{ij} = \delta_{ij} \cdot t_{ij}$ between any two nodes $i$ and $j$ in the shareability network, where $\delta_{ij} = \frac{t_{ij} + t_{jv_d}}{t_{iv_d}}$, represents the detour factor ($\delta_{ij} \geq 1$). 
The detour factor is a measure of how much of a detour node $i$ has to incur to share a trip with node $j$, as opposed to node $i$ being connected to the destination directly.

With the definition of the adjusted travel time, we apply a mechanism that a node only share trips with \textit{adjusted nearby nodes} in the shareability network. 
More formally, the \textit{adjusted nearby nodes} for any node $i$ are generated by calculating the adjusted travel time between node $i$ and all other nodes $V \setminus \{i\}$, sorting the nodes in ascending order by the adjusted travel time, and choosing the $k$ closest nodes such that $k \leq \beta \cdot C$, where $\beta$ is a pruning parameter. 
If nodes correspond to bus stops with multiple students as a result of the node compression, we consider the $k$ closest nodes such that the sum of students is less than $\beta \cdot C$. Note that $\sum_{j=1}^k n(m_j) \leq \beta \cdot C$, where $m_j$ represents the bus stop and $n(m_j)$ is the number of students at stop $m_j$.

% $\beta$ is a control parameter and it should be greater than 1 (at least considering $C$ students nearby). 
The number of edges decreases as $\beta$ decreases, so the number of feasible bus trips $|T_b|$ also decreases.
We tune $\beta$ to generate an appropriate sized set of feasible bus trips $T_b$ that maintain tractability (size of $|T_b|$) while being large enough to provide diverse set of bus routes.
Intuitively, if $\beta$ is too small, many feasible trips that belong to the optimal solution might be eliminated.
On the other hand, if $\beta$ is too large and only a few edges are eliminated in the shareability network, the trip list will be too large and large-scale SSRPs will remain intractable.
Therefore, $\beta$ provides a balance between tractability and optimality. 
The appropriate value of $\beta$ (minimum necessary for good solutions) will in practice be different for different problem instances and spatial distributions of the students, and should be chosen according to the computational budget.
Figure \ref{fig:edge_compression} shows an example of the shareability network before and after applying the edge pruning technique. 

\begin{figure*}[!h]
    \centering
    \subfloat[The shareability network after applying node compression]{{\includegraphics[scale=0.235]{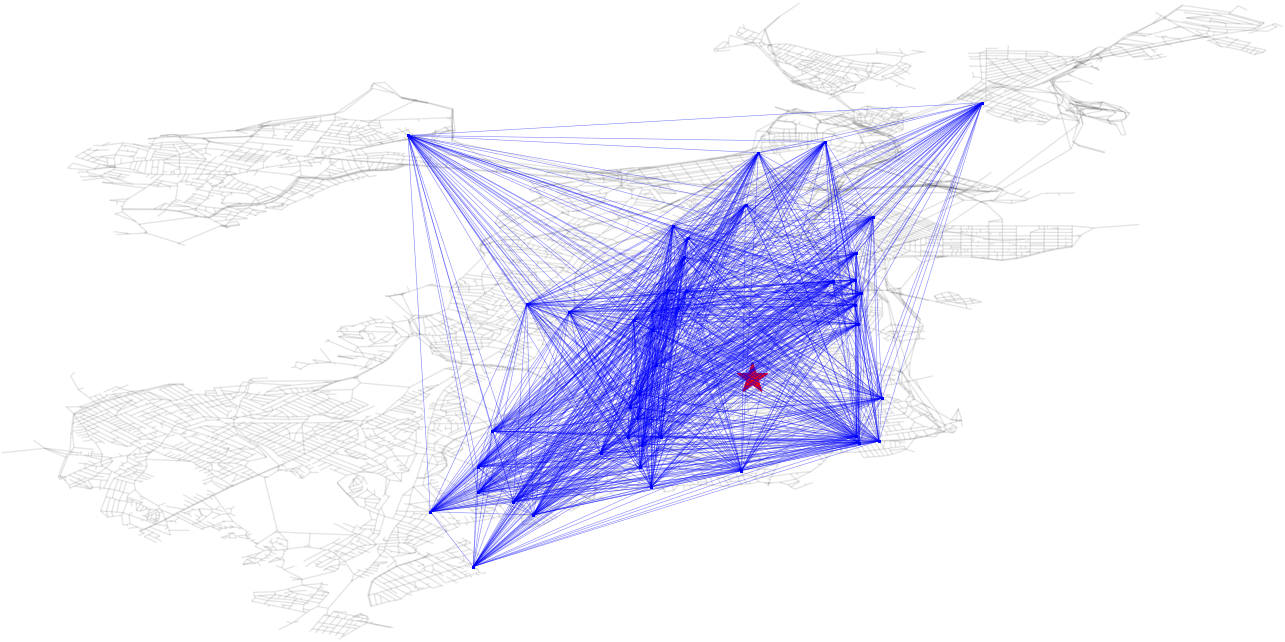} }}
    \qquad
    \subfloat[The shareability network after applying both node compression and edge pruning ($\beta = 2$)]{{\includegraphics[scale=0.235]{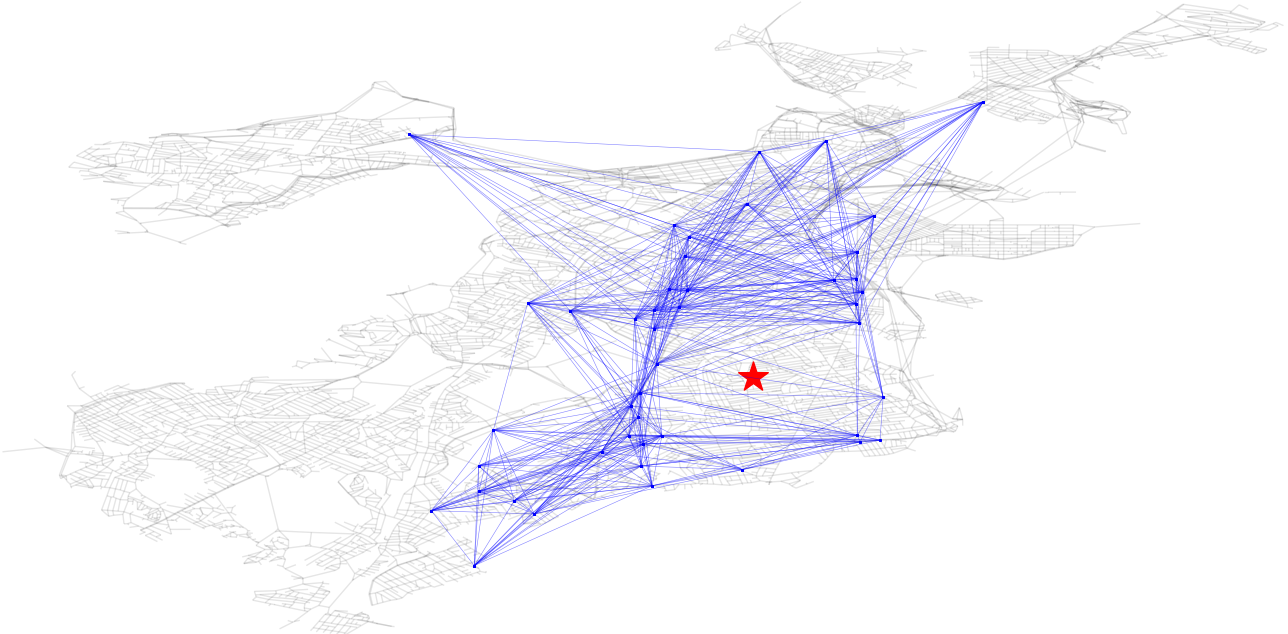} }}%
    \caption{Instance of applying edge pruning for the shareability network. These figures are generated using data from the Dennis Eckersley School in the synthetic BPS data set. The instance has 403 students (75 students with door-to-door pickups - we consider a virtual walking distance 0.5 miles). The red star denotes the school location and the blue graph shows the shareability network. The node compression generates 45 bus stops. After applying edge pruning with $\beta = 2$, the shareability network becomes very sparse with a much smaller number of cliques. }
    \label{fig:edge_compression}
\end{figure*}

Once again, we note that the edge pruning technique is a heuristic approach and does not provide any guarantees regarding the loss of optimality. 
In particular, when the size of SSRP instances becomes larger, in order to make problems tractable we have to use a smaller value of $\beta$ and restrict the size of feasible bus trip configuration list $|T_b|$.
By the criterion of finding a feasible trip in Lemma \ref{lemma}, all students in a feasible trip form a clique in the shareability network. 
Experimental results show that the edge compression can lead to unsatisfactory results when $\beta$ is too small. 
In order to compensate for this side-effect, we propose a relaxation of the necessary condition from Lemma \ref{lemma}. We introduce the $\gamma$-quasi-clique process based on the Algorithm \ref{alg1} to find groups of students who form \textit{quasi-cliques} in the shareability network.
More precisely, the $\gamma$-quasi-clique is a sub-graph similar to a clique defined by a heuristic parameter $\gamma$, which indicates the connection between quasi-cliques and cliques.
A $\gamma$-quasi-clique of size $k+1$ is a sub-graph built from adding a new node $i$ to a $\gamma$-quasi-clique of size $k$ where node $i$ connects with at least a $\gamma$ proportion of nodes in the size-$k$ $\gamma$-quasi-clique.
Note that $\gamma$ is upper bounded by 1 by definition. 

The $\gamma$-quasi-clique process replaces the function $\sfunction{CliqueCheck}(\tau,G(V,E))$ in the Algorithm \ref{alg1} and details are shown in the Algorithm \ref{alg2}.
Figure \ref{fig:quasi_clique_explain} shows a simple instance for explaining edge pruning techniques. Figure (\ref{fig:quasi_clique_explain}A) and (\ref{fig:quasi_clique_explain}B) display optimal routes before and after applying $\beta$ edge pruning method with $\beta = 0.5$. The edge compression technique leads to sub-optimal solution which requires an additional bus. Figure (\ref{fig:quasi_clique_explain}C) indicates the optimal routes after applying both $\beta$ edge pruning method and $\gamma$-quasi-clique process with $\beta = 0.5$ and $\gamma = 0.4$, i.e., the threshold for passing the $\gamma$-quasi-clique process is 40\% in this instance. Starting from a feasible trip $\tau' = (S_A, S_B)$, trip $\tau'' = (S_A, S_B, S_C)$ passes the quasi-clique check process since student $S_C$ connects with 50\% of students in trip $\tau'$. Similarly, trip $\tau = (S_A, S_B, S_C, S_D)$ passes the quasi-clique check process given student $S_D$ connects with 67\% of students in trip $\tau''$. 

\begin{algorithm}[h]
\caption{$\gamma$-quasi-clique process. Input: the shareability network $G$, a feasible bus trip $\tau \in T_b$, a student $s \notin \tau$, heuristic parameter $\gamma$.}
\label{alg2}
\begin{algorithmic}[1]
\Function{$\gamma$-QuasiCliqueProcess}{$\tau,s,G(V,E),\gamma$}
    \State $x \gets 0$ 
    \For{$s' \in S(\tau)$}
        \If {$e(s,s') \notin E$}
            \State $x \gets x + 1$
        \EndIf
    \EndFor
    \If {$x > \gamma \cdot |\tau|$}  
        \State \textbf{return} false
    \EndIf
    \State \textbf{return} true
\EndFunction
\end{algorithmic}
\end{algorithm}

\begin{figure}[!h]
    \centering
    \includegraphics[scale=0.55]{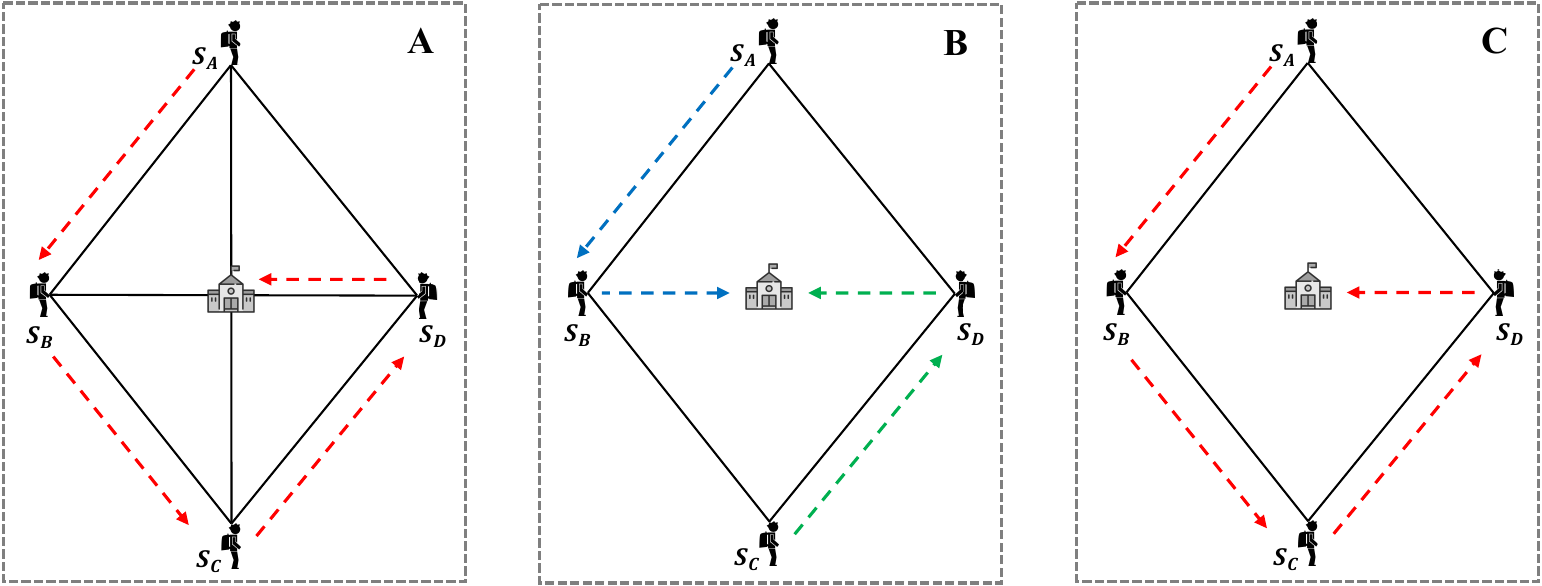}
    \caption{Instance for applying edge pruning techniques. Black solid lines represent the shareability network and colored dashed lines represent optimal routes. There are four students need to be picked up with school buses of capacity $C = 4$. Each sub-figure shows both the shareability network and the optimal bus routes. (A): scenario without applying any edge pruning techniques; there is a fully-connected shareability network and the optimal routes is to have one bus pick up all four students. (B): applying edge pruning technique with $\beta = 0.5$; each student is connected to two nearest students in the shareability network and the optimal routes is having two school buses serving four students. (C): applying edge pruning techniques with $\beta = 0.5$ and $\gamma = 0.4$; it has the identical shareability network as the scenario B but produces the same optimal routes as the scenario A; the quasi-clique process considers the trip $\tau = (S_A, S_B, S_C, S_D)$ feasible although students in trip $\tau$ do not form a clique in the shareability network.}
    \label{fig:quasi_clique_explain}
\end{figure}

Introducing the $\gamma$-quasi-clique process in the Algorithm 1 improves the solution by augmenting the bus trip configuration list $T_b$ given a fixed $\beta$ and produces better results in practice.
The $\beta$ edge pruning approach dramatically shrinks the trip configuration list, which is important for tractability but has the potential to remove good trips. The $\gamma$-quasi-clique process can be thought of as a sampling technique for adding some of these deleted trips that are deemed to be good---by considering trips that are almost cliques (or quasi-cliques) in the compressed shareability network.

\section{Experiments}
\label{sec:experiments}

To test the applicability of our proposed algorithms to large-scale SSRPs, we conducted a number of numerical experiments using some publicly accessible benchmark problem instances. All the experiments were run on a 3.0 GHz AMD Threadripper 2970WX with 128 GB Memory using Python 3.7. 

\subsection{BPS synthetic school bus data}
\label{subsec:BPS}

The first data set we used is from the Transportation Challenge held by the BPS~\citep{BPS_data}, which contains 22420 simulated students addresses (to protect student privacy) from 134 public schools. The data set was created with the same aggregate pickup location distributions as in the real-world. 
%It can be downloaded from \url{https://www.bostonpublicschools.org/transportationchallenge}.
The spatial distribution of students and schools from this dataset is shown in the Figure \ref{fig:data}.
    
\begin{figure*}[h]
\centering
\includegraphics[scale=0.4]{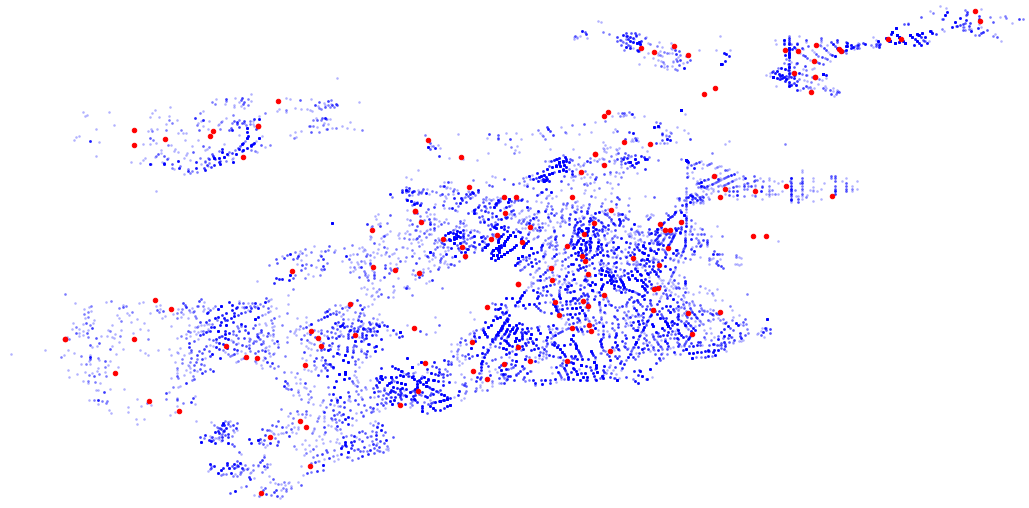}
\caption{Simulated data from BPS. Blue dots represent student locations and red circles denote school positions.}
\label{fig:data}
\end{figure*}

The Boston road network $G_r=(V_r ,E_r)$ is obtained from Open Street Maps (OSM)~\cite{OpenStreetMap} using the open source Python library OSMnx \cite{BOEING2017126}. 
All the visualizations of the results are generated using the Python library NetworkX \cite{SciPyProceedings_11} and Matplotlib. 
The implementation used for the $\sfunction{PathTsp}$ function is shown in \ref{apend:tsp}.
We used off-the-shelf ILP solver Gurobi 9.0.3~\cite{gurobi} for solving the assignment problem in the experiments (with a 21600-second maximum computation time).
We adopt the following assumptions and parameter choices, primarily based on the rules specified for the BPS challenge, which were determined based on school district requirements and practical considerations. This allows us to accurately compare our results against other solutions techniques for these benchmark instances. The parameters corresponding to alternate modes are our choices.
\begin{itemize}
\item School buses can start at any location in the network.
\item The delay time $t^{delay}_m$ (in seconds) for buses at each bus stop $m$ follows a function $t^{delay}_m = 15 + 5 \cdot n(m)$ , where $n(m)$ is the number of students at the bus stop $m$.
\item Each door-to-door student has the same virtual maximum walking distance $d_s^{max} = 0.5$ miles.
\item There is no restriction on the bus fleet size and the school bus capacity $C = 72$.
\item The maximum ride time for students is $t^{max} = 3600$ seconds.
\item The set of potential bus stops $M$ is same to the set of nodes in the road network $V_r$. 
\item The set of alternate modes $A$ only contains dedicated single occupancy vehicles (e.g. single student by taxi/ride-hail), i.e., $A = \{direct\}$.
\item The fixed cost of each bus (including the driver wage) is $\alpha_{C_1}^b = 200$ dollars per day. The cost for operating a bus is $\alpha_{C_2}^b = 1$ dollar per mile~\cite{Bus}. This implies that the capital cost for ownership is much larger than the operational cost, which leads to solutions that minimize the number of buses needed. The cost for taking a dedicated vehicle is $\alpha_C^{direct} = 2$ dollars per mile~\cite{Uber}.
\end{itemize}

%~\footnote{The cost for owning a bus per day is estimated by the Bus Operating Cost Calculator. \url{http://www.freightmetrics.com.au/Calculators\%7CRoad/BusOperatingCost/tabid/671/Default.aspx}}

\begin{table}[h!]
\centering
\caption{Computational results for the ten representative schools from the BPS data set without allowing alternate modes.  The total students and door-to-door students are denoted by $N_S$ and $N_{S_{d2d}}$,  the number of bus stops is given by $N_{M^*}$, and $N_{T_b}$ designates the length of the feasible trip configuration list. The objective value for the shareability network based decomposition approach without alternate modes is given by SND and $N_B$ stands for the number of buses needed. 
The overall computation time (in seconds) is indicated by $T$.
$GAP$ denotes the optimality gap of the solutions, and $OPT$ in this column represents that the solution is optimal given the current heuristic settings $\beta$ and $\gamma$.
The dash ($-$) in the $\beta$ and $\gamma$ columns indicates that the edge pruning heuristic was not used.}
\label{tab:computation_base}
\begin{tabular}{m{8em} m{2em} m{2em} m{2em} m{4em} m{1em} m{1em} m{3em} m{1em} m{2em} m{2em}}
 \hline
 School & $N_S$ & $N_{S_{d2d}}$ & $N_{M^*}$ & $N_{T_b}$ & $\beta$ & $\gamma$ & SND & $N_B$ & GAP & $T$ \\
 \hline
 Tommy Harper & 51 & 7  & 19 & 62632 & -  & - & 649.16 & 3  & OPT & 58\\
 Craig Kimbrel  & 71 & 11  & 20 & 107253 & - & - & 648.97 & 3 & OPT & 567\\
 Deven Marrero & 91 & 14 & 22 & 100021 & - & - & 653.94 & 3  & OPT & 40\\
 Frank Malzone & 160 & 30 & 30 & 1281510 & 1.5 & 0.4 & 879.40 & 4 & 5.38\% & 21743\\
 Dick  Williams & 183 & 28 & 22 & 71672 & - & - & 659.98  & 3 & OPT & 10\\
 Dick Bresciani  & 208 & 40 & 35 & 1396354 & 1.5& 0.4& 882.99 & 4 & OPT & 270\\
 Dutch Leonard  & 243 & 42& 42& 2428467 & 2& 0.4 & 1113.83 & 5 & OPT & 535\\
 Christian Vazquez & 344 & 66& 35 & 1212711 & 3.5 & 0.4 & 1109.87 & 5 & OPT & 230\\
 Dennis  Eckersley & 403 & 75 &45 & 1138279 &2.5 & 0.4 & 1328.66 & 6 & OPT & 2429\\
 Rick  Ferrell  & 573 & 109&55 & 546319 &2.5 & 0.4 & 1971.91 & 9 & 8.59\% & 21657\\
 \hline
\end{tabular}
\end{table}

\begin{table}[h!]
\centering
\caption{Computational results for the ten representative schools from the BPS data set with alternate modes. 
The objective value for the shareability network based decomposition approach with alternate modes is given by SND-A.
The number of students who are assigned to dedicated vehicles is given by $N_U$.
$Improv$ implies the improvement of objective value for SND-A compared to SND. 
$\Delta$ indicates the difference of the number of buses needed between SND-A and SND.
}
\label{tab:computation_alternate}
\begin{tabular}{m{8em} m{1em} m{1em} m{4em} m{4em} m{1em} m{1em} m{2em} m{2em} m{1em} m{3em}}
 \hline
 School  & $\beta$ & $\gamma$ & SND & SND-A &  $N_B$ & $N_U$ & GAP & $T$ & $\Delta$ & Improv \\
 \hline
 Tommy Harper & - & - & 649.16 & 342.00 & 1 & 24 & OPT & 7 & -2 & 47.32\% \\
 Craig Kimbrel & - & - & 648.97 & 461.39 & 2 & 2 & OPT & 14 & -1 & 28.90\% \\
 Deven Marrero & - & - & 653.94 & 527.03 & 2 & 13 & OPT & 11 & -1 & 19.41\% \\
 Frank Malzone & 1.5 & 0.4 & 879.40 & 725.66 & 3 & 5 & OPT & 157 & -1 & 17.48\% \\
 Dick Williams & - & - & 659.98 & 622.07 & 2 & 42 & OPT & 10 & -1 & 5.74\% \\
 Dick Bresciani & 1.5 & 0.4 & 882.99 & 790.47 & 3 & 23 & OPT & 162 & -1 & 10.48\% \\
 Dutch Leonard & 2 & 0.4 & 1113.83 & 937.06 & 3 & 34 & OPT & 283 & -2 & 15.87\% \\
 Christian Vazquez & 3.5 & 0.4 & 1109.87 & 1109.87 & 5 & 0 & OPT & 470 & 0 & 0.00\% \\
 Dennis Eckersley & 2.5 & 0.4 & 1328.66 & 1322.44 & 5 & 45 & 2.97\% & 21706 & -1 & 0.47\% \\
 Rick Ferrell & 2.5 & 0.4 & 1971.91 & 1781.92 & 8 & 2 & OPT & 73 & -1 & 9.63\% \\
 \hline
\end{tabular}
\end{table}

Numerical experiments were conducted for ten representative schools from the BPS dataset that span the range of school sizes and complexity of the student spatial distributions. 
The relevant metrics from results without alternate modes are shown in Table \ref{tab:computation_base} and optimal school bus schedules can be found in \ref{apen:results_base}.
Table \ref{tab:computation_alternate} displays results considering alternate modes for students and optimal school bus schedules can be found in \ref{apen:results_alternate}.

Compared to the scenario without alternate modes, introducing alternate modes to students can significantly reduce the cost of school bus schedules, 15.5\% on average over 10 instances.
Also, the number of buses needed decreases for most school instances by transporting a group of students by dedicated vehicles.
Allowing alternate modes is particularly efficient for Craig Kimbrel, Frank Malzone and Rick Frrell, where a school bus can be replaced by transporting a few students to schools with alternate modes.
It is worth mentioning that we assume each student is willing to take alternate modes in these experiments, and the willingness of students can be incorporated by adding additional constraints into the model. The school bus schedules generated considering alternate modes are at least as good as the schedules produced without alternate modes.

\subsection{Benchmark testing}
To further evaluate the performance and scalability of this approach, 
we compared our approach with two state-of-the-art methods for solving SSRPs~\cite{SCHITTEKAT2013518,MIT_SBRP}.

\subsubsection{BiRD (Bi-objective Routing Decomposition) algorithm}
First we compared against the single-school route generation component of the BiRD (Bi-objective Routing Decomposition) algorithm~\cite{MIT_SBRP} using the BPS synthetic data. 
%In particular, we implemented the single-school routing algorithm from~\cite{MIT_SBRP} and compared results using the BPS synthetic data. 
In the comparison, we set $\alpha_{C_1}^b = 200, \alpha_{C_2}^b = 1$ and $\alpha_C^a = \infty, \forall a \in A$ (to consider the objective of minimizing the number of school buses needed and eliminate alternative models).
In the BiRD algorithm, the quality of the solution is a function of the number of feasible trips $N$ covering each bus stop. In the application described in~\citet{MIT_SBRP}, $N$ was set to be 400. 
Because the BiRD algorithm includes a randomized trip generation step and we are not able to exactly replicate the trips that were generated, we conservatively set $N$ equal to 1000 (instead of the 400 used in their experiments) and reported the best result out of 10 different BiRD runs (to further reduce the probability of getting inferior results due to randomization).
A comparison of the results is shown in the Table \ref{tab:MIT_comparison}. The corresponding optimal school bus schedules can be found in \ref{apen:results_MIT}.

\begin{table}[h!]
\centering
\caption{Comparison results with the single-school route generation component of the $BiRD$ algorithm~\cite{MIT_SBRP}. $BiRD$ stands for the objective value for the Bi-objective Routing Decomposition algorithm, $Improv$ implies the improvement of objective value for $SND$ compared to $BiRD$. $\Delta$ indicates the difference of the number of buses needed between $SND$ and $BiRD$.}
\label{tab:MIT_comparison}
\begin{tabular}{|m{8em} | m{1em} m{1.5em} m{1.5em} m{4em} m{1em} m{1em}| m{3em}| m{3em}| m{1em} | m{3em}|}
 \hline
 School & $N_S$ & $N_{S_{d2d}}$ & $N_{M^*}$ & $N_{T_b}$ & $\beta$ & $\gamma$ & SND & BiRD & $\Delta$ & Improv  \\
 \hline
 Tommy Harper & 51 & 7  & 19 & 62632 & - & - & 649.16& 653.44&0 & 0.66$\%$\\
 \hline
 Craig Kimbrel  & 71 & 11  & 20 & 107253 & - & - & 648.97 & 654.28 &0 & 0.81$\%$\\
 \hline
 Deven Marrero & 91 & 14 & 22 & 100021 & - & - &653.94& 658.34& 0 & 0.67$\%$\\
 \hline
 Frank Malzone & 160 & 30 & 30 & 1281510 & 1.5 &0.4&879.40&881.90& 0 &0.28$\%$\\
 \hline
 Dick  Williams & 183 & 28&22 &71672 &- &- &659.98&660.32&0 &0.05$\%$\\
 \hline
 Dick Bresciani  & 208 & 40 &35 &1396354 &1.5 &0.4 &882.99&883.91&0 & 0.10$\%$\\
 \hline
 Dutch Leonard  & 243 & 42& 42& 2428467&2 &0.4 &1113.83&1320.28& -1 &15.64$\%$\\
 \hline
 Christian Vazquez & 344 &66 &35 & 1212711&3.5 &0.4 &1109.87&1313.03&-1 & 15.47$\%$\\
 \hline
 Dennis  Eckersley & 403 & 75&45 &1138279 &2.5 &0.4 &1328.66&1532.52&-1 & 13.30$\%$\\
 \hline
 Rick  Ferrell  & 573 & 109 &55 &546319 &2.5 &0.4&1971.91& 1973.46&0 & 0.08$\%$\\
 \hline
\end{tabular}
\end{table}

\subsubsection{Metaheuristic algorithm with neighbourhood search}
\label{subsubsection:MH}

An alternative approach in~\citet{SCHITTEKAT2013518} used a metaheuristic algorithm (greedy search procedure with neighborhood search) to simultaneously solve both the bus stop selection and bus route generation problem for a single school. In contrast, we solve these two problems sequentially. To compare our method with this approach, we used the Euclidean space instances given in~\cite{SCHITTEKAT2013518} (instead of network instances with shortest path distances) and adopted our algorithm to the settings used in their experiments. In particular, by first assigning students to bus stops and then solving for optimal bus routes.
The comparison results are shown in Table \ref{tab:meta_comparison1}.
For synthetic instances generated by \citet{SCHITTEKAT2013518}, we limited our comparisons to instances with a maximum walking distance of 5 units (which was the minimum considered), since larger distances led to fewer pickup points and simpler routing problems. Such instances could be solved without novel techniques that we proposed.

\begin{table}[h!]
\centering
\caption{Comparison results with metaheuristic method~\cite{SCHITTEKAT2013518}. $ID$ corresponds to the instance number, $stop$ denotes the number of bus stops, $stud$ represents the number of students, $cap$ indicates the bus capacity, $wd$ is the maximum walking distance for students in Euclidean space, $\beta$ and $\gamma$ are parameters for the edge pruning techniques, $N_{T_b}$ stands for the number of feasible bus trips, $MH$ is the objective value for metaheuristic method~\cite{SCHITTEKAT2013518}, $SND$ indicates the objective value for the shareability network based decomposition approach, $Improv$ implies the improvement of objective value for $SND$ compared to $MH$.}
\label{tab:meta_comparison1}
\begin{tabular}{m{2em}  m{2em} m{2em} m{2em} m{2em} m{2em} m{2em} m{3em} m{4em} m{4em} m{3em}}
 \hline
 ID & stop & stud & cap & wd &  $\beta$ & $\gamma$ & $N_{T_b}$ & MH & SND & Improv  \\
 \hline
 73 & 40 & 200 & 25 & 5 & 2&0.3 &1106 &831.94 &826.45 &$+0.66\%$ \\
 74 & 40 & 200 & 50 & 5 & 1.5& 0.3&127717 & 593.35&588.23 &$+0.86\%$ \\
 81 & 40& 400& 25& 5& -& -& 3414& 1407.05&1386.6 & $+1.45\%$ \\
 82 & 40& 400& 50& 5& 3&0.3 &6713 & 858.80& 853.45 &$+0.62\%$ \\
 89 & 40& 800& 25& 5& - &- &40 & 2900.14& 3085.11 &$-6.38\%$\\
 90 & 40& 800& 50& 5&- &- &10775 & 1345.70& 1402.23 &$-4.20\%$\\
 97 & 80& 400& 25& 5& 3& 0.3&5184 & 1546.23& 1499.6 &$+3.02\%$\\
 98 & 80& 400& 50& 5& 2& 0.3& 182115& 1048.56& 1016.07 & $+3.10\%$\\
 105 & 80& 800& 25& 5& -& -& 13863& 2527.96& 2665.81 &$-5.45\%$\\
 106 & 80& 800& 50& 5& 3& 0.3 & 6473 &1530.58 & 1513& $+1.15\%$\\
 \hline
\end{tabular}
\end{table}

\begin{table}[h!]
\centering
\caption{Comparison results with metaheuristic method~\cite{SCHITTEKAT2013518} with bus stop separations. $N^{max}$ represents the maximum number of students for each bus stop.}
\label{tab:meta_comparison2}
\begin{tabular}{m{2em}  m{2em} m{2em} m{2em} m{1.5em} m{1.5em} m{1.5em} m{1.5em} m{3em} m{4em} m{4em} m{3em}}
 \hline
 ID & stop & stud & cap & wd & $N^{max}$& $\beta$ & $\gamma$ & $N_{T_b}$ & MH & SND & Improv  \\
 \hline
 73 & 40 & 200 & 25 & 5 & 5 & 2&0.3 &52077 &831.94 & 816.78 &$+1.82\%$ \\
 74 & 40 & 200 & 50 & 5 &10& 1.5& 0.3&127717 & 593.35& 588.23 &$+0.86\%$ \\
 81 & 40& 400& 25& 5& 5 &2.5& 0.3& 64003& 1407.05& 1311.46 &$+6.79\%$ \\
 82 & 40& 400& 50& 5& 10&3&0.3 &122808 & 858.80& 830.19 &$+3.33\%$ \\
 89 & 40& 800& 25& 5& 5 &2.5 &0.3 &54997 & 2900.14&  2791.97&$+3.73\%$\\
 90 & 40& 800& 50& 5& 10 &3 &0.3 &214840 & 1345.70& 1331.41&$+1.06\%$\\
 97 & 80& 400& 25& 5& 5&3& 0.3& 582302 & 1546.23& 1465.22 &$+5.24\%$\\
 98 & 80& 400& 50& 5& 10 &2& 0.3& 414522 & 1048.56& 1011.49 & $+3.54\%$\\
 105 & 80& 800& 25& 5& 5 & 2.5& 0.3& 127019 & 2527.96& 2520.14 &$+0.31\%$\\
 106 & 80& 800& 50& 5& 10 & 3& 0.3& 1141177 &1530.58 & 1487.76 & $+2.80\%$\\
 \hline
\end{tabular}
\end{table}

However, our approach yields worse solutions for instance 89, 90 and 105 because of the limitations induced by sub-optimal allocations of students to bus stops. Due to our bus stop selection process (Section \ref{subsubsection:NC}) minimizing the number of bus stops, in these instances, large numbers of students were assigned to a small number of bus stops, thereby reducing the size of trip configuration list and leading to less flexibility in terms of possible routes (compared to other instances). To address this issue, we split each bus stop into several virtual bus stops by imposing a maximum number of students $N^{max}$ that can be assigned to a given bus stop. We picked the value of $N^{max}$ such that a certain minimum number of bus stops is created, where the desired number of stops is also a function of bus capacity. Clearly, having more bus stops allows for more route flexibility, but comes at the cost of additional computational complexity. Therefore, a balance between the competing goals of solution quality and computational cost is needed.  In these experiments, we set $N^{max} = 10$ for instances with vehicle capacity equals to $50$ and $N^{max} = 5$ for instances with vehicle capacity equals to $25$.

Table \ref{tab:meta_comparison2} shows the results after we apply the post-processing approach for separating each bus stop into several stops with a students cap $N^{max}$.
By splitting the overcrowded bus stops into several virtual stops, our shareability network-based decomposition approach gets better solutions than the MH approach in all problem instances. It also improves our solution without bus stop splitting in all but instance 74.

% \begin{figure}[h]
% \centering
% \includegraphics[scale=0.45]{Figures/beta.pdf}
% \caption{Sensitivity analysis for $\beta$. We calculate the objective value, number of buses, running time and number of feasible bus trips corresponding to different $\beta$ values. }
% \label{beta}
% \end{figure}

% \begin{figure}[h]
% \centering
% \includegraphics[scale=0.45]{Figures/gamma.pdf}
% \caption{Sensitivity analysis for $\gamma$. We calculate the objective value, number of buses, running time and number of feasible bus trips corresponding to different $\gamma$ values. }
% \label{gamma}
% \end{figure}

\subsection{Sensitivity analyses}

Finally, we conduct sensitivity analyses for i) parameters of edge pruning techniques, ii) cost parameter for alternate modes, and iii) virtual walking distance for students with door-to-door pickups in the node compression step.
The first two analyses will be conducted using the large-scale BPS instance of the Christian Vazquez School with 344 students, where 66 students need door-to-door pickup. The results are also supported by another BPS instance (the Dick Williams School with 183 students) shown in \ref{apend:addition_sensitivity}.
The virtual walking distance is discussed on the Craig Kimbrel School with 71 students including 11 students need door-to-door pickup. We pick a small instance so that the problem is feasible without applying edge pruning techniques.

\subsubsection{Edge pruning techniques}
\label{subsubsec:edge_compression}

\begin{figure}[h!]%
 \centering
 \subfloat[Number of buses]{\includegraphics[width=.49\linewidth]{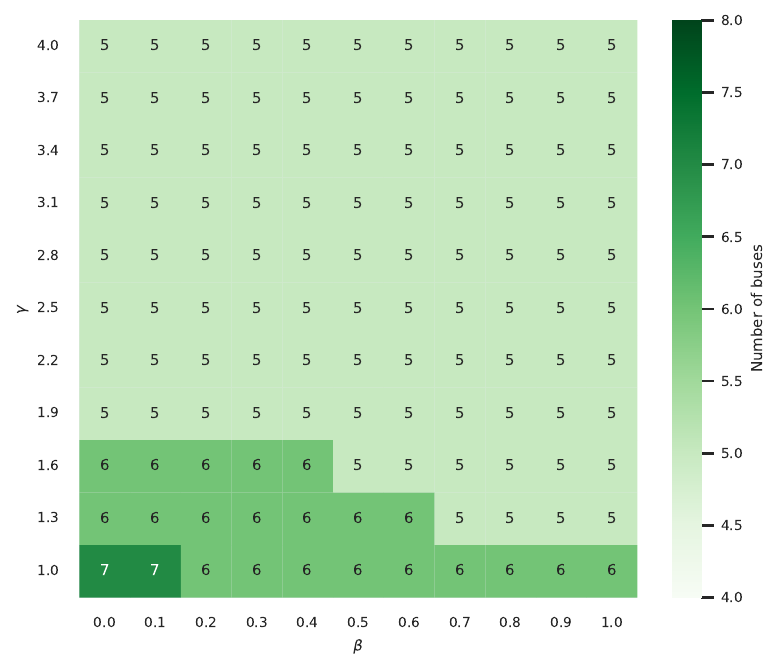}\label{fig:beta_gamma_sensitivity_a}}%
 \subfloat[Objective value]{\includegraphics[width=.49\linewidth]{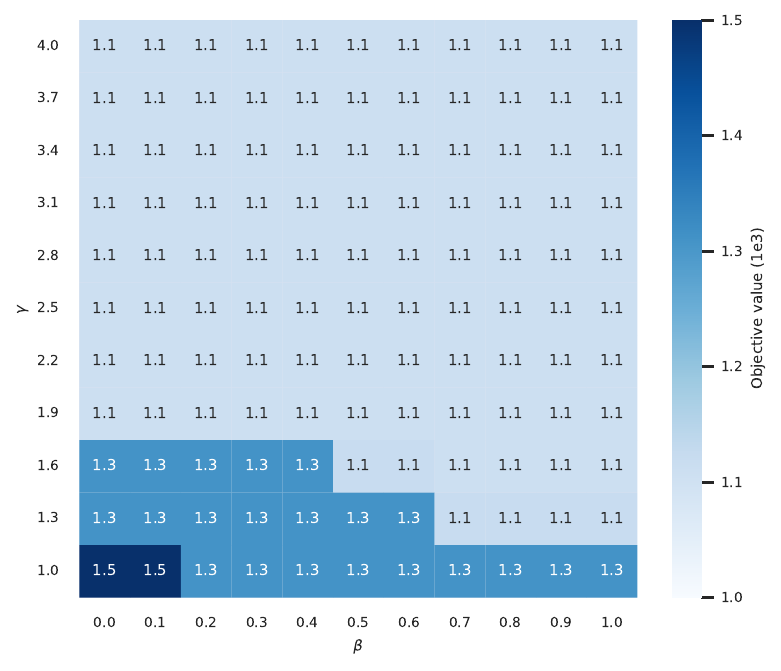}\label{fig:beta_gamma_sensitivity_b}}\\
 \subfloat[Size of trip configuration list]{\includegraphics[width=.49\linewidth]{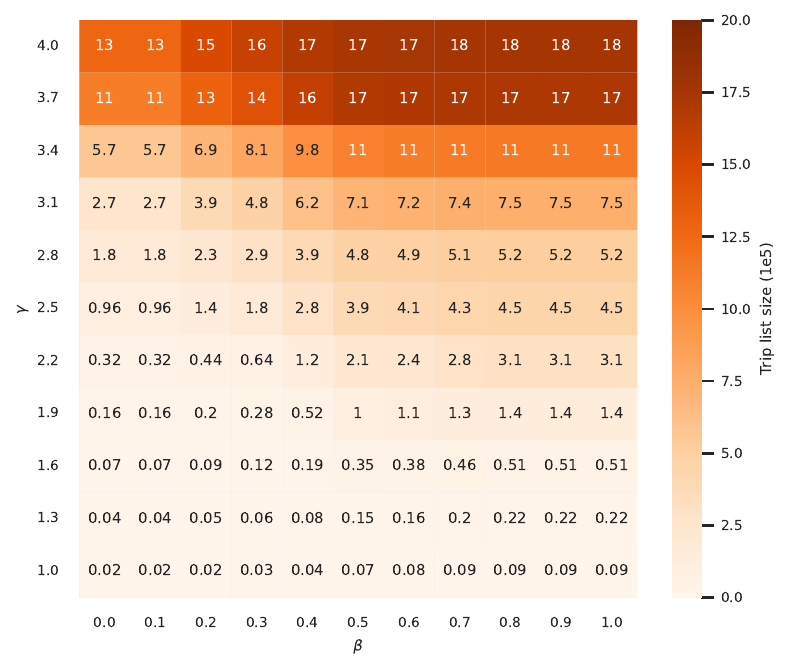}\label{fig:beta_gamma_sensitivity_c}}%
 \subfloat[Computation time]{\includegraphics[width=.49\linewidth]{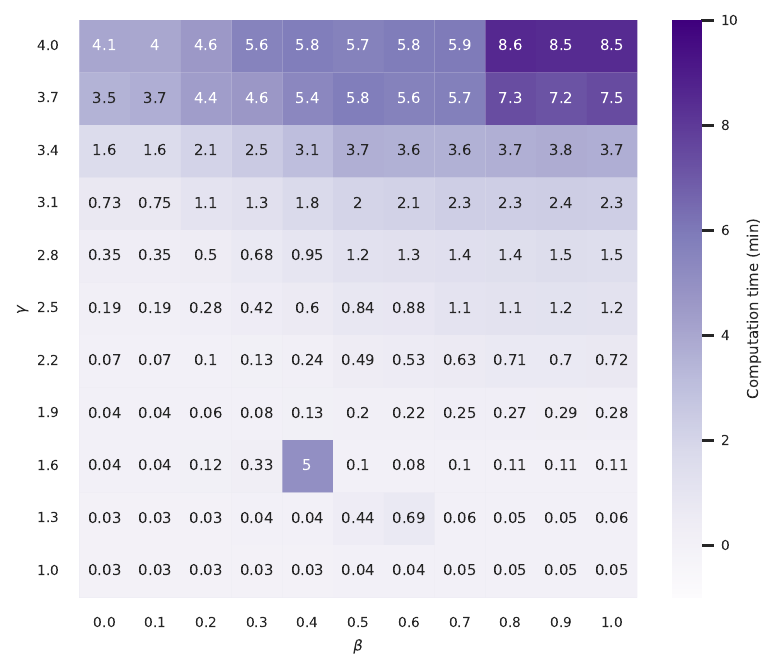}\label{fig:beta_gamma_sensitivity_d}}%
 \caption{Sensitivity analysis for parameters $\beta$ and $\gamma$ in edge pruning techniques.}%
 \label{fig:beta_gamma_sensitivity}%
\end{figure}

To measure the sensitivity of control parameters $\beta$ and $\gamma$ with respect the edge pruning technique, we choose $\beta$ values ranging from 1 to 4 with a step size of 0.3 and $\gamma$ values ranging from 0 to 1.0 with a step size of 0.1.
The sensitivity analyses results are shown in Figure \ref{fig:beta_gamma_sensitivity}.
When $\beta$ and $\gamma$ increase, a larger set of sharing options is made available via the shareability network and school bus schedules can be potentially fit into a smaller number of buses, which leads to better (lower) objective values.
In general, the computation time increases when the size of trip configuration list increases, which is induced by the increase of $\beta$ and $\gamma$.  

The computation time fluctuates in an unexpected manner when $\beta = 1.6$ and $\gamma = 0.4$.  We conjecture that this specific problem instance leads to a particularly complex solution space that is "structurally hard" and challenging for the solver to search through. We note that even small "bad" perturbations of the input to a high-dimensional non-convex problem can lead to disproportionate increases in "hardness" for a solver.

In general, as most large-scale ILP instances, the computation time that the solver takes depends on many factors, including the problem size (e.g., length of trip configuration lists), problem structure (e.g., spatial distribution of pickup locations), and specifics of the solver (e.g., cutting plane methods and local search heuristics). Therefore, even for different schools with the same size of trip configuration list, the computation time can have sizable differences due to the hardness of the instance and for example the integrality gap of the LP relaxation. For some "structurally easy" instances with tighter integrality gaps, the solver is able to find good solutions with a tight lower bound within a very short time, due to the reduced need for branching, and even reach optimality. For other "structurally harder" instances, the solver (Gurobi in our case) has to work much harder (via branching etc.) to reach good solutions. 
An in-depth analysis of the different structural properties of different instances is a very interesting research problem and something we would like to explore, but is beyond the scope of this paper. 

For the Christian Vazquez School instance without alternate modes, the size of trip configuration list without applying heuristics is 1,818,065 and the optimal objective value is 1109.87.
The minimum length of trip configuration list achieving the optimal school bus schedule is 43622 when $\beta = 2.2$ and $\gamma = 0.2$.
With a proper choice of $\beta$ and $\gamma$, the optimal school bus schedule can be found by exploring only 2.4\% feasible trips.
To produce the same school bus schedule without utilizing $\gamma$-quasi-clique process, it requires $\beta = 2.5$ which generates a trip configuration list with 95658 trips.
A problem-specific combination of $\beta$ and $\gamma$ should be used to solve SSRPs more efficiently.

\subsubsection{Alternate modes}

To measure the sensitivity of the alternate modes cost parameter $\alpha_C^{direct}$, we considered a range from 0 (replaced by 0.1 since zero cost is unrealistic) to 5 with a step size of 0.5. 
The sensitivity analyses results are shown in Figure \ref{fig:alternate_sensitivity}. 
Unsurprisingly, we notice that the number of buses needed will increase when $\alpha_C^{direct}$ increases, leading to a higher cost of using alternate modes.
The number of students who take alternate modes decreases when $\alpha_C^{direct}$ increases.
The objective value also increases when $\alpha_C^{direct}$ increases, and it converges to the scenario where alternate modes are not available for students.

\begin{figure}[h!]%
 \centering
 \subfloat[Number of buses]{\includegraphics[scale=.33]{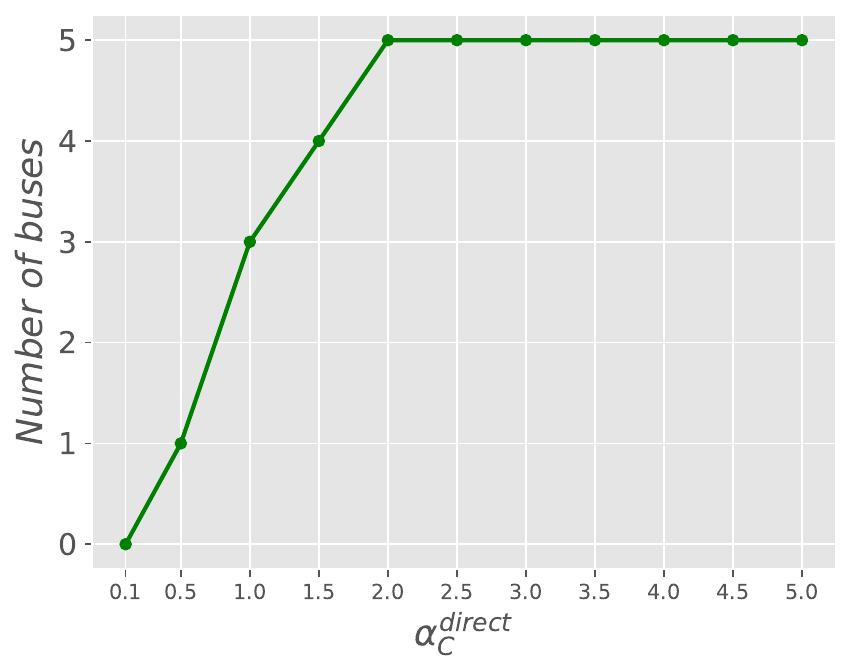}\label{fig:alternate_sensitivity_a}}%
 \subfloat[Objective value]{\includegraphics[scale=.33]{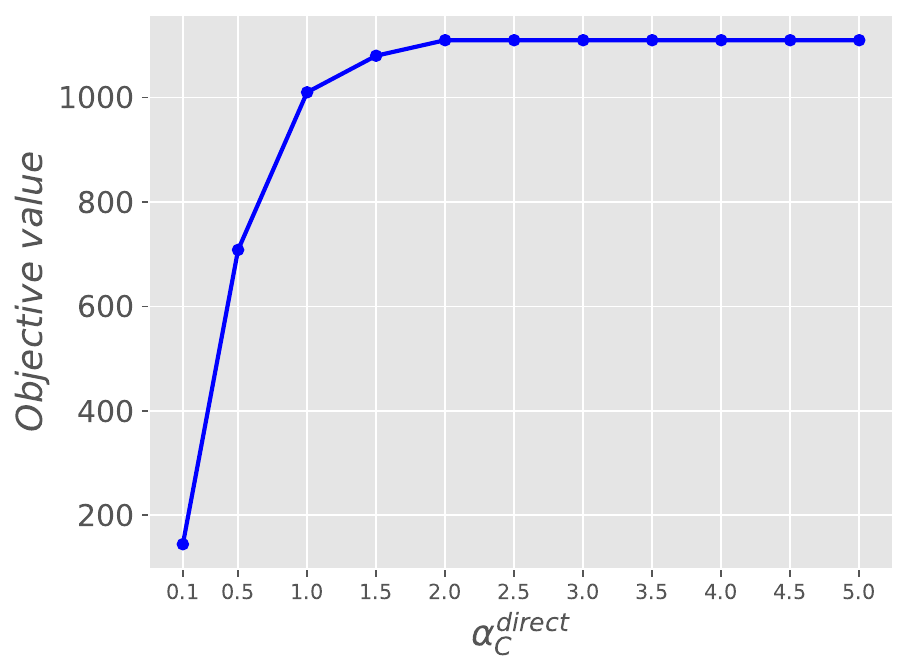}\label{fig:alternate_sensitivity_b}}
 \subfloat[Number of students using alternate modes]{\includegraphics[scale=.33]{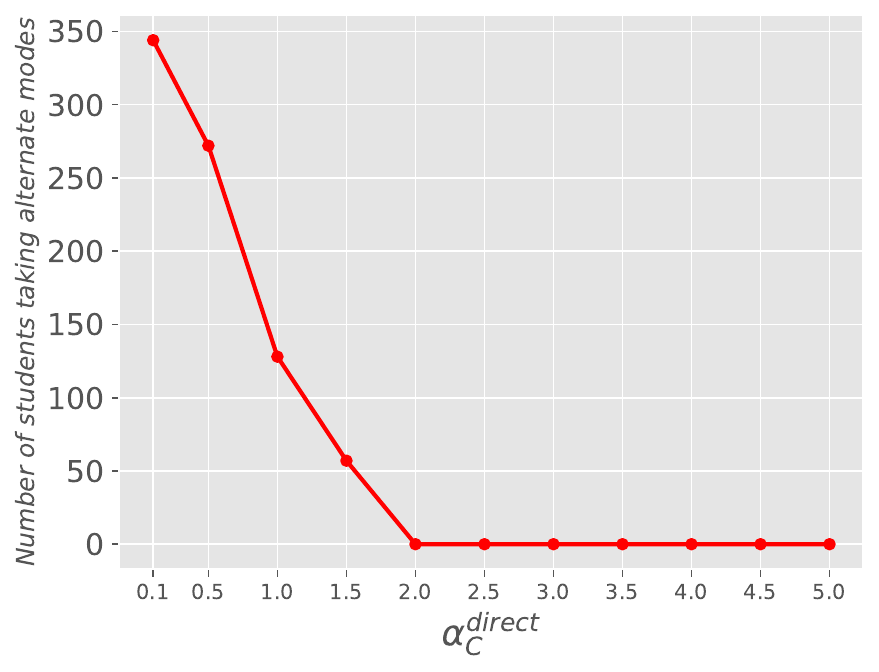}\label{fig:alternate_sensitivity_c}}%
 \caption{Sensitivity analysis for the cost parameter $\alpha_C^{direct}$ of alternate modes.}%
 \label{fig:alternate_sensitivity}%
\end{figure}

\subsubsection{Node compression techniques}

To account for the impact of virtual walking distance parameter in the node compression techniques, we conduct a sensitivity analysis considering distances ranging from 0 to 0.5 with a step size of 0.1. All experiments are conducted on the Craig Kimbrel School (11 door-to-door students out of 71 students) without applying any edge pruning techniques. Also, alternative modes are not considered in this experiment. The results are shown in Table \ref{tab:virtual_walking_distance}. As expected a larger virtual walking distance leads to a reduction in the number of bus stops, which in turn leads to an large reduction in the feasible trip configurations. 

\begin{table}[h!]
\centering
\caption{Sensitivity analyses results for virtual walking distance in node compression techniques. $wd^*$ stands for the virtual walking distance of students need door-to-door pickups.}
\label{tab:virtual_walking_distance}
\begin{tabular}{m{2em} m{2em} m{4em} m{3em} m{2em} m{3em} m{3em}}
 \hline
 $wd^*$ & $N_{M*}$ & $N_{T_b}$ & SND & $N_B$ & GAP & T \\
 \hline
 0 & 26 & 9840120 & 443.75 & 2 & OPT & 3701 \\
 0.1 & 25 & 3581060 & 447.11 & 2 & OPT & 354 \\
 0.2 & 23 & 1105815 & 645.18 & 3 & 24.7\% & 21683 \\
 0.3 & 22 & 499623 & 446.35 & 2 & OPT & 37 \\
 0.4 & 20 & 112040 & 649.04 & 3 & OPT & 585 \\
 0.5 & 20 & 107253 & 648.97 & 3 & OPT & 567 \\
 \hline
\end{tabular}
\end{table}

Regarding the solution quality, the scenario without introducing a virtual walking distance ($wd^* = 0$) produces optimal routes with the minimum cost. Under the optimal scenario, there are 26 bus stops to be served and around 10 million feasible trips considered in the optimization problem. The optimal routes consist of two school buses with 43.75 miles travelled in total\footnote{Each bus costs 200\$ per day and operating a bus costs 1\$ per mile.}. For other scenarios considering a virtual walking distance ($wd^* > 0$), the objective value consists of bus owning cost, bus operation cost and penalties for door-to-door students assigned to bus stops. There is no monotonic relationship between virtual walking distance and objective value due to heterogeneous penalties under different scenarios. On the other hand, larger virtual walking distances lead to more buses used to transport students. This is intuitive; the node compression step minimizes the number of bus stops, and fewer stops lead to crowded bus stops and shorter feasible trip configuration list, which could reduce possibilities for having less buses. 
        
It is worth mentioning that the scenario with 0.2 virtual walking distance includes a "structurally hard" ILP, which can not be solved optimally within the maximum computation time. Overall, the virtual walking distance in the node compression technique induces a trade-off between solution quality and computation complexity. In the case of the Craig Kimbrel School, a 0.5 mile virtual walking distance could cut down 99\% of feasible trips. As the number of binary variables increases exponentially with respect to the number of bus stops, considering virtual walking distances plays a significant role on solving large-scale SSRPs. 

\section{Discussion}
\label{sec:disc}

This paper proposes a shareability network based approach for solving large-scale SSRPs, a reduced SBRP under a single school setting.
We simplified the ILP for solving the SSRP by building the connection between SSRP and WSCP by decomposition through the shareability network.
Moreover, we presented a node compression technique and heuristic edge pruning techniques to obtain a simplified ILP and enable tractable computation of the SSRP at-scale. 
The node compression technique used an ILP to generate bus stops while satisfying maximum walking distance constraints for all students and decreases the number of nodes in the shareability network. 
In contrast, the edge pruning techniques were heuristics that were applied to reduce the density of the shareability network, and worked well in practice.
We evaluated our solution using synthetic dataset from the BPS and show that our approach could compute decent solutions for large-scale SSRP instances. 
Two state-of-the-art SSRP solving techniques were compared with our proposed shareability network based decomposition method to further evaluated performances of our proposed approaches.
Finally, sensitivity analyses of the network compression parameters and cost parameter for alternate modes were provided to get insights into how they influenced trade-offs between solution quality and computation complexity.

It is also worth mentioning that our proposed ILP formulation of the SSRP naturally allows alternate modes for students, which based on our experiments could bring a substantial cost reduction on average regarding the cost for school bus schedules compared to bus-only schedules. Practically speaking, schools can reduce costs by cooperating with multiple transportation service providers (e.g., transit agencies and TNCs) to offer several options for students to travel between home and school. Our approach allows prescribing these options at an individual student level, which allows for accommodating both individual student preferences and legal/policy considerations at the student level. 
The introduction of alternate modes in the school bus scheduling can bring benefits to both the schools by cutting costs and students by providing more options. This also provides an opportunity for TNCs to be integrated into the school transportation ecosystem when appropriate. 

The main limitation of this study is a simplified setting of SSRP which coincides with the problem specification given by the BPS~\cite{BPS_data}.
We consider the SSRP without allowing mixed loads for students from different schools. 
Further work could generalize our proposed methodologies to the mult-school SBRP allowing mixed loads.
We also separated bus stop selection and route generation steps and solved them sequentially, which led to optimality losses.
Further work could incorporate the node compression step into the ILP problem for computing optimal school bus schedules.

This work is the first attempt to adapt techniques from the ridepooling literature on the shareability networks to the SSRP and connect SSRP with the WSCP. 
The key extension for enabling these techniques to work in practice for large-scale SSRP instances is the compression of the shareability network. 
One important future direction is to develop more sophisticated and nuanced edge pruning schemes that more precisely target edges that are unlikely to be relevant to optimal solutions. 
Moreover, the simplified ILP we presented can be combined with state-of-the-art ILP solving techniques (e.g. column generation) to solve extreme large-scale SSRP instances or allow for larger $\beta$ values. This is also an area to be explored. Lastly, the shareability network based decomposition method we present is not restricted to solving SSRPs. It can be fairly easily adapted to many vehicle routing problems which require generating optimal routes and schedules, such as many-to-one shuttle services (e.g., airport shuttles, work shuttles, and first last miles services to transit).

\section*{Acknowledgements}
The contents of this report reflect the views of the authors, who are responsible for the facts and the accuracy of the information presented herein. This document is disseminated in the interest of information exchange. The report is funded, partially or entirely, by a grant from the U.S. Department of Transportation’s University Transportation Centers Program. However, the U.S. Government assumes no liability for the contents or use thereof. The authors would also like to thank Yang Liu for his comments on the paper.

\bibliographystyle{model1-num-names}
\bibliography{reference.bib}

\begin{thebibliography}{36}
\expandafter\ifx\csname natexlab\endcsname\relax\def\natexlab#1{#1}\fi
\providecommand{\bibinfo}[2]{#2}
\ifx\xfnm\relax \def\xfnm[#1]{\unskip,\space#1}\fi
%Type = Misc
\bibitem[{{American School Bus Council}(2013)}]{ASBC}
\bibinfo{author}{{American School Bus Council}}, \bibinfo{title}{Environmental
  benefits}, \bibinfo{howpublished}{\url{
  http://www.americanschoolbuscouncil.org/issues/environmental-benefits }},
  \bibinfo{year}{2013}.
%Type = Manual
\bibitem[{DIG(2017)}]{DIGEST}
\bibinfo{title}{Digest of Education Statistics: 2016},
  \bibinfo{organization}{U.S. Department of Education. Institute of Education
  Sciences, National Center for Education Statistics.}, \bibinfo{year}{2017}.
%Type = Book
\bibitem[{Toth and Vigo(2002)}]{VRP}
\bibinfo{author}{P.~Toth}, \bibinfo{author}{D.~Vigo}, \bibinfo{title}{The
  vehicle routing problem}, \bibinfo{publisher}{Society for Industrial and
  Applied Mathematics}, \bibinfo{address}{Philadelphia}, \bibinfo{year}{2002}.
%Type = Article
\bibitem[{Park et~al.(2012)Park, Tae, and Kim}]{PARK2012204}
\bibinfo{author}{J.~Park}, \bibinfo{author}{H.~Tae}, \bibinfo{author}{B.-I.
  Kim},
\newblock \bibinfo{title}{A post-improvement procedure for the mixed load
  school bus routing problem},
\newblock \bibinfo{journal}{European Journal of Operational Research}
  \bibinfo{volume}{217} (\bibinfo{year}{2012}) \bibinfo{pages}{204 -- 213}.
%Type = Article
\bibitem[{Riera-Ledesma and Salazar-González(2012)}]{RIERALEDESMA2012391}
\bibinfo{author}{J.~Riera-Ledesma}, \bibinfo{author}{J.-J. Salazar-González},
\newblock \bibinfo{title}{Solving school bus routing using the multiple vehicle
  traveling purchaser problem: A branch-and-cut approach},
\newblock \bibinfo{journal}{Computers \& Operations Research}
  \bibinfo{volume}{39} (\bibinfo{year}{2012}) \bibinfo{pages}{391–404}.
%Type = Article
\bibitem[{Schittekat et~al.(2013)Schittekat, Kinable, Sorensen, Sevaux,
  Spieksma, and Springael}]{SCHITTEKAT2013518}
\bibinfo{author}{P.~Schittekat}, \bibinfo{author}{J.~Kinable},
  \bibinfo{author}{K.~Sorensen}, \bibinfo{author}{M.~Sevaux},
  \bibinfo{author}{F.~Spieksma}, \bibinfo{author}{J.~Springael},
\newblock \bibinfo{title}{A metaheuristic for the school bus routing problem
  with bus stop selection},
\newblock \bibinfo{journal}{European Journal of Operational Research}
  \bibinfo{volume}{229} (\bibinfo{year}{2013}) \bibinfo{pages}{518 -- 528}.
%Type = Article
\bibitem[{Santi et~al.(2014)Santi, Resta, Szell, Sobolevsky, Strogatz, and
  Ratti}]{Santi13290}
\bibinfo{author}{P.~Santi}, \bibinfo{author}{G.~Resta},
  \bibinfo{author}{M.~Szell}, \bibinfo{author}{S.~Sobolevsky},
  \bibinfo{author}{S.~H. Strogatz}, \bibinfo{author}{C.~Ratti},
\newblock \bibinfo{title}{Quantifying the benefits of vehicle pooling with
  shareability networks},
\newblock \bibinfo{journal}{Proceedings of the National Academy of Sciences}
  \bibinfo{volume}{111} (\bibinfo{year}{2014}) \bibinfo{pages}{13290--13294}.
%Type = Misc
\bibitem[{{Boston Public Schools} and {Department of
  Transportation}(2016)}]{BPS_principle}
\bibinfo{author}{{Boston Public Schools}}, \bibinfo{author}{{Department of
  Transportation}}, \bibinfo{title}{2016-2017 {BPS} {DOT} principal \&
  headmaster handbook}, \bibinfo{howpublished}{\url{
  https://www.bostonpublicschools.org/cms/lib/MA01906464/Centricity/Domain/2263/FY17\%20Principal\%20Handbook.doc}},
  \bibinfo{year}{2016}.
%Type = Article
\bibitem[{Newton and Thomas(1969)}]{NEWTON196975}
\bibinfo{author}{R.~M. Newton}, \bibinfo{author}{W.~H. Thomas},
\newblock \bibinfo{title}{Design of school bus routes by computer},
\newblock \bibinfo{journal}{Socio-Economic Planning Sciences}
  \bibinfo{volume}{3} (\bibinfo{year}{1969}) \bibinfo{pages}{75 -- 85}.
%Type = Article
\bibitem[{Park and Kim(2010)}]{PARK2010311}
\bibinfo{author}{J.~Park}, \bibinfo{author}{B.-I. Kim},
\newblock \bibinfo{title}{The school bus routing problem: A review},
\newblock \bibinfo{journal}{European Journal of Operational Research}
  \bibinfo{volume}{202} (\bibinfo{year}{2010}) \bibinfo{pages}{311 -- 319}.
%Type = Article
\bibitem[{Ellegood et~al.(2020)Ellegood, Solomon, North, and
  Campbell}]{ELLEGOOD2020102056}
\bibinfo{author}{W.~A. Ellegood}, \bibinfo{author}{S.~Solomon},
  \bibinfo{author}{J.~North}, \bibinfo{author}{J.~F. Campbell},
\newblock \bibinfo{title}{School bus routing problem: Contemporary trends and
  research directions},
\newblock \bibinfo{journal}{Omega} \bibinfo{volume}{95} (\bibinfo{year}{2020})
  \bibinfo{pages}{102056}.
%Type = Article
\bibitem[{Bekta{\c{s}} and Elmasta{\c{s}}(2007)}]{Bekta2007}
\bibinfo{author}{T.~Bekta{\c{s}}}, \bibinfo{author}{S.~Elmasta{\c{s}}},
\newblock \bibinfo{title}{Solving school bus routing problems through integer
  programming},
\newblock \bibinfo{journal}{Journal of the Operational Research Society}
  \bibinfo{volume}{58} (\bibinfo{year}{2007}) \bibinfo{pages}{1599--1604}.
%Type = Inproceedings
\bibitem[{Sghaier et~al.(2013)Sghaier, Guedria, and
  {Mraihi}}]{Sghaier_Guedria_Mraihi}
\bibinfo{author}{S.~B. Sghaier}, \bibinfo{author}{N.~B. Guedria},
  \bibinfo{author}{R.~{Mraihi}},
\newblock \bibinfo{title}{Solving school bus routing problem with genetic
  algorithm},
\newblock in: \bibinfo{booktitle}{2013 International Conference on Advanced
  Logistics and Transport}, pp. \bibinfo{pages}{7--12}.
%Type = Article
\bibitem[{Ellegood et~al.(2015)Ellegood, Campbell, and
  North}]{Ellegood_Campbell_North_2015}
\bibinfo{author}{W.~A. Ellegood}, \bibinfo{author}{J.~F. Campbell},
  \bibinfo{author}{J.~North},
\newblock \bibinfo{title}{Continuous approximation models for mixed load school
  bus routing},
\newblock \bibinfo{journal}{Transportation Research Part B: Methodological}
  \bibinfo{volume}{77} (\bibinfo{year}{2015}) \bibinfo{pages}{182–198}.
%Type = Article
\bibitem[{Braca et~al.(1997)Braca, Bramel, Posner, and Simchi-levi}]{BRACA1997}
\bibinfo{author}{J.~Braca}, \bibinfo{author}{J.~Bramel},
  \bibinfo{author}{B.~Posner}, \bibinfo{author}{D.~Simchi-levi},
\newblock \bibinfo{title}{A computerized approach to the new york city school
  bus routing problem},
\newblock \bibinfo{journal}{IIE Transactions} \bibinfo{volume}{29}
  (\bibinfo{year}{1997}) \bibinfo{pages}{693--702}.
%Type = Article
\bibitem[{Shafahi et~al.(2018)Shafahi, Wang, and
  Haghani}]{Shafahi_Wang_Haghani_2018}
\bibinfo{author}{A.~Shafahi}, \bibinfo{author}{Z.~Wang},
  \bibinfo{author}{A.~Haghani},
\newblock \bibinfo{title}{Speedroute: Fast, efficient solutions for school bus
  routing problems},
\newblock \bibinfo{journal}{Transportation Research Part B: Methodological}
  \bibinfo{volume}{117} (\bibinfo{year}{2018}) \bibinfo{pages}{473–493}.
%Type = Article
\bibitem[{Miranda et~al.(2018)Miranda, de~Camargo, Conceição, Porto, and
  Nunes}]{Miranda_2018}
\bibinfo{author}{D.~M. Miranda}, \bibinfo{author}{R.~S. de~Camargo},
  \bibinfo{author}{S.~V. Conceição}, \bibinfo{author}{M.~F. Porto},
  \bibinfo{author}{N.~T. Nunes},
\newblock \bibinfo{title}{A multi-loading school bus routing problem},
\newblock \bibinfo{journal}{Expert Systems with Applications}
  \bibinfo{volume}{101} (\bibinfo{year}{2018}) \bibinfo{pages}{228–242}.
%Type = Article
\bibitem[{Sales et~al.(2018)Sales, Melo, Bonates, and
  Prata}]{Sales_Melo_Bonates_Prata_2018}
\bibinfo{author}{L.~d. P.~A. Sales}, \bibinfo{author}{C.~S. Melo},
  \bibinfo{author}{T.~d. O.~e. Bonates}, \bibinfo{author}{B.~d.~A. Prata},
\newblock \bibinfo{title}{Memetic algorithm for the heterogeneous fleet school
  bus routing problem},
\newblock \bibinfo{journal}{Journal of Urban Planning and Development}
  \bibinfo{volume}{144} (\bibinfo{year}{2018}) \bibinfo{pages}{04018018}.
%Type = Article
\bibitem[{Mokhtari and Ghezavati(2018)}]{Mokhtari_Ghezavati_2018}
\bibinfo{author}{N.-a. Mokhtari}, \bibinfo{author}{V.~Ghezavati},
\newblock \bibinfo{title}{Integration of efficient multi-objective ant-colony
  and a heuristic method to solve a novel multi-objective mixed load school bus
  routing model},
\newblock \bibinfo{journal}{Applied Soft Computing} \bibinfo{volume}{68}
  (\bibinfo{year}{2018}) \bibinfo{pages}{92–109}.
%Type = Article
\bibitem[{Bertsimas et~al.(2019)Bertsimas, Delarue, and Martin}]{MIT_SBRP}
\bibinfo{author}{D.~Bertsimas}, \bibinfo{author}{A.~Delarue},
  \bibinfo{author}{S.~Martin},
\newblock \bibinfo{title}{Optimizing schools{\textquoteright} start time and
  bus routes},
\newblock \bibinfo{journal}{Proceedings of the National Academy of Sciences}
  \bibinfo{volume}{116} (\bibinfo{year}{2019}) \bibinfo{pages}{5943--5948}.
%Type = Article
\bibitem[{Alonso-Mora et~al.(2017)Alonso-Mora, Samaranayake, Wallar, Frazzoli,
  and Rus}]{Alonso-Mora462}
\bibinfo{author}{J.~Alonso-Mora}, \bibinfo{author}{S.~Samaranayake},
  \bibinfo{author}{A.~Wallar}, \bibinfo{author}{E.~Frazzoli},
  \bibinfo{author}{D.~Rus},
\newblock \bibinfo{title}{On-demand high-capacity ride-sharing via dynamic
  trip-vehicle assignment},
\newblock \bibinfo{journal}{Proceedings of the National Academy of Sciences}
  \bibinfo{volume}{114} (\bibinfo{year}{2017}) \bibinfo{pages}{462--467}.
%Type = Article
\bibitem[{Kucharski and Cats(2020)}]{Kucharski2020}
\bibinfo{author}{R.~Kucharski}, \bibinfo{author}{O.~Cats},
\newblock \bibinfo{title}{{Exact matching of attractive shared rides (ExMAS)
  for system-wide strategic evaluations}},
\newblock \bibinfo{journal}{Transportation Research Part B: Methodological}
  \bibinfo{volume}{139} (\bibinfo{year}{2020}) \bibinfo{pages}{285--310}.
%Type = Article
\bibitem[{Tafreshian and Masoud(2020)}]{Tafreshian2020}
\bibinfo{author}{A.~Tafreshian}, \bibinfo{author}{N.~Masoud},
\newblock \bibinfo{title}{{Trip-based graph partitioning in dynamic
  ridesharing}},
\newblock \bibinfo{journal}{Transportation Research Part C: Emerging
  Technologies} \bibinfo{volume}{114} (\bibinfo{year}{2020})
  \bibinfo{pages}{532--553}.
%Type = Article
\bibitem[{Luo et~al.(2021)Luo, Nagarajan, Sundt, Yin, Vincent, and
  Shahabi}]{Luo2021}
\bibinfo{author}{Q.~Luo}, \bibinfo{author}{V.~Nagarajan},
  \bibinfo{author}{A.~Sundt}, \bibinfo{author}{Y.~Yin},
  \bibinfo{author}{J.~Vincent}, \bibinfo{author}{M.~Shahabi},
\newblock \bibinfo{title}{{Efficient Algorithms for Stochastic Ridepooling
  Assignment with Mixed Fleets}}  (\bibinfo{year}{2021})
  \bibinfo{pages}{1--39}.
%Type = Article
\bibitem[{Syed et~al.(2019)Syed, Kaltenhaeuser, Gaponova, and
  Bogenberger}]{Syed2019}
\bibinfo{author}{A.~A. Syed}, \bibinfo{author}{B.~Kaltenhaeuser},
  \bibinfo{author}{I.~Gaponova}, \bibinfo{author}{K.~Bogenberger},
\newblock \bibinfo{title}{{Asynchronous Adaptive Large Neighborhood Search
  Algorithm for Dynamic Matching Problem in Ride Hailing Services}},
\newblock \bibinfo{journal}{2019 IEEE Intelligent Transportation Systems
  Conference, ITSC 2019}  (\bibinfo{year}{2019}) \bibinfo{pages}{3006--3012}.
%Type = Article
\bibitem[{Simonetto et~al.(2019)Simonetto, Monteil, and
  Gambella}]{Simonetto2019}
\bibinfo{author}{A.~Simonetto}, \bibinfo{author}{J.~Monteil},
  \bibinfo{author}{C.~Gambella},
\newblock \bibinfo{title}{{Real-time city-scale ridesharing via linear
  assignment problems}},
\newblock \bibinfo{journal}{Transportation Research Part C: Emerging
  Technologies} \bibinfo{volume}{101} (\bibinfo{year}{2019})
  \bibinfo{pages}{208--232}.
%Type = Article
\bibitem[{Wang and Yang(2019)}]{Wang_Yang_2019}
\bibinfo{author}{H.~Wang}, \bibinfo{author}{H.~Yang},
\newblock \bibinfo{title}{Ridesourcing systems: A framework and review},
\newblock \bibinfo{journal}{Transportation Research Part B: Methodological}
  \bibinfo{volume}{129} (\bibinfo{year}{2019}) \bibinfo{pages}{122–155}.
%Type = Inproceedings
\bibitem[{{Guo} et~al.(2018){Guo}, {Liu}, and {Samaranayake}}]{Xiaotong}
\bibinfo{author}{X.~{Guo}}, \bibinfo{author}{Y.~{Liu}},
  \bibinfo{author}{S.~{Samaranayake}},
\newblock \bibinfo{title}{Solving the school bus routing problem at scale via a
  compressed shareability network},
\newblock in: \bibinfo{booktitle}{2018 21st International Conference on
  Intelligent Transportation Systems (ITSC)}, pp. \bibinfo{pages}{1900--1907}.
%Type = Article
\bibitem[{Laporte(1992)}]{Laporte_1992}
\bibinfo{author}{G.~Laporte},
\newblock \bibinfo{title}{The traveling salesman problem: An overview of exact
  and approximate algorithms},
\newblock \bibinfo{journal}{European Journal of Operational Research}
  \bibinfo{volume}{59} (\bibinfo{year}{1992}) \bibinfo{pages}{231–247}.
%Type = Misc
\bibitem[{{Boston Public Schools}(2017)}]{BPS_data}
\bibinfo{author}{{Boston Public Schools}}, \bibinfo{title}{Transportation
  challenge}, \bibinfo{howpublished}{\url{
  https://www.bostonpublicschools.org/transportationchallenge}},
  \bibinfo{year}{2017}.
%Type = Misc
\bibitem[{{OpenStreetMap contributors}(2017)}]{OpenStreetMap}
\bibinfo{author}{{OpenStreetMap contributors}}, \bibinfo{title}{{Planet dump
  retrieved from https://planet.osm.org }}, \bibinfo{howpublished}{\url{
  https://www.openstreetmap.org }}, \bibinfo{year}{2017}.
%Type = Article
\bibitem[{Boeing(2017)}]{BOEING2017126}
\bibinfo{author}{G.~Boeing},
\newblock \bibinfo{title}{Osmnx: New methods for acquiring, constructing,
  analyzing, and visualizing complex street networks},
\newblock \bibinfo{journal}{Computers, Environment and Urban Systems}
  \bibinfo{volume}{65} (\bibinfo{year}{2017}) \bibinfo{pages}{126 -- 139}.
%Type = Inproceedings
\bibitem[{Hagberg et~al.(2008)Hagberg, Schult, and Swart}]{SciPyProceedings_11}
\bibinfo{author}{A.~A. Hagberg}, \bibinfo{author}{D.~A. Schult},
  \bibinfo{author}{P.~J. Swart},
\newblock \bibinfo{title}{Exploring network structure, dynamics, and function
  using networkx},
\newblock in: \bibinfo{editor}{G.~Varoquaux}, \bibinfo{editor}{T.~Vaught},
  \bibinfo{editor}{J.~Millman} (Eds.), \bibinfo{booktitle}{Proceedings of the
  7th Python in Science Conference}, \bibinfo{address}{Pasadena, CA USA}, pp.
  \bibinfo{pages}{11 -- 15}.
%Type = Misc
\bibitem[{{Gurobi Optimization, Inc.}(2016)}]{gurobi}
\bibinfo{author}{{Gurobi Optimization, Inc.}}, \bibinfo{title}{Gurobi optimizer
  reference manual}, \bibinfo{year}{2016}.
%Type = Misc
\bibitem[{{Alleghany county public schools}(2016)}]{Bus}
\bibinfo{author}{{Alleghany county public schools}}, \bibinfo{title}{School bus
  per-mile operating cost analysis}, \bibinfo{year}{2016}.
%Type = Misc
\bibitem[{{Ridester Staff}(2018)}]{Uber}
\bibinfo{author}{{Ridester Staff}}, \bibinfo{title}{The comprehensive guide:
  How much does uber cost?}, \bibinfo{howpublished}{\url{
  https://www.ridester.com/uber-rates-cost }}, \bibinfo{year}{2018}.

\end{thebibliography}

\newpage
\appendix
\section{Heuristic insertion path-TSP solver}
\label{apend:tsp}

One of the most time-consuming parts in our proposed decomposition method through the shareability network is generating the bus trip list $T_b$.
In Algorithm \ref{alg1}, we need to solve a path-TSP every time we find a clique in the shareability network to determine the feasibility of this trip. 
The path-TSP itself is NP-hard and it takes an off-the-shelf solver seconds to solve even with a small amount of students as the input.
The problem becomes intractable if we have to check millions of cliques in the Algorithm \ref{alg1}.
Therefore, we propose a heuristic insertion path-TSP solving technique to decrease the computation time while outputting a satisfying feasible bus trip list $T_b$.

\begin{algorithm}[!b]
\caption{Trips feasibility check with the heuristic inserting path-TSP solving technique. Input: a feasible bus trip $\tau \in T_b$, the optimal route $p_\tau^*$ for the trip $\tau$, the optimal travel time $t_{\tau}^*$ for the trip $\tau$, a student $s$, maximum travel time $t^{max}$, bus capacity $C$, travel time function between any two vertices in the road network $\{t_{ij}|\forall i,j \in V_r\}$.}
\label{alg3}
\begin{algorithmic}[1]
\Function{HeuristicFeasibilityCheck}{$\tau, p_{\tau}^*,t_{\tau}^*, s, t^{max},C,\{t_{ij}|\forall i,j \in V_r\}$}
    \State $n \gets |S(\tau)|$
    \State $p_{\tau}^* := [s_1,s_2,...,s_n]$
    \State $t^* \gets \infty, p^* \gets \emptyset$
    \For{$i$ in $[0,1,...,n]$}
        \If{i = 0}
            \State $t' \gets t_{v_sv_{s_1}} + t_{\tau}^*$ 
            \If{$t' < t^*$}
                \State $t^* \gets t'; p^* \gets [s,s_1,s_2,...,s_n]$
            \EndIf
        \ElsIf{i = n}
            \State $t' \gets t_{v_{s_n}v_s} + t_{v_sv_d} - t_{v_{s_n}v_d} + t_{\tau}^*$ 
            \If{$t' < t^*$}
                \State $t^* \gets t'; p^* \gets [s_1,s_2,...,s_n,s]$
            \EndIf
        \Else
            \State $t' \gets t_{v_{s_i}v_s} + t_{v_sv_{s_{i+1}}} - t_{v_{s_i}v_{s_{i+1}}} + t_{\tau}^*$ 
            \If{$t' < t^*$}
                \State $t^* \gets t'; p^* \gets [s_1,...,s_i,s,s_{i+1},...,s_n]$
            \EndIf 
        \EndIf
    \EndFor
    \If {$t^* \leq t^{max}$ and $|S(\tau)| + 1 \leq C$}
        \State \textbf{return} true
    \Else
        \State \textbf{return} false
    \EndIf
\EndFunction
\end{algorithmic}
\end{algorithm}

The essential idea of this heuristic insertion path-TSP solving technique is that if we know the optimal path $p^*$ for a set of $k$ students $S_k$, we generate a sub-optimal path for $k+1$ students $S_k \cup \{ s\}$ by fixing $p^*$ and inserting the student $s$ into the order of $p^*$ where yields a path with the minimal travel time. 
To be specific, we modified the function $\sfunction{FeasibilityCheck}(\tau,t_max,C,\sfunction{PathTsp}(\cdot))$ to Algorithm \ref{alg3}.
For the input of Algorithm \ref{alg3}, we need the optimal path and travel time for a feasible trip $\tau$. 
We can store the optimal routes and travel time once we generate a feasible trip $\tau$ with $k$ students, which will be used when considering trips with $k+1$ students.
 
Algorithm \ref{alg3} yields a sub-optimal route and travel time within linear time.
The experiment results in Section \ref{sec:experiments} show the routes computed by Algorithm \ref{alg3} are decent.

\section{Experiment results without alternate modes: Boston Public Schools (BPS) synthetic benchmark data}
\label{apen:results_base}

This section shows the experiment results without alternate modes for best school bus schedules given synthetic data from the Boston Public School (BPS).
For figures in the following, the red star denotes the school location and blue dots represent student locations. 

\begin{figure*}[h]
\centering
\includegraphics[scale=0.3]{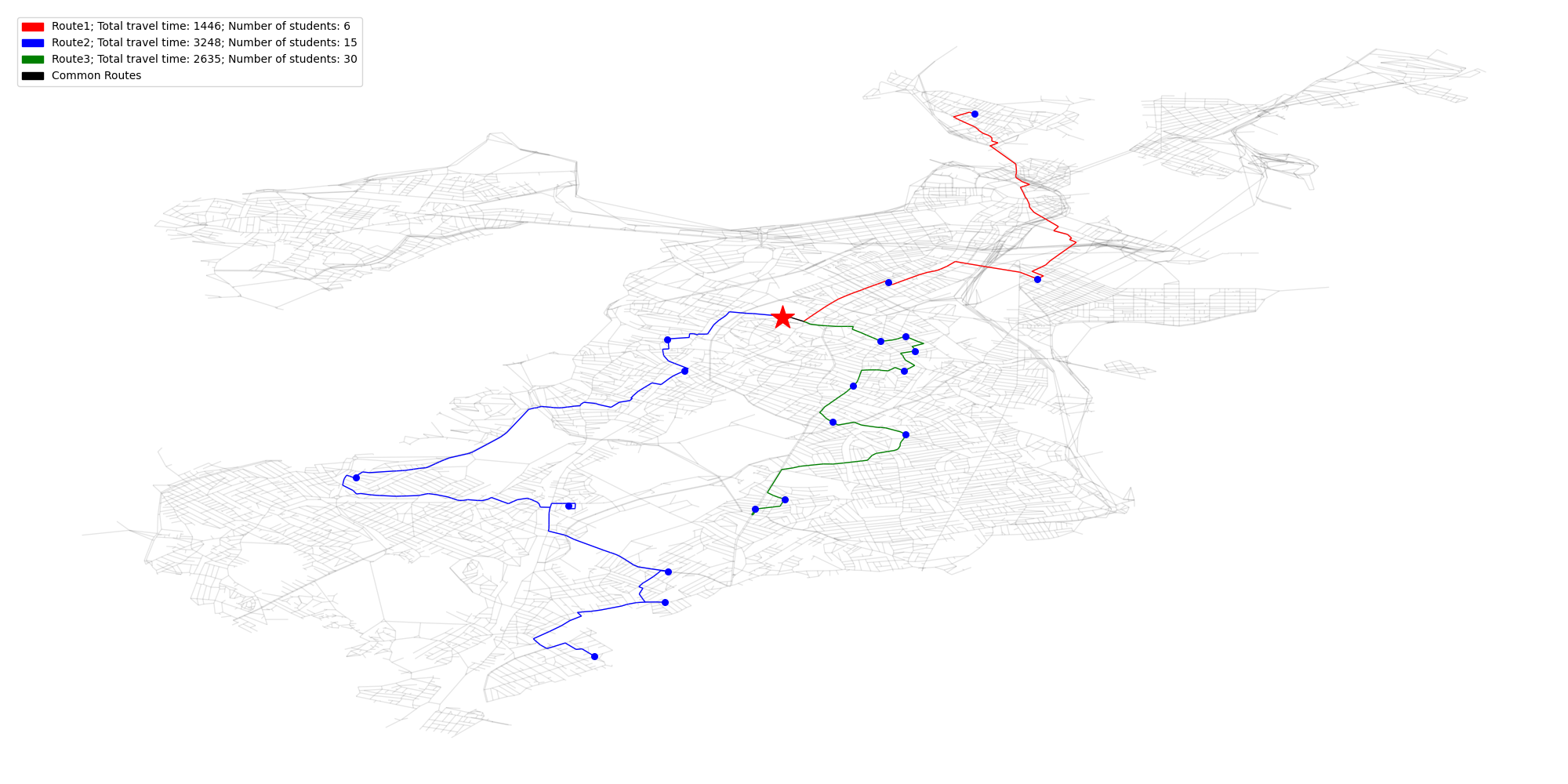}
\caption{School bus schedules for the Tommy Harper (Optimal)}
\end{figure*}

\begin{figure*}[h]
\centering
\includegraphics[scale=0.3]{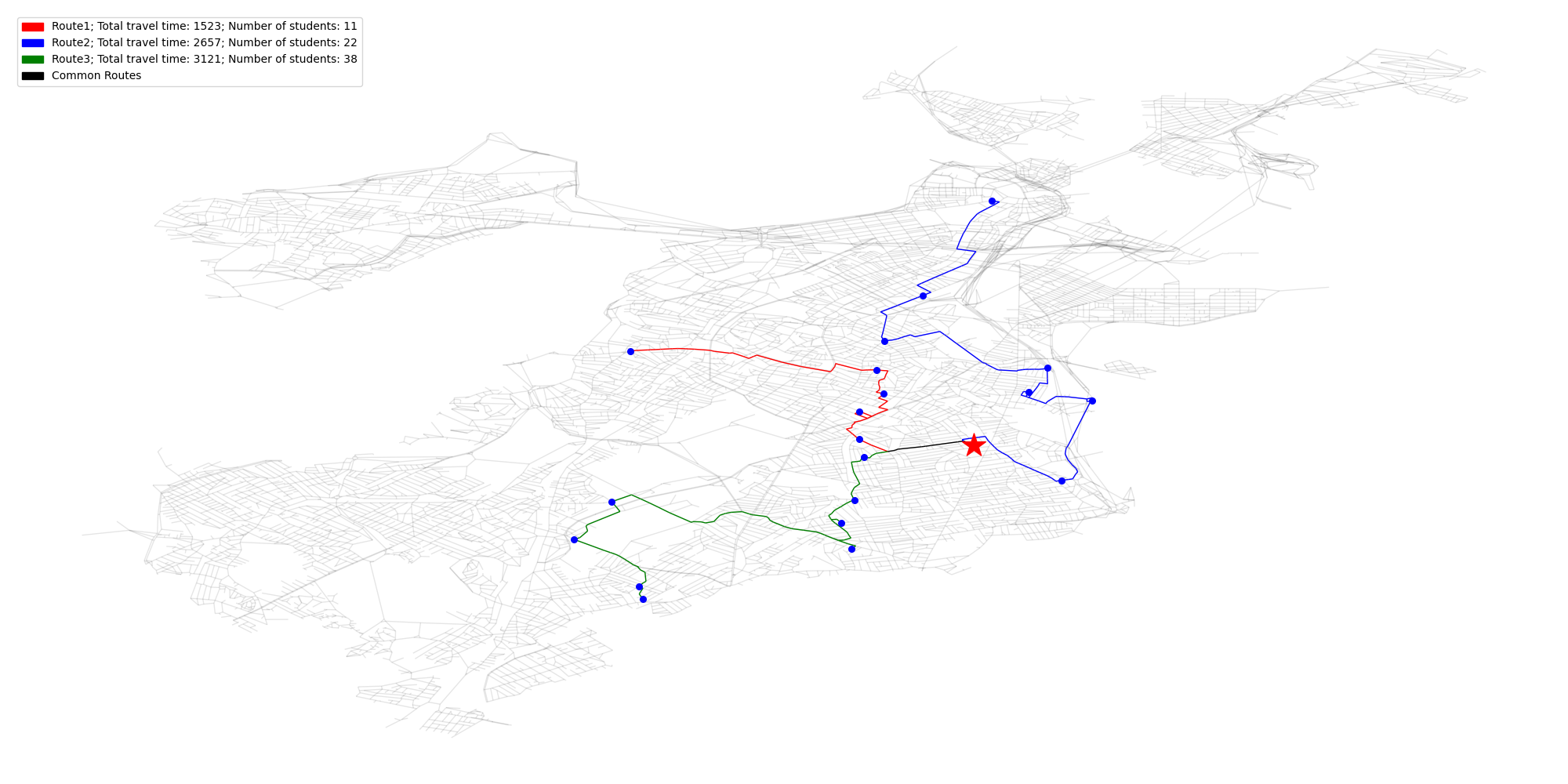}
\caption{School bus schedules for the Craig Kimbrel (Optimal)}
\end{figure*}

\begin{figure*}[h]
\centering
\includegraphics[scale=0.3]{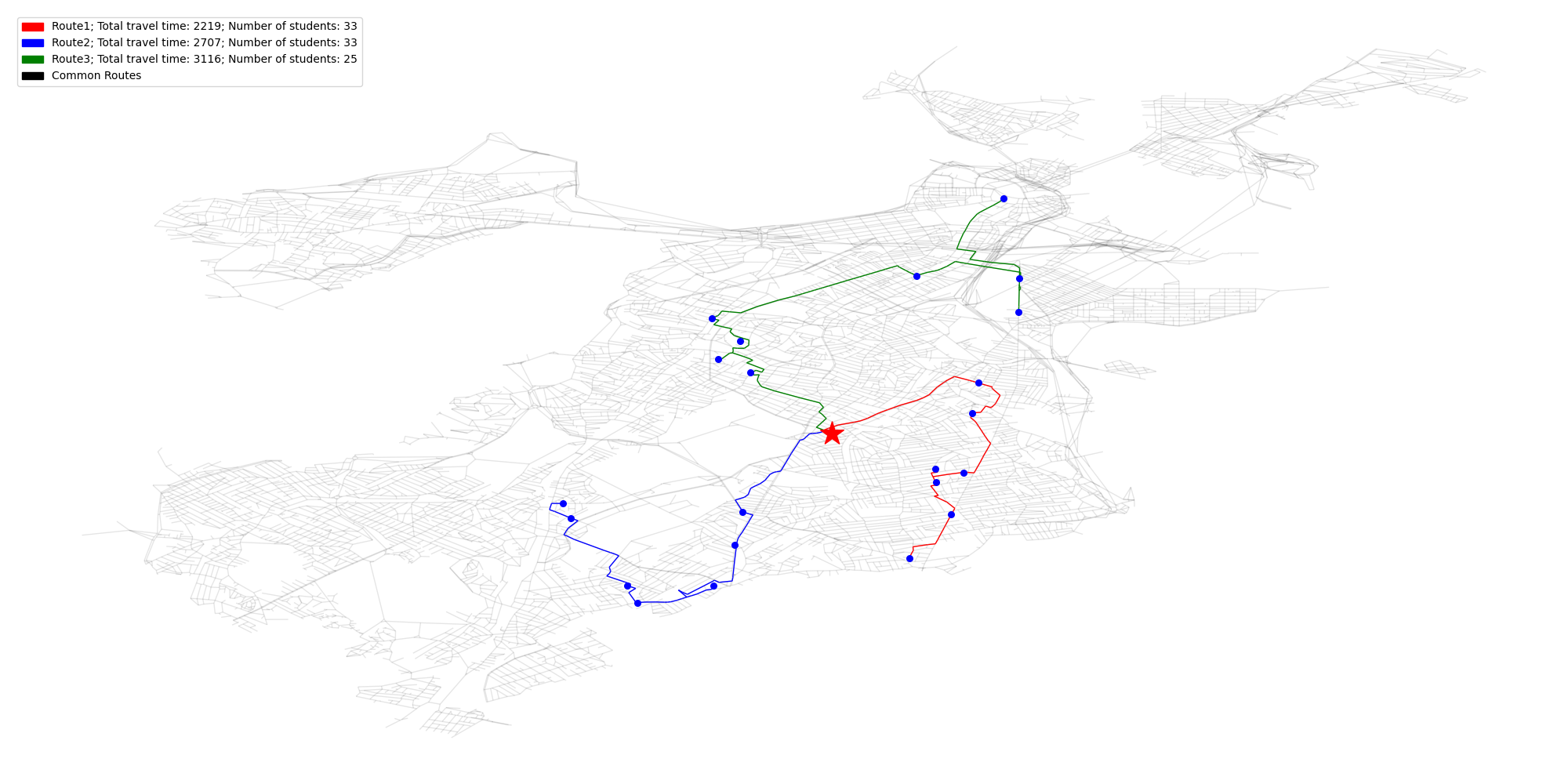}
\caption{School bus schedules for the Deven Marrero (Optimal)}
\end{figure*}

\begin{figure*}[h]
\centering
\includegraphics[scale=0.3]{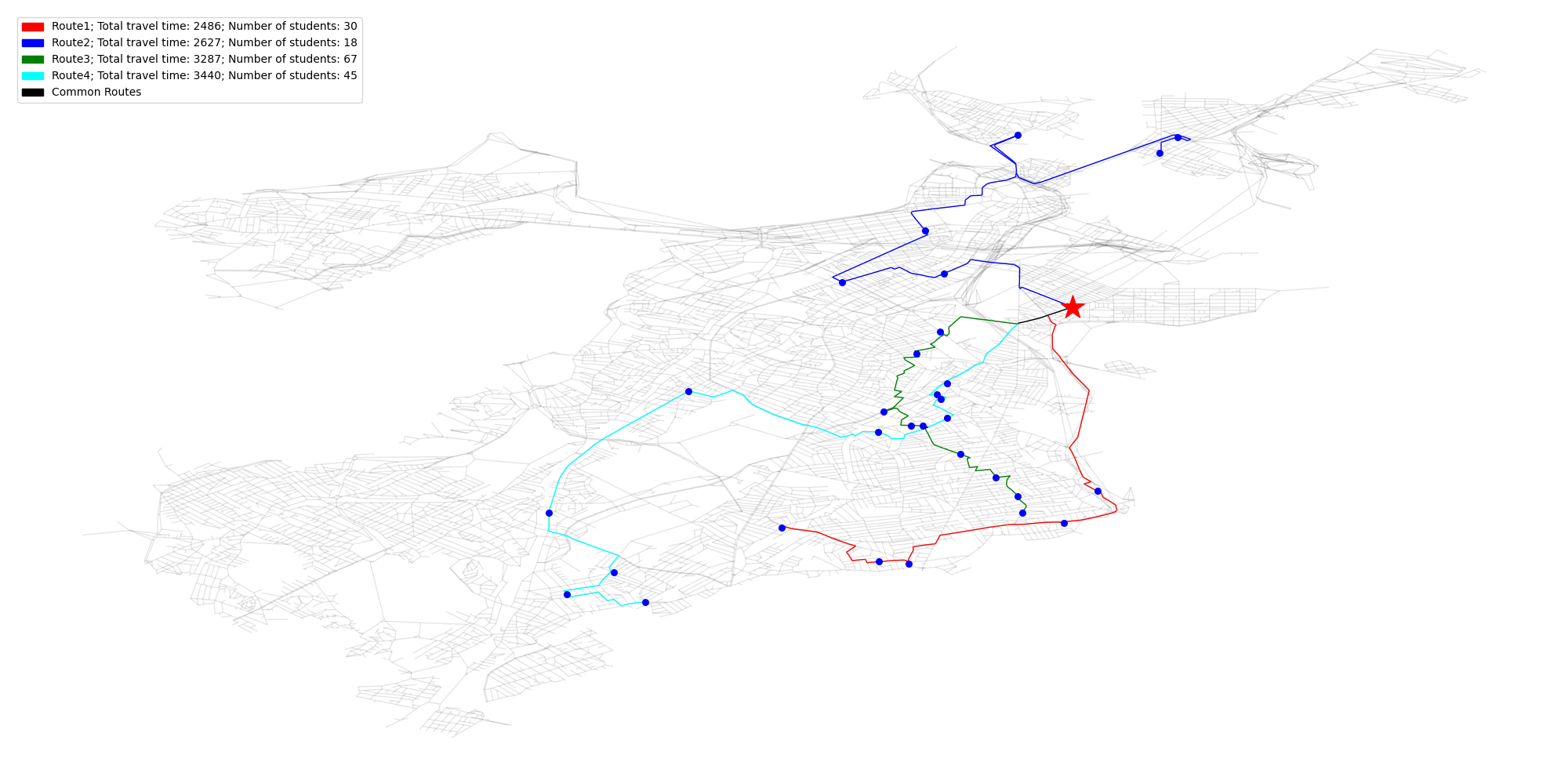}
\caption{School bus schedules for the Frank Malzone (sub-optimal with $\beta = 1.5$, $\gamma = 0.4$)}
\end{figure*}

\begin{figure*}[h]
\centering
\includegraphics[scale=0.3]{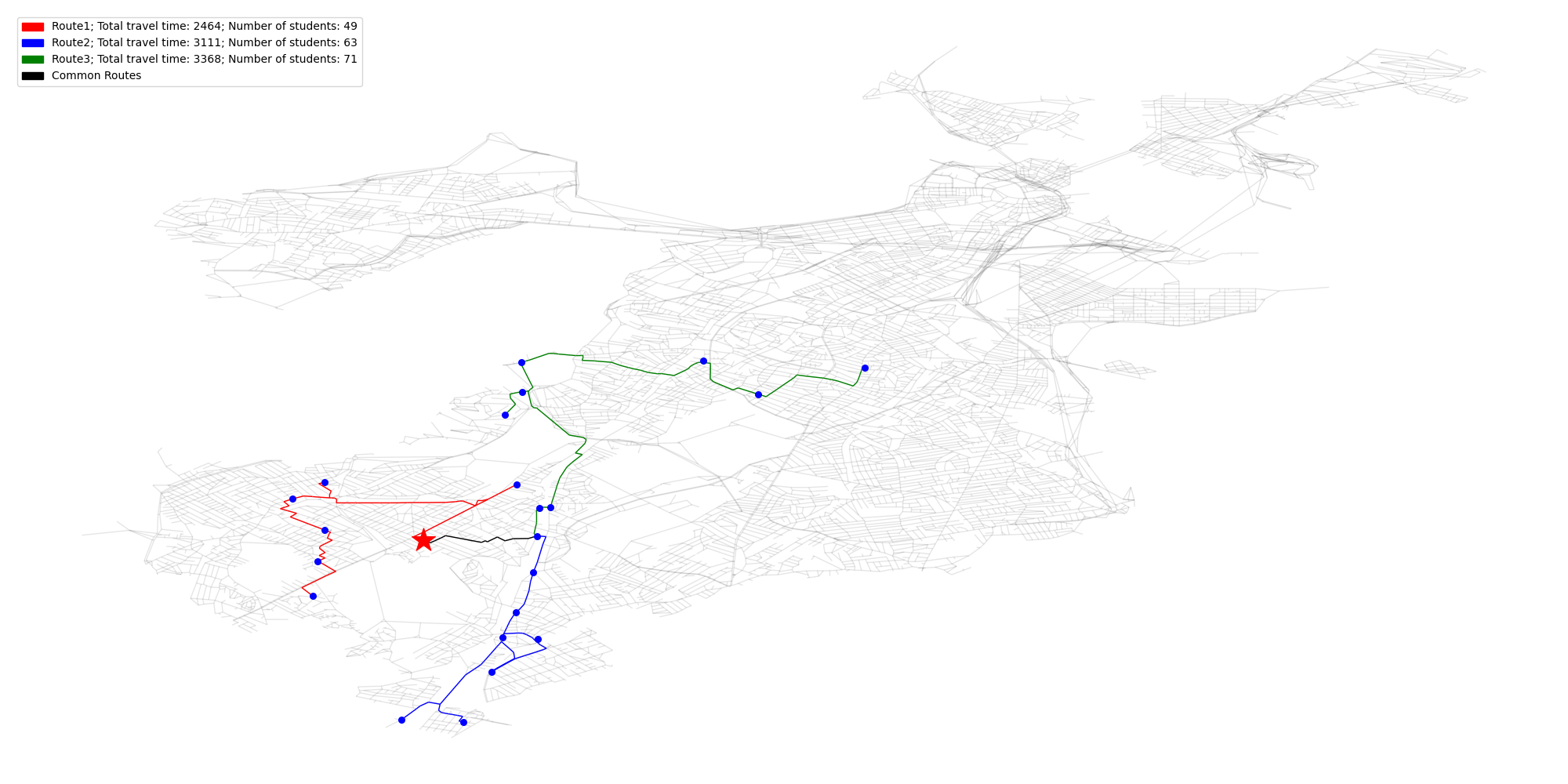}
\caption{School bus schedules for the Dick Williams (Optimal)}
\end{figure*}

\begin{figure*}[h]
\centering
\includegraphics[scale=0.3]{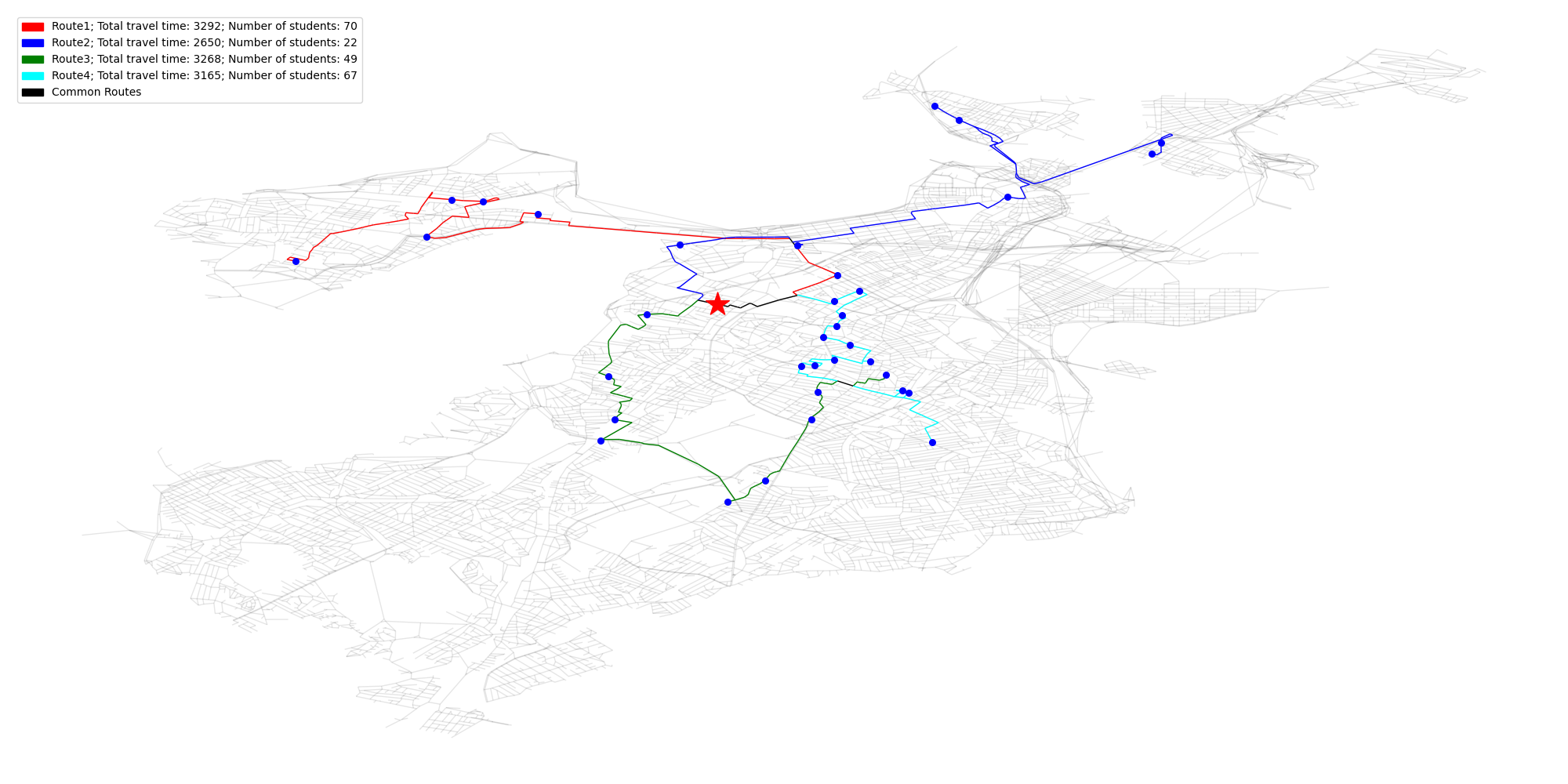}
\caption{School bus schedules for the Dick Bresciani (sub-optimal with $\beta = 1.5$, $\gamma = 0.4$)}
\end{figure*}

\begin{figure*}[h]
\centering
\includegraphics[scale=0.3]{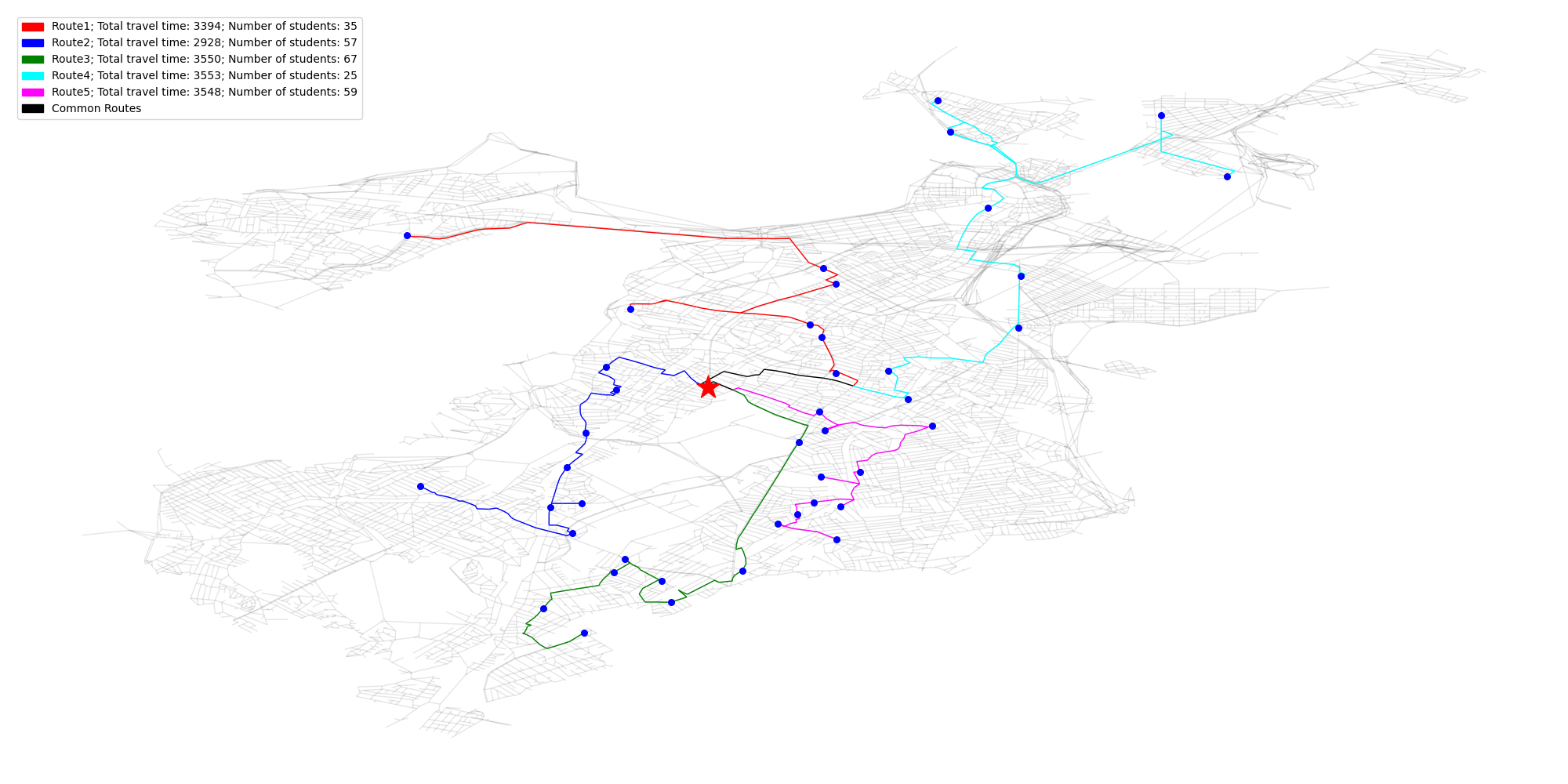}
\caption{School bus schedules for the Dutch Leonard (sub-optimal with $\beta = 2$, $\gamma = 0.4$)}
\end{figure*}

\begin{figure*}[h]
\centering
\includegraphics[scale=0.3]{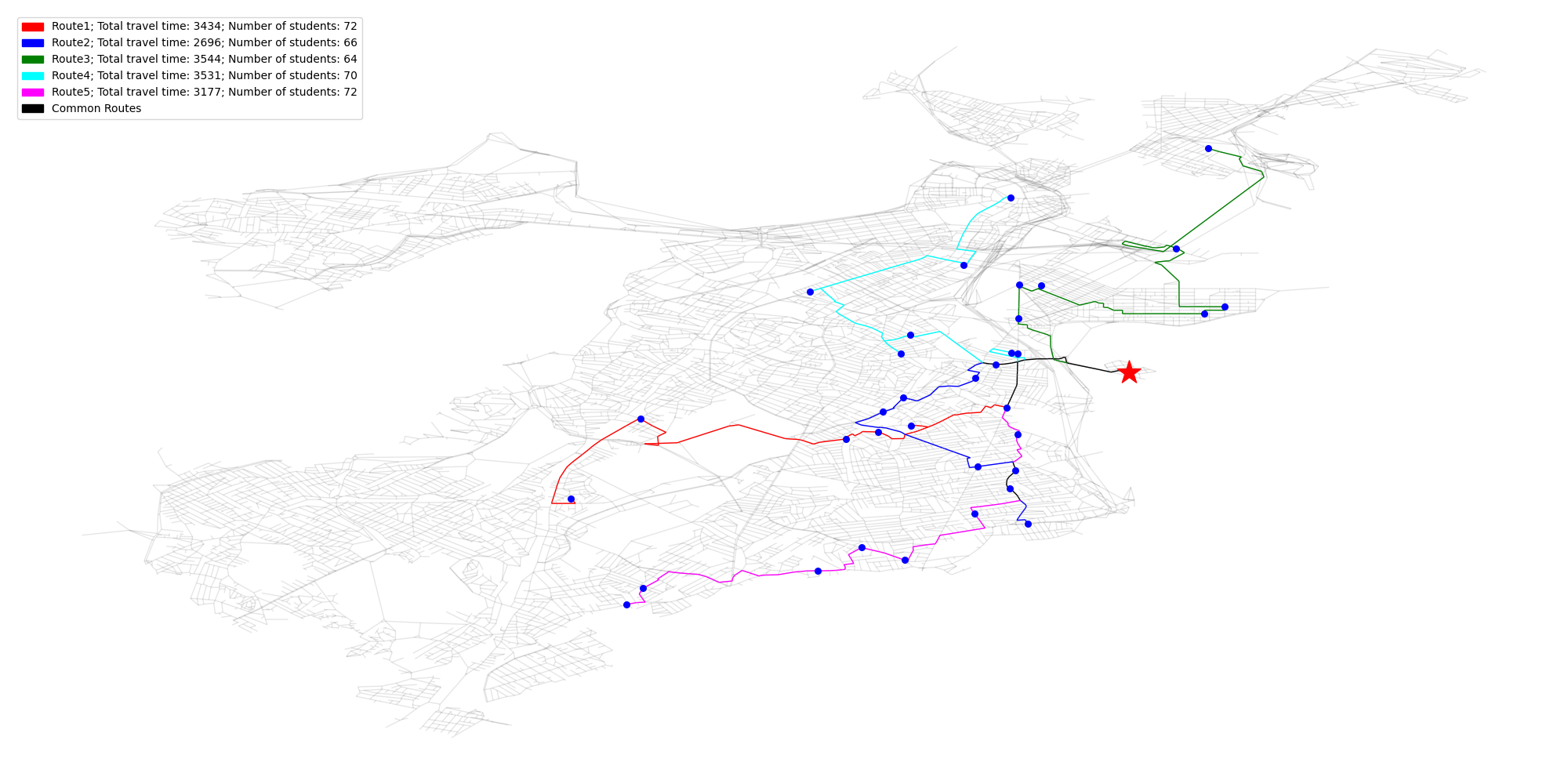}
\caption{School bus schedules for the Christian Vazquez (sub-optimal with $\beta = 3.5$, $\gamma = 0.4$)}
\end{figure*}

\begin{figure*}[h]
\centering
\includegraphics[scale=0.3]{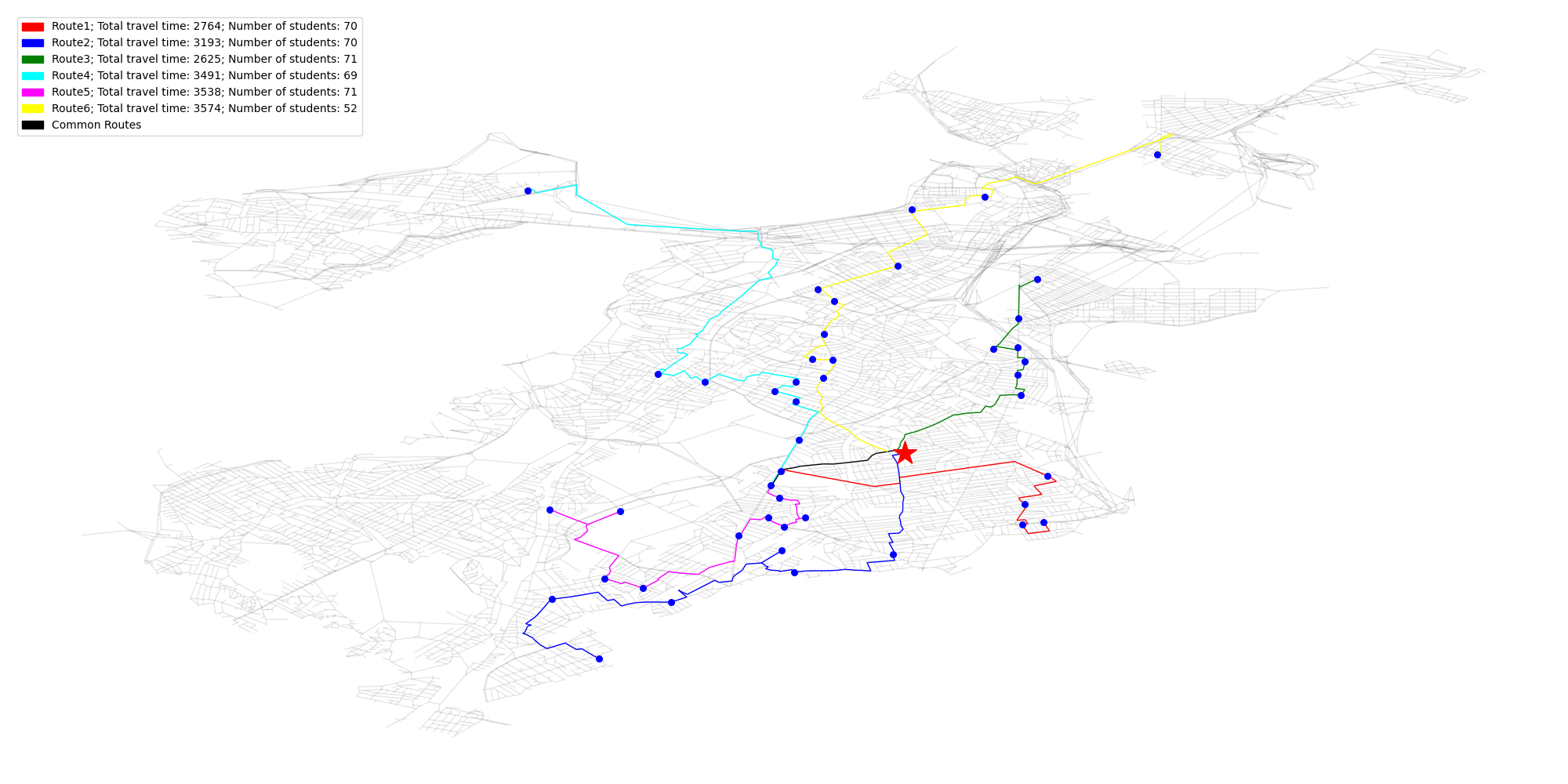}
\caption{School bus schedules for the Dennis Eckerley (sub-optimal with $\beta = 2.5$, $\gamma = 0.4$)}
\end{figure*}

\begin{figure*}[h]
\centering
\includegraphics[scale=0.3]{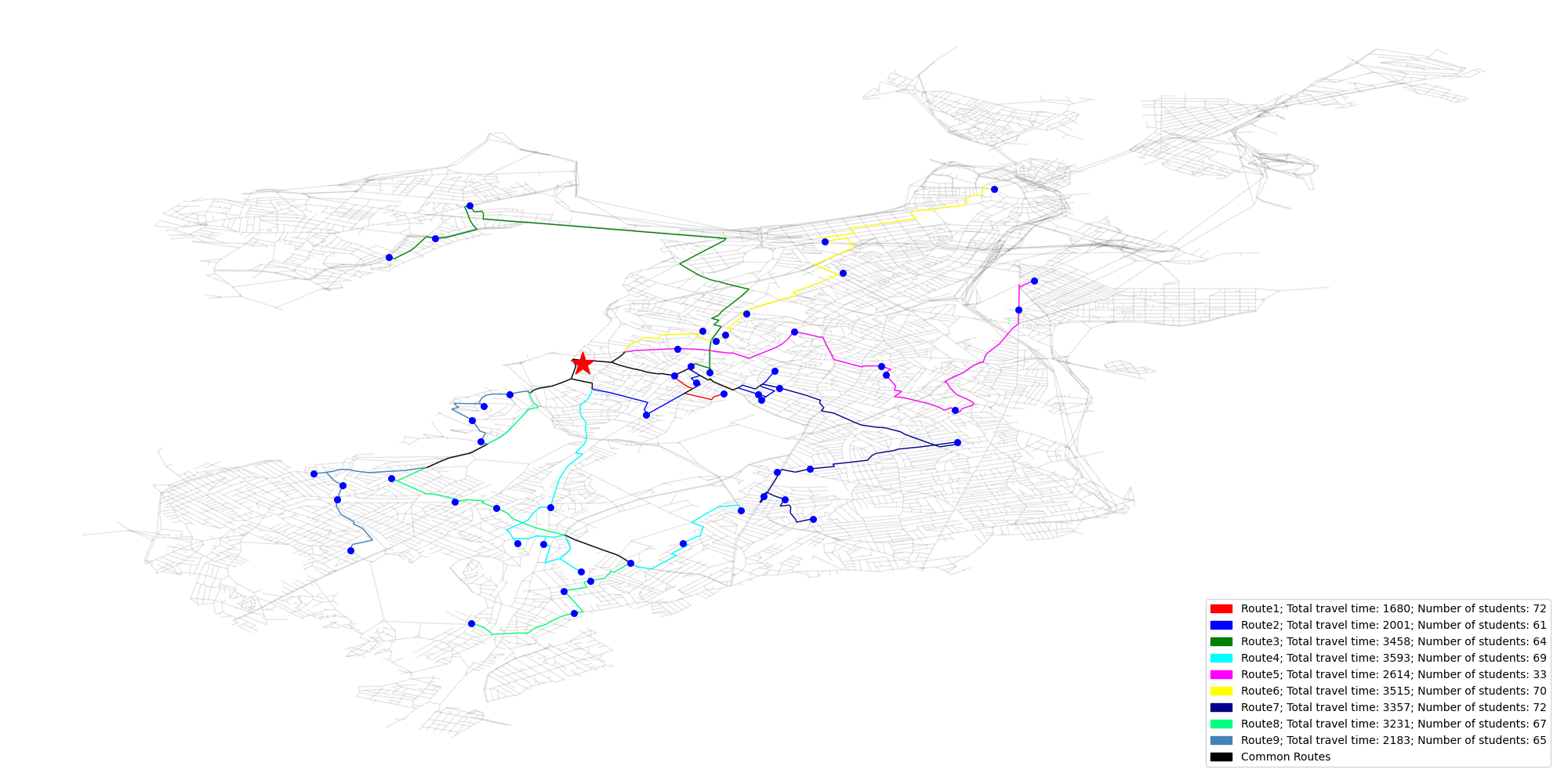}
\caption{School bus schedules for the Rick Ferrell (sub-optimal with $\beta = 2.5$, $\gamma = 0.4$)}
\end{figure*}

\FloatBarrier

\section{Experiment results with alternate modes: Boston Public Schools (BPS) synthetic benchmark data}
\label{apen:results_alternate}

This section shows the experiment results with alternate modes for best school bus schedules given synthetic data from the Boston Public School (BPS).
For figures in the following, the red star denotes the school location, blue dots represent student locations and red crosses indicate students who take dedicated vehicles. 

\begin{figure*}[h]
\centering
\includegraphics[scale=0.3]{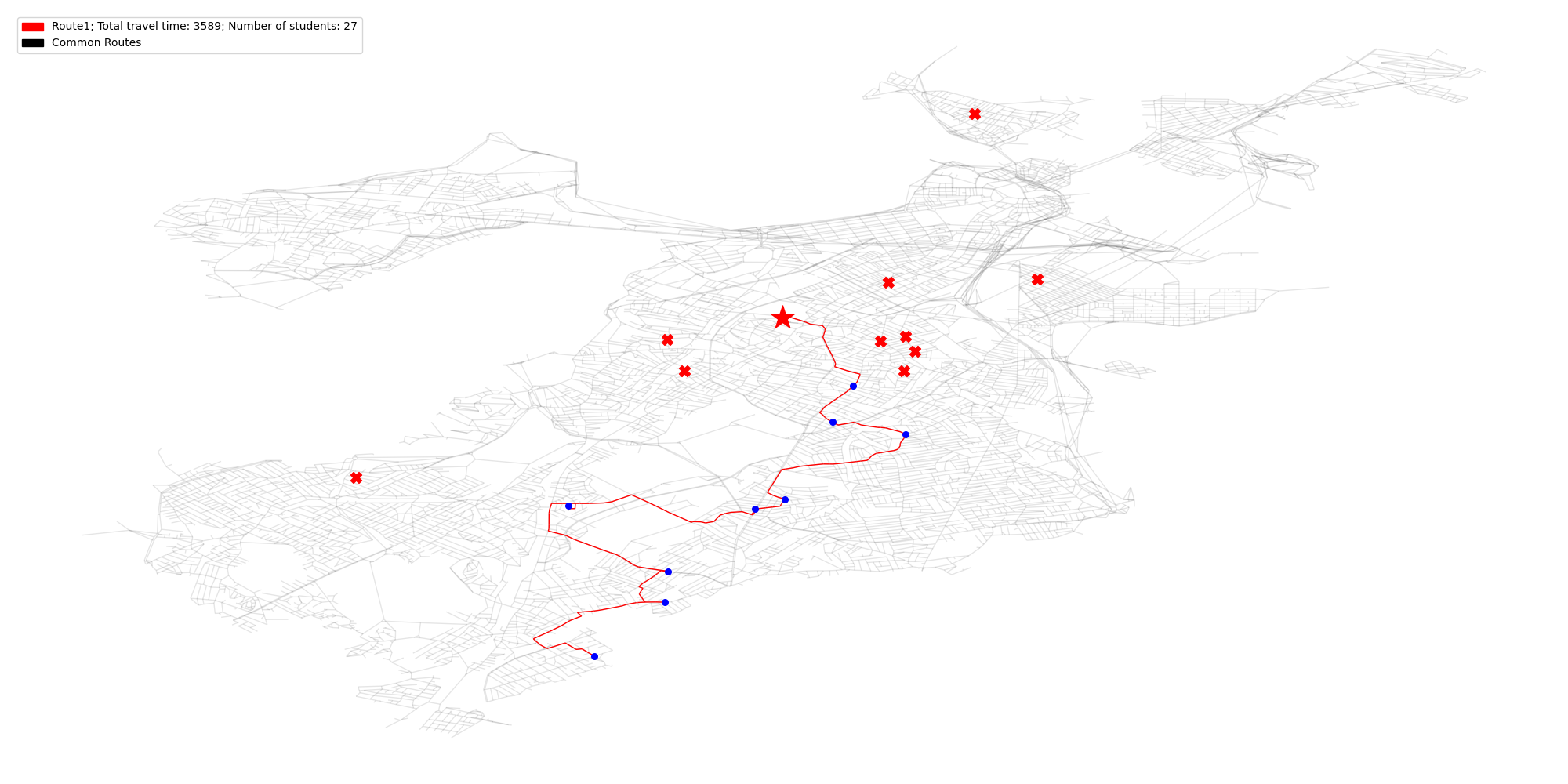}
\caption{School bus schedules for the Tommy Harper (Optimal)}
\end{figure*}

\begin{figure*}[h]
\centering
\includegraphics[scale=0.3]{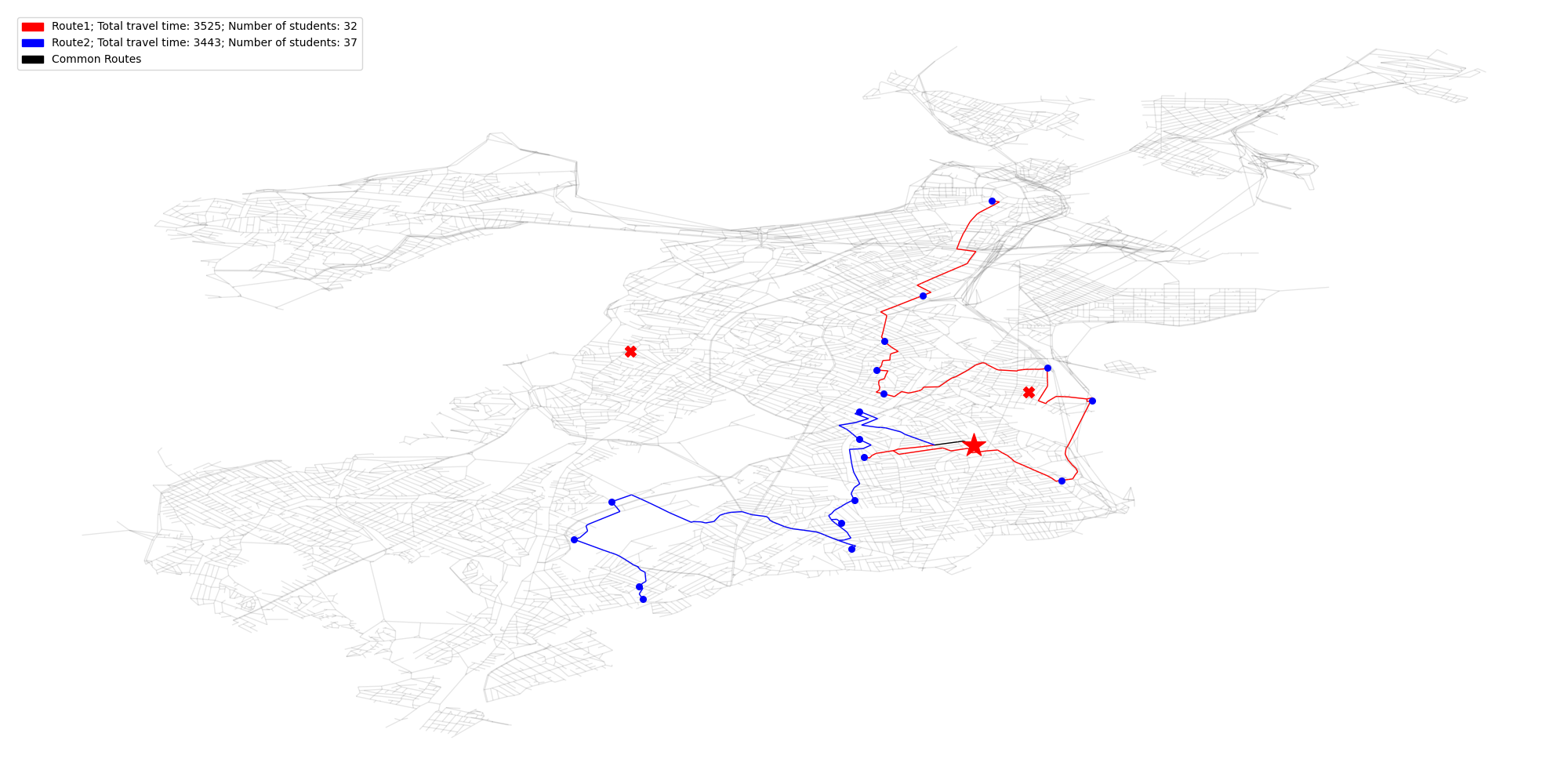}
\caption{School bus schedules for the Craig Kimbrel (Optimal)}
\end{figure*}

\begin{figure*}[h]
\centering
\includegraphics[scale=0.3]{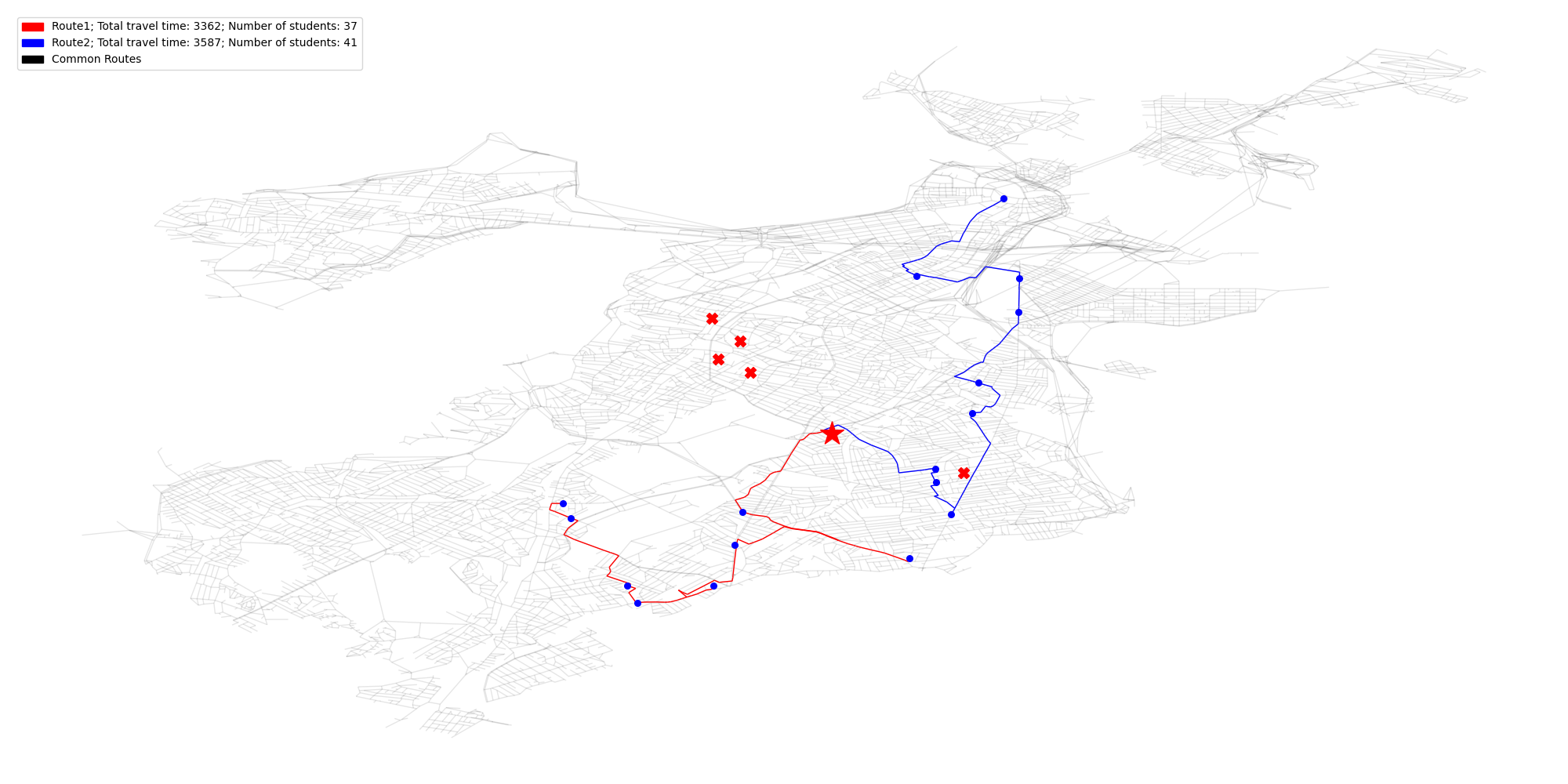}
\caption{School bus schedules for the Deven Marrero (Optimal)}
\end{figure*}

\begin{figure*}[h]
\centering
\includegraphics[scale=0.3]{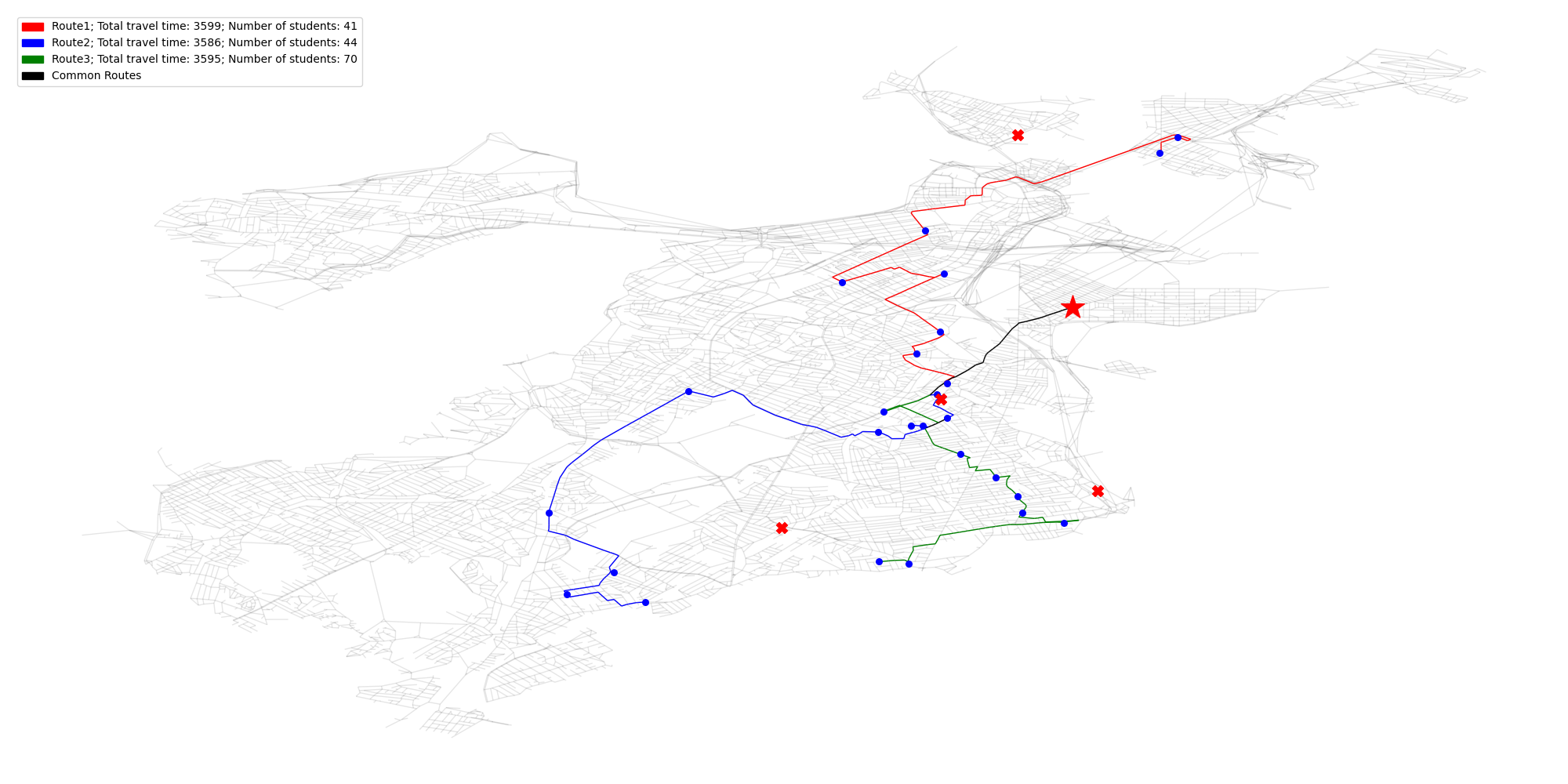}
\caption{School bus schedules for the Frank Malzone (sub-optimal with $\beta = 1.5$, $\gamma = 0.4$)}
\end{figure*}

\begin{figure*}[h]
\centering
\includegraphics[scale=0.3]{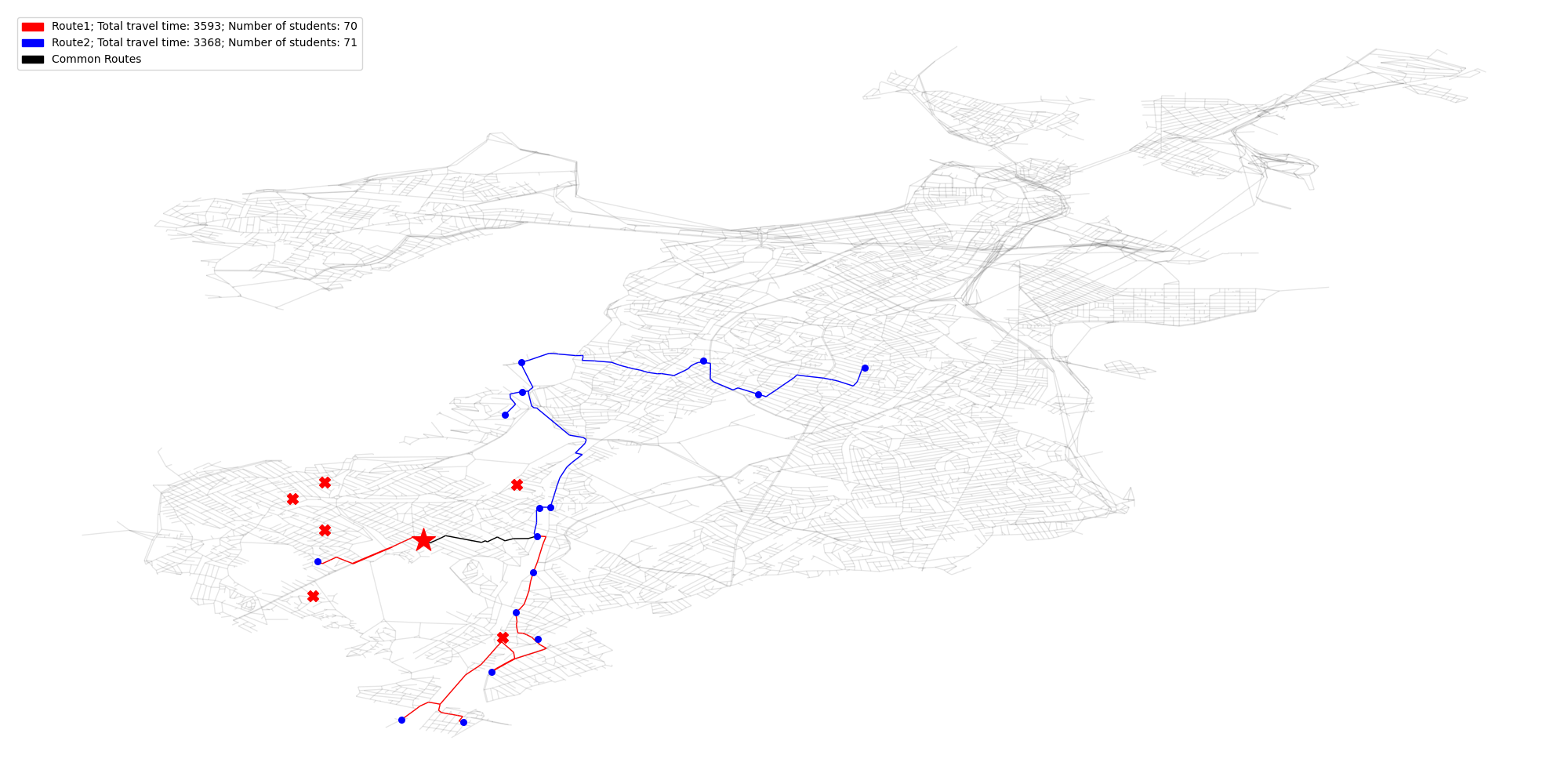}
\caption{School bus schedules for the Dick Williams (Optimal)}
\end{figure*}

\begin{figure*}[h]
\centering
\includegraphics[scale=0.3]{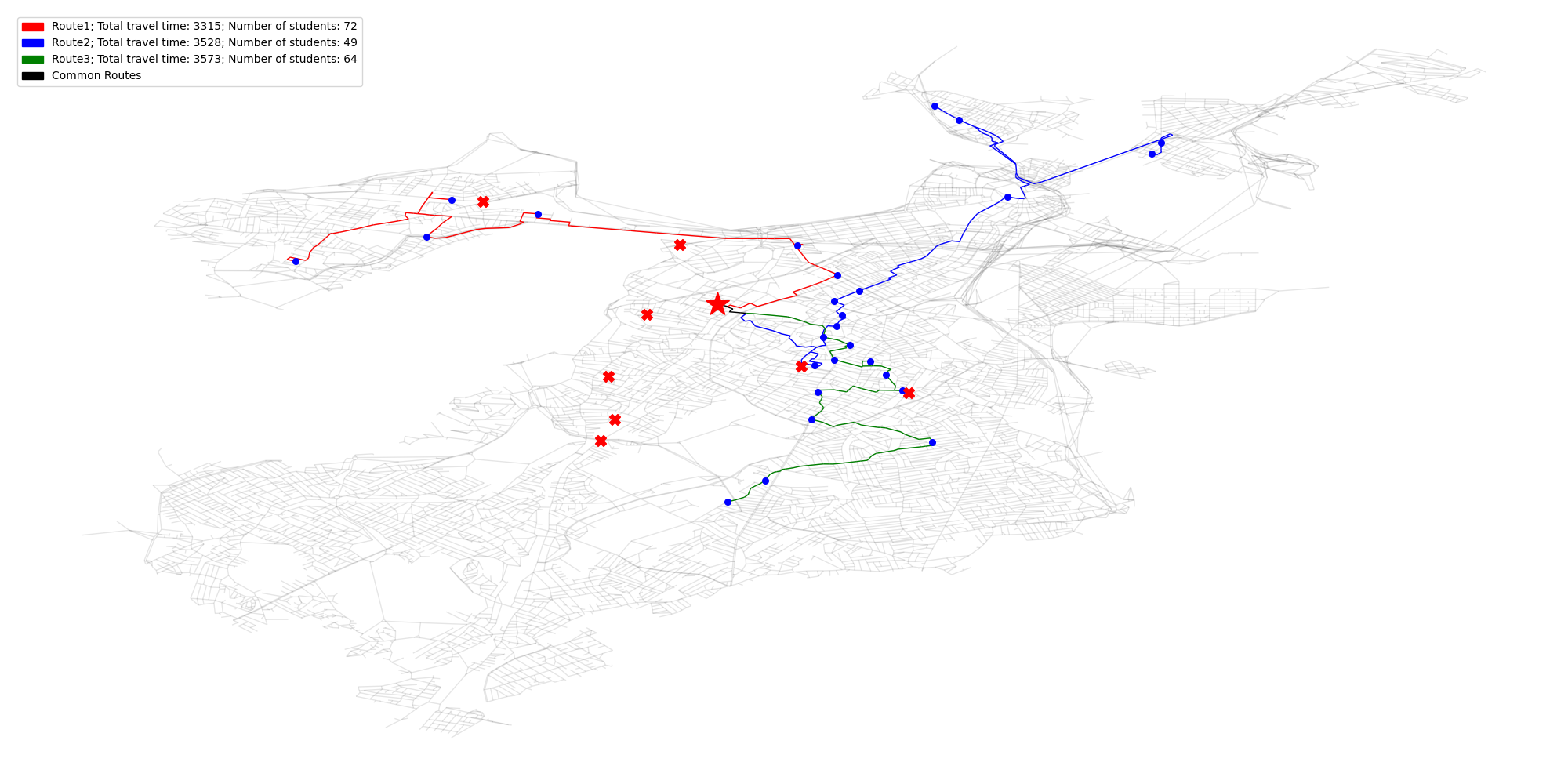}
\caption{School bus schedules for the Dick Bresciani (sub-optimal with $\beta = 1.5$, $\gamma = 0.4$)}
\end{figure*}

\begin{figure*}[h]
\centering
\includegraphics[scale=0.3]{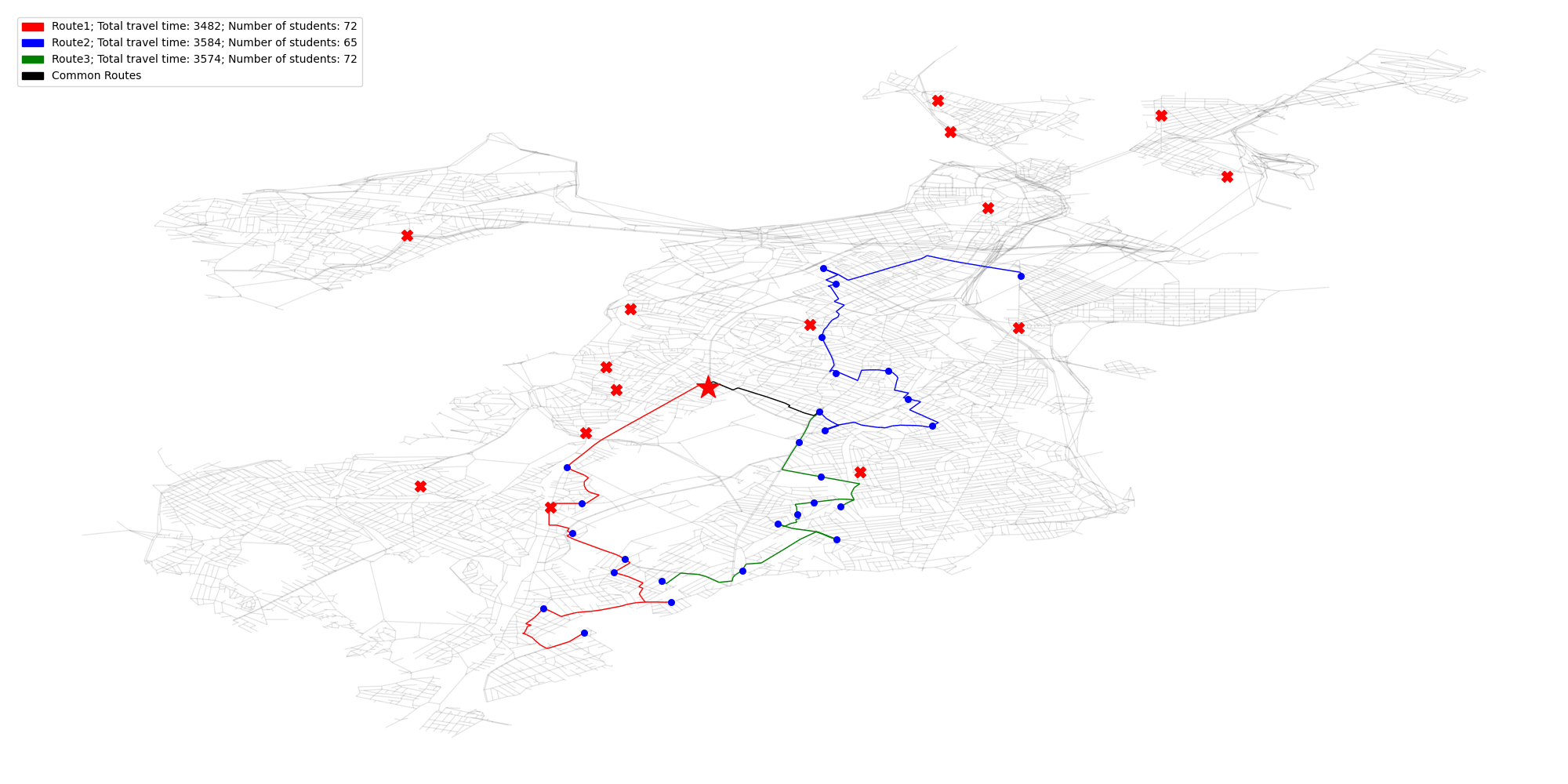}
\caption{School bus schedules for the Dutch Leonard (sub-optimal with $\beta = 2$, $\gamma = 0.4$)}
\end{figure*}

\begin{figure*}[h]
\centering
\includegraphics[scale=0.3]{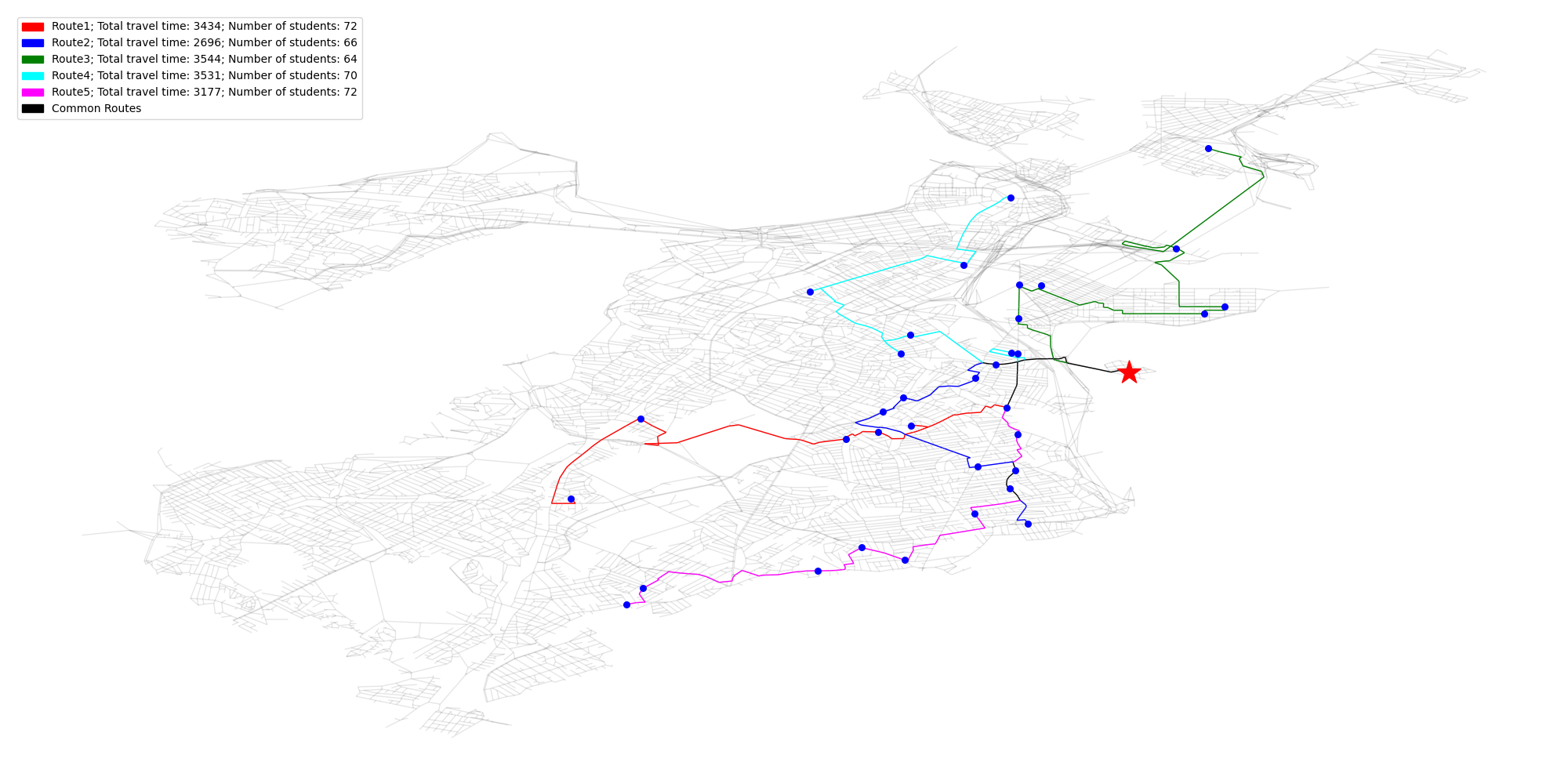}
\caption{School bus schedules for the Christian Vazquez (sub-optimal with $\beta = 3.5$, $\gamma = 0.4$)}
\end{figure*}

\begin{figure*}[h]
\centering
\includegraphics[scale=0.3]{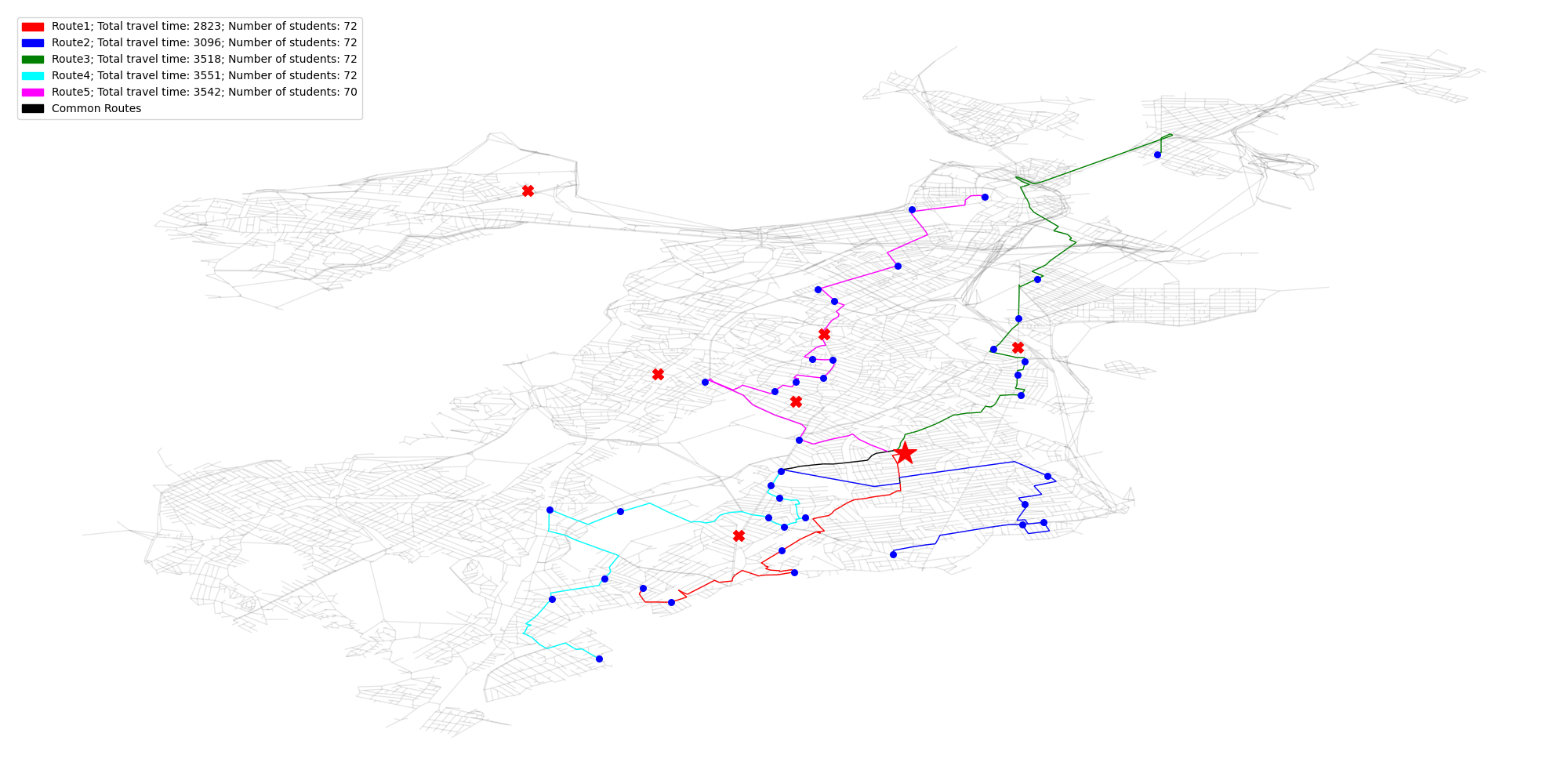}
\caption{School bus schedules for the Dennis Eckerley (sub-optimal with $\beta = 2.5$, $\gamma = 0.4$)}
\end{figure*}

\begin{figure*}[h]
\centering
\includegraphics[scale=0.3]{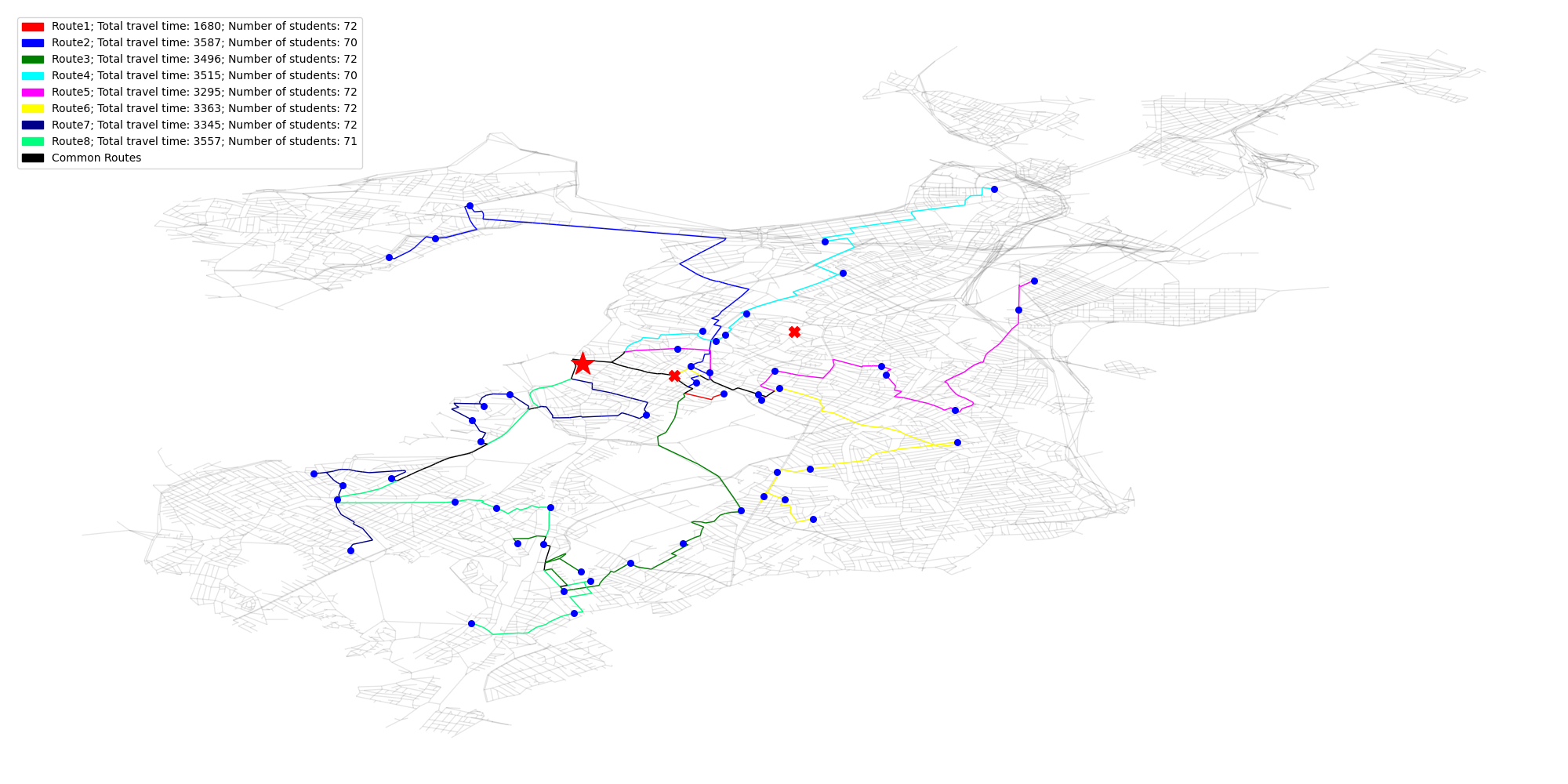}
\caption{School bus schedules for the Rick Ferrell (sub-optimal with $\beta = 2.5$, $\gamma = 0.4$)}
\end{figure*}

\FloatBarrier

\section{Benchmark results}
\label{apen:results_MIT}

This section states the school bus schedules generated by the single school route generation component from the BiRD algorithm. 
For figures in the following, red stars represent school locations and blue dots denote student locations.

\begin{figure*}[h]
\centering
\includegraphics[scale=0.3]{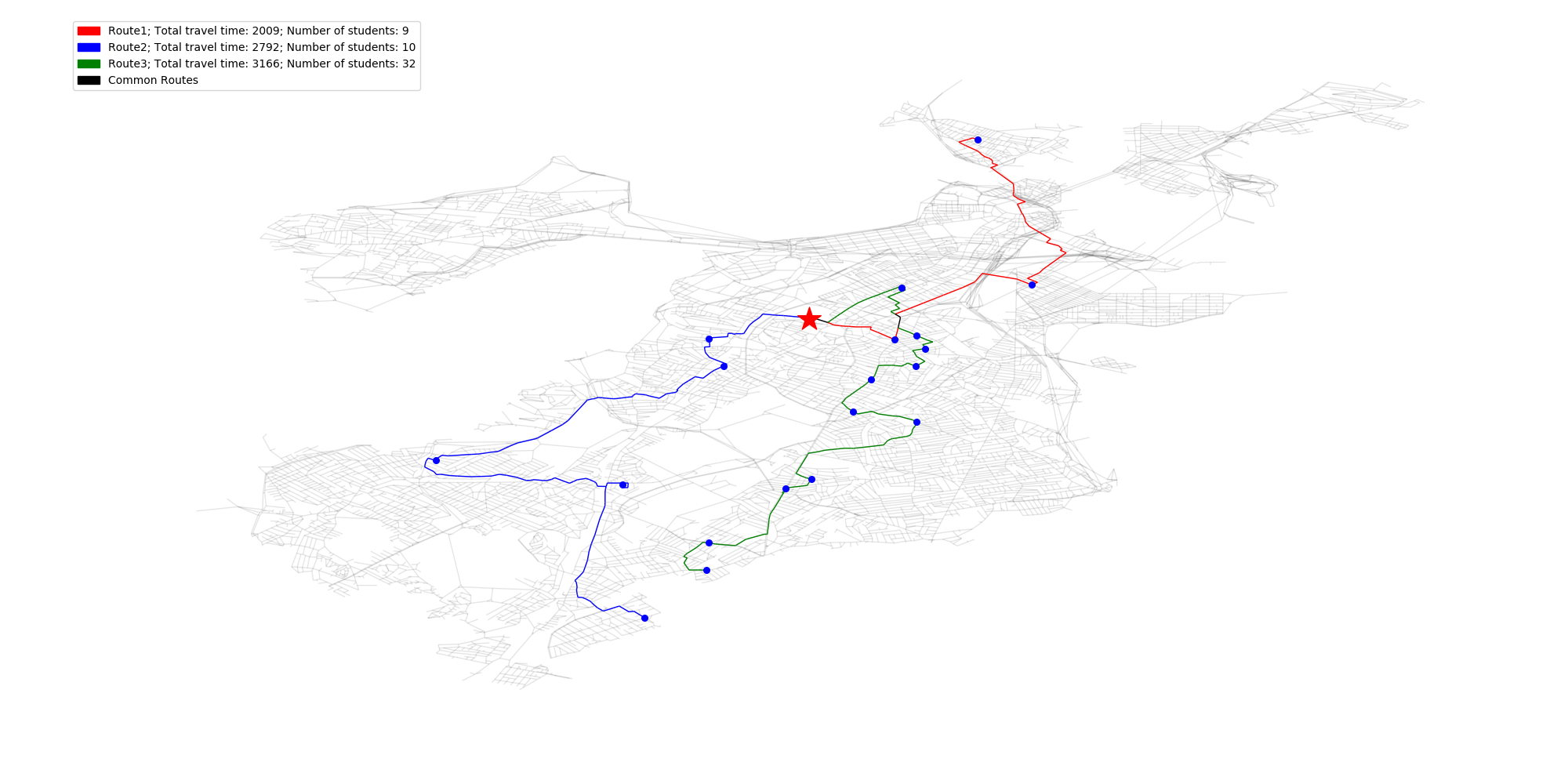}
\caption{School bus schedules for the Tommy Harper with BiRD algorithm}
\end{figure*}

\begin{figure*}[h]
\centering
\includegraphics[scale=0.3]{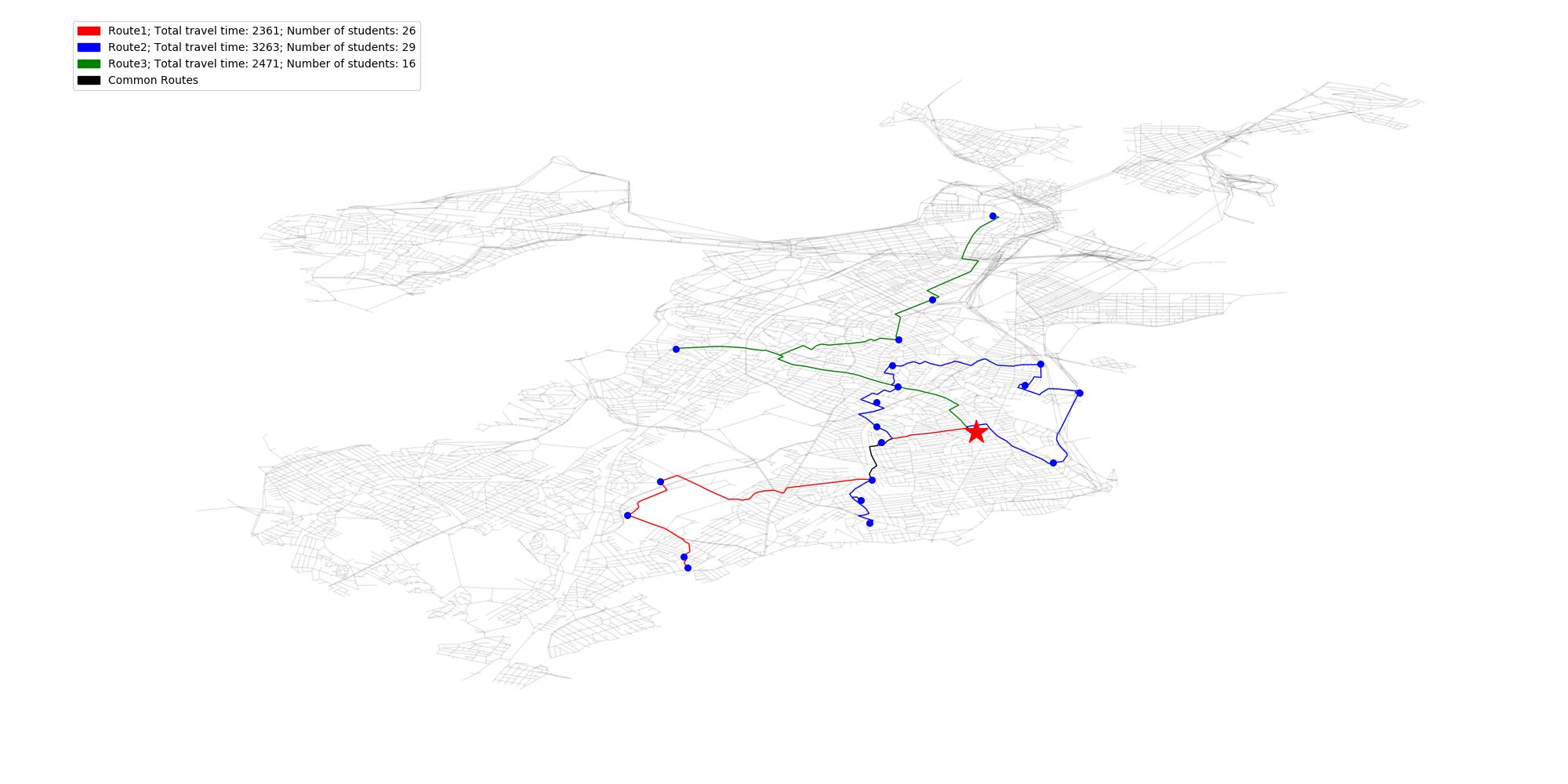}
\caption{School bus schedules for the Craig Kimbrel with BiRD algorithm}
\end{figure*}

\begin{figure*}[h]
\centering
\includegraphics[scale=0.3]{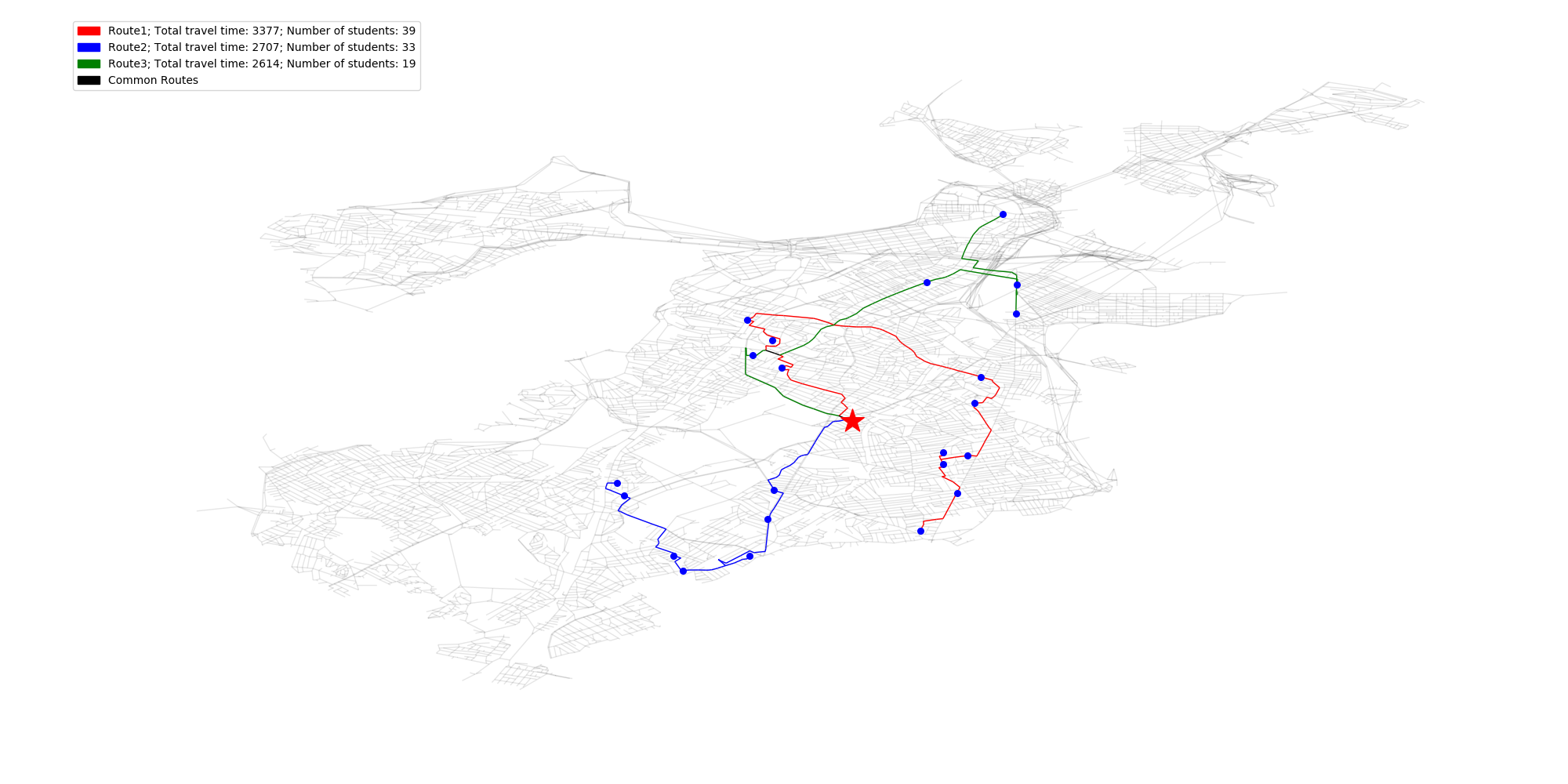}
\caption{School bus schedules for the Deven Marrero with BiRD algorithm}
\end{figure*}

\begin{figure*}[h]
\centering
\includegraphics[scale=0.3]{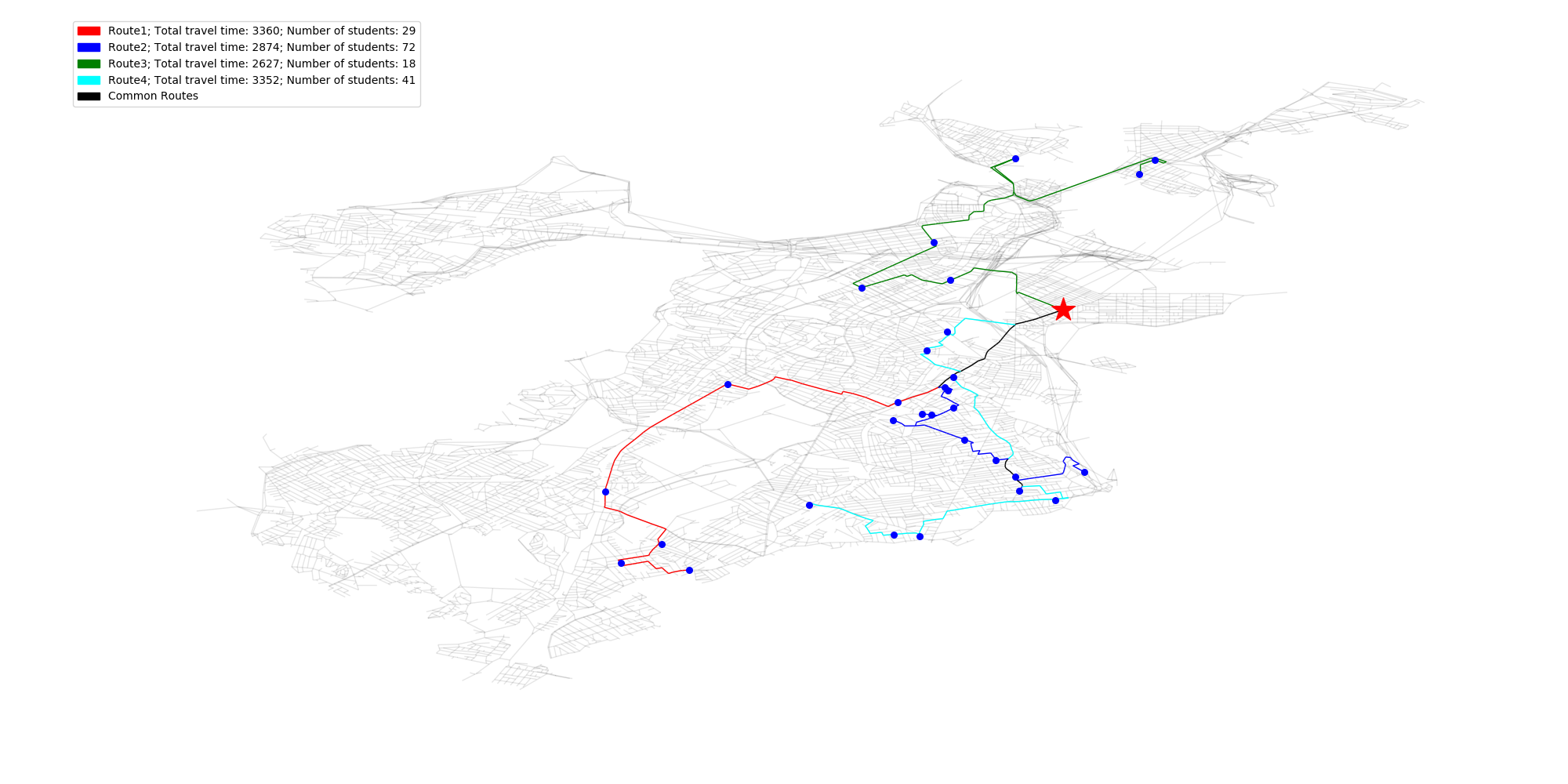}
\caption{School bus schedules for the Frank Malzone with BiRD algorithm}
\end{figure*}

\begin{figure*}[h]
\centering
\includegraphics[scale=0.3]{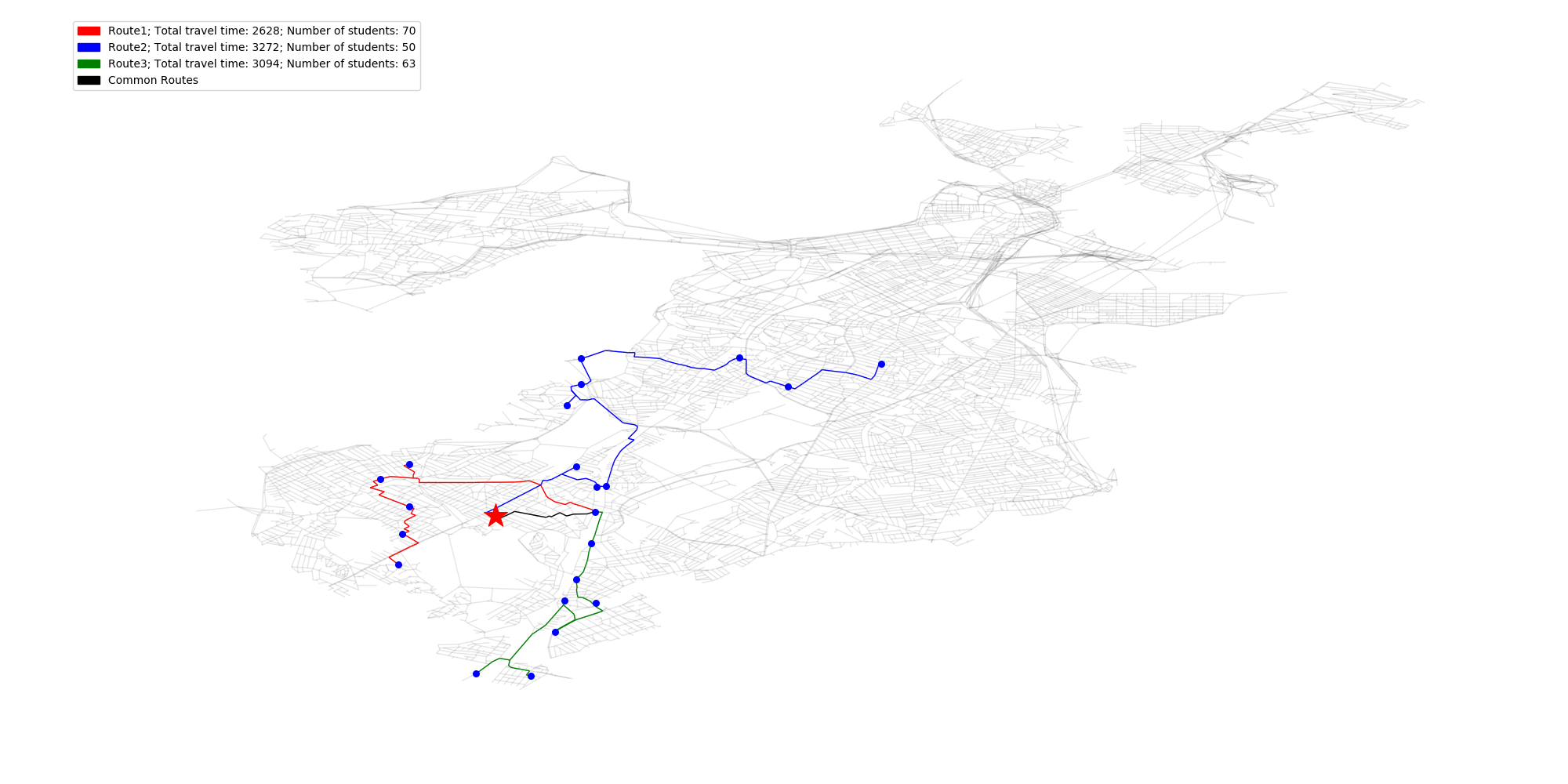}
\caption{School bus schedules for the Dick Williams with BiRD algorithm}
\end{figure*}

\begin{figure*}[h]
\centering
\includegraphics[scale=0.3]{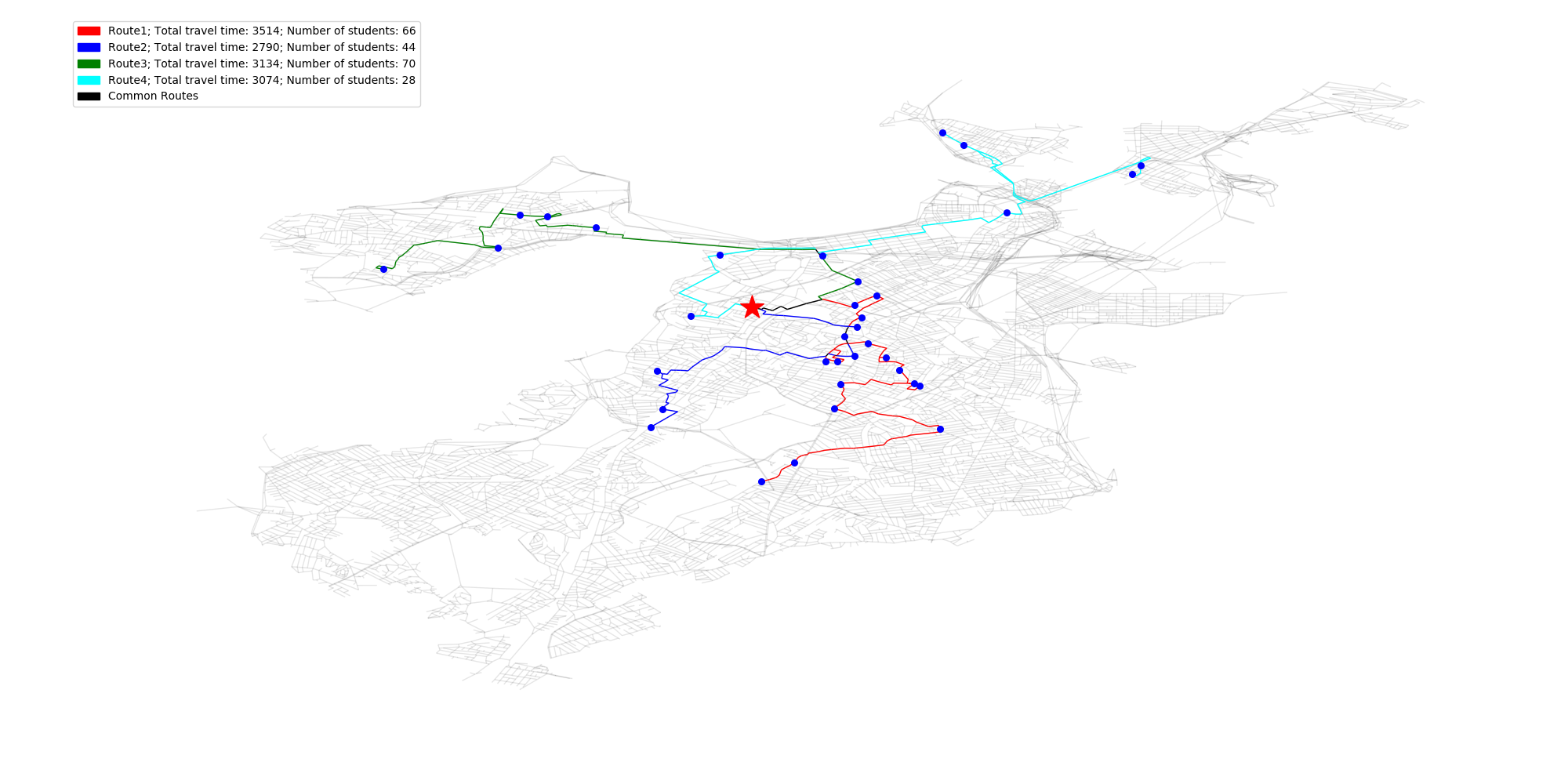}
\caption{School bus schedules for the Dick Bresciani with BiRD algorithm}
\end{figure*}

\begin{figure*}[h]
\centering
\includegraphics[scale=0.3]{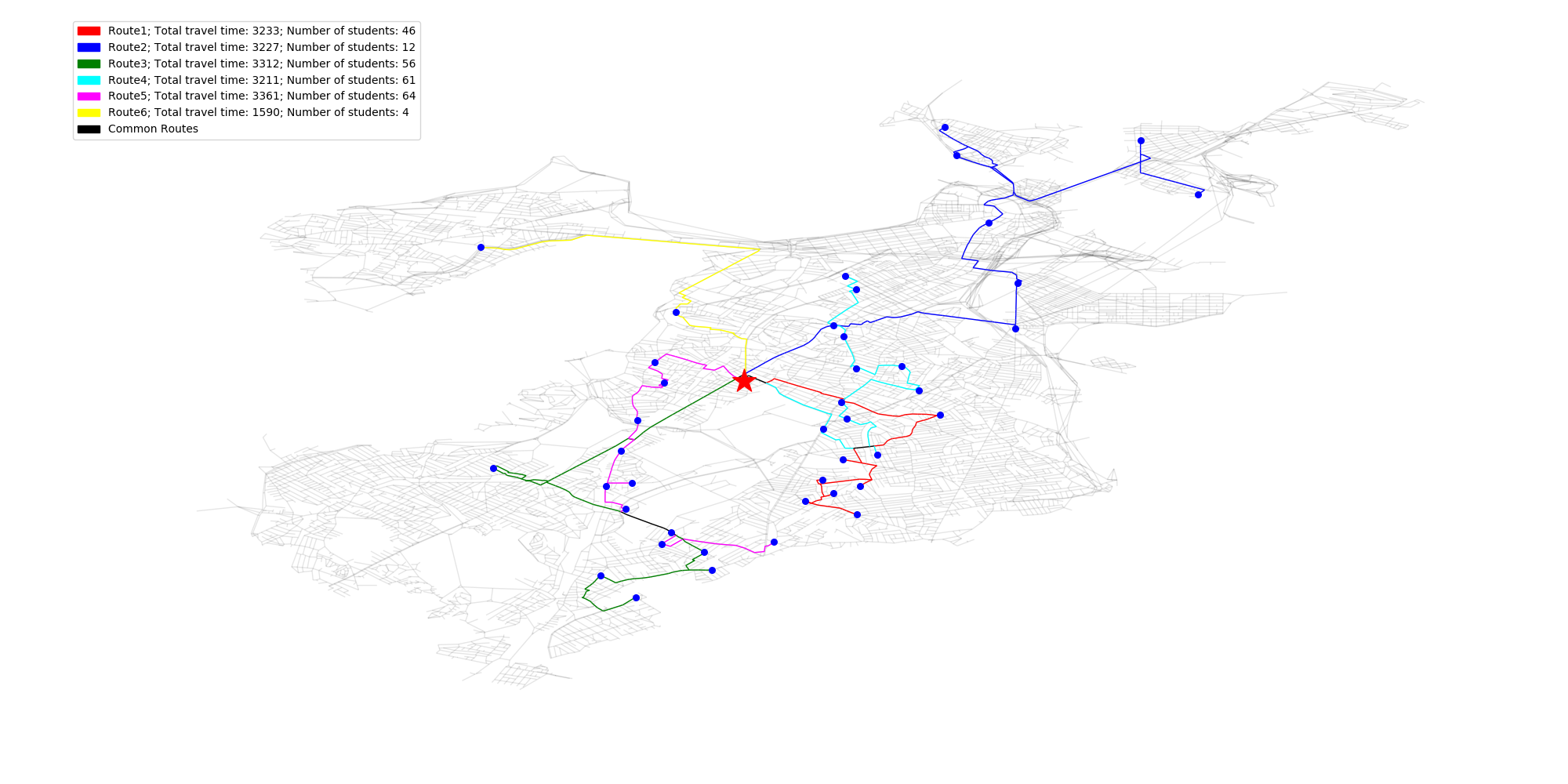}
\caption{School bus schedules for the Dutch Leonard with BiRD algorithm}
\end{figure*}

\begin{figure*}[h]
\centering
\includegraphics[scale=0.3]{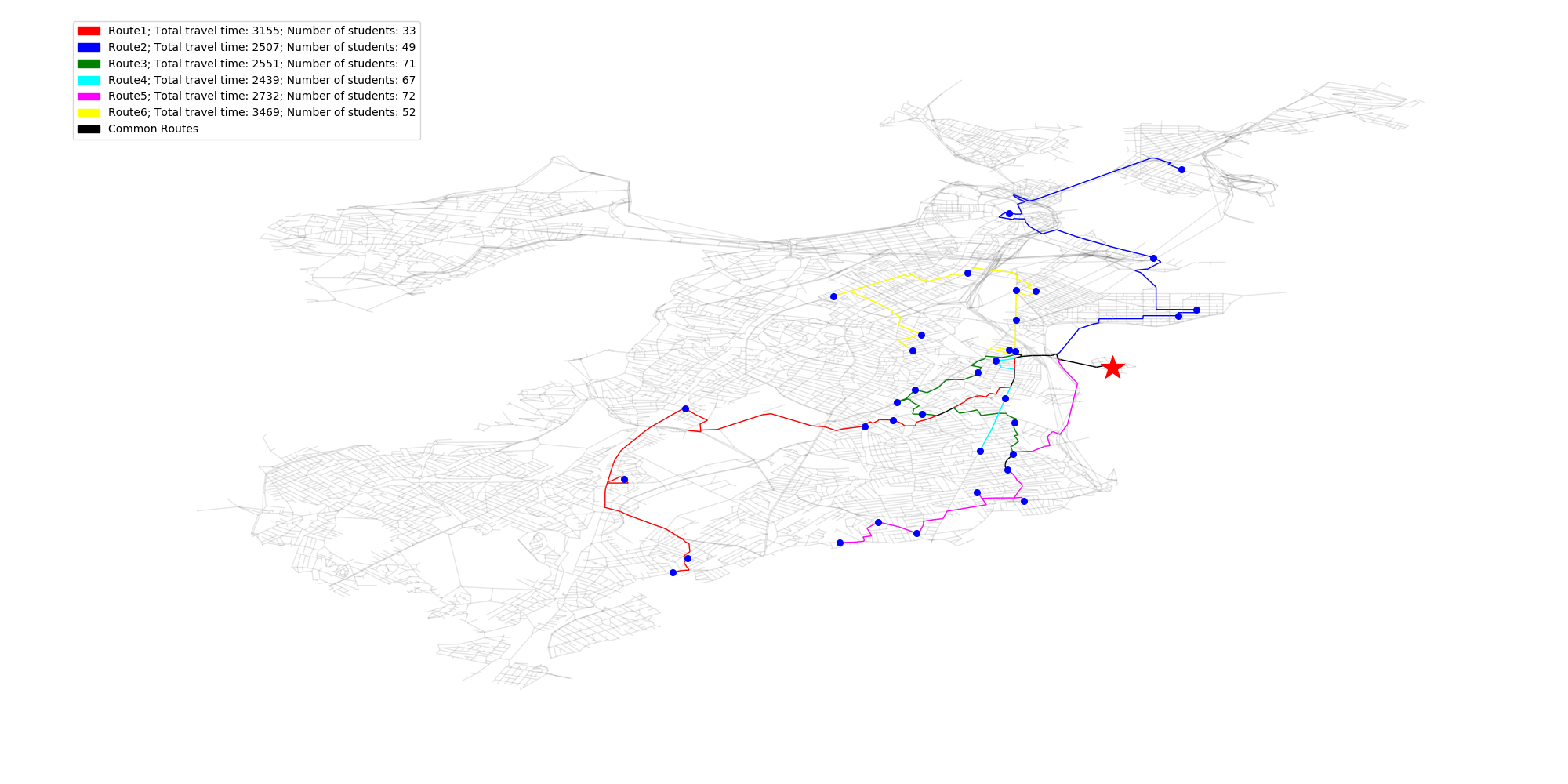}
\caption{School bus schedules for the Christian Vazquez with BiRD algorithm}
\end{figure*}

\begin{figure*}[h]
\centering
\includegraphics[scale=0.3]{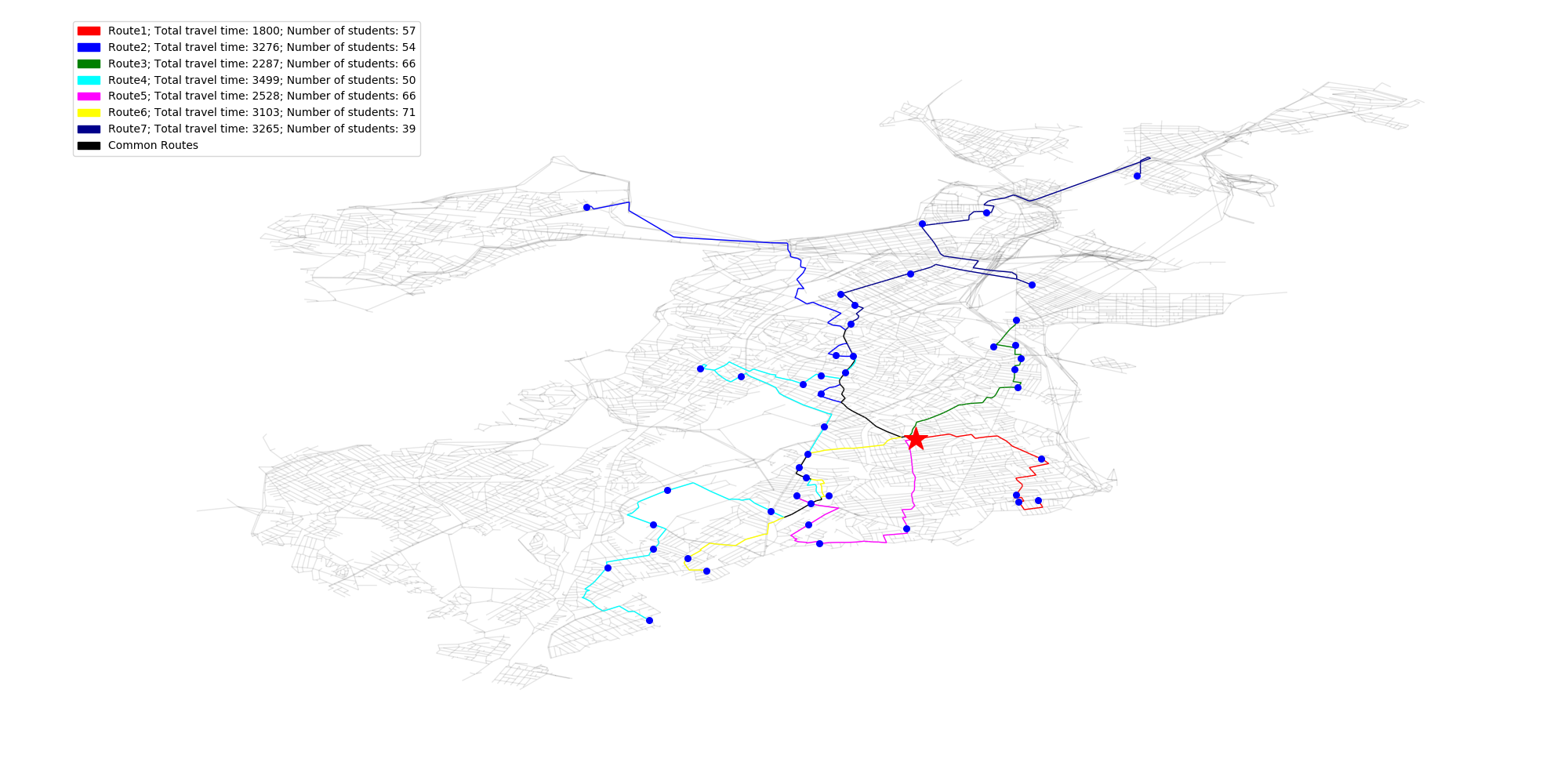}
\caption{School bus schedules for the Dennis Eckerley with BiRD algorithm}
\end{figure*}

\begin{figure*}[h]
\centering
\includegraphics[scale=0.3]{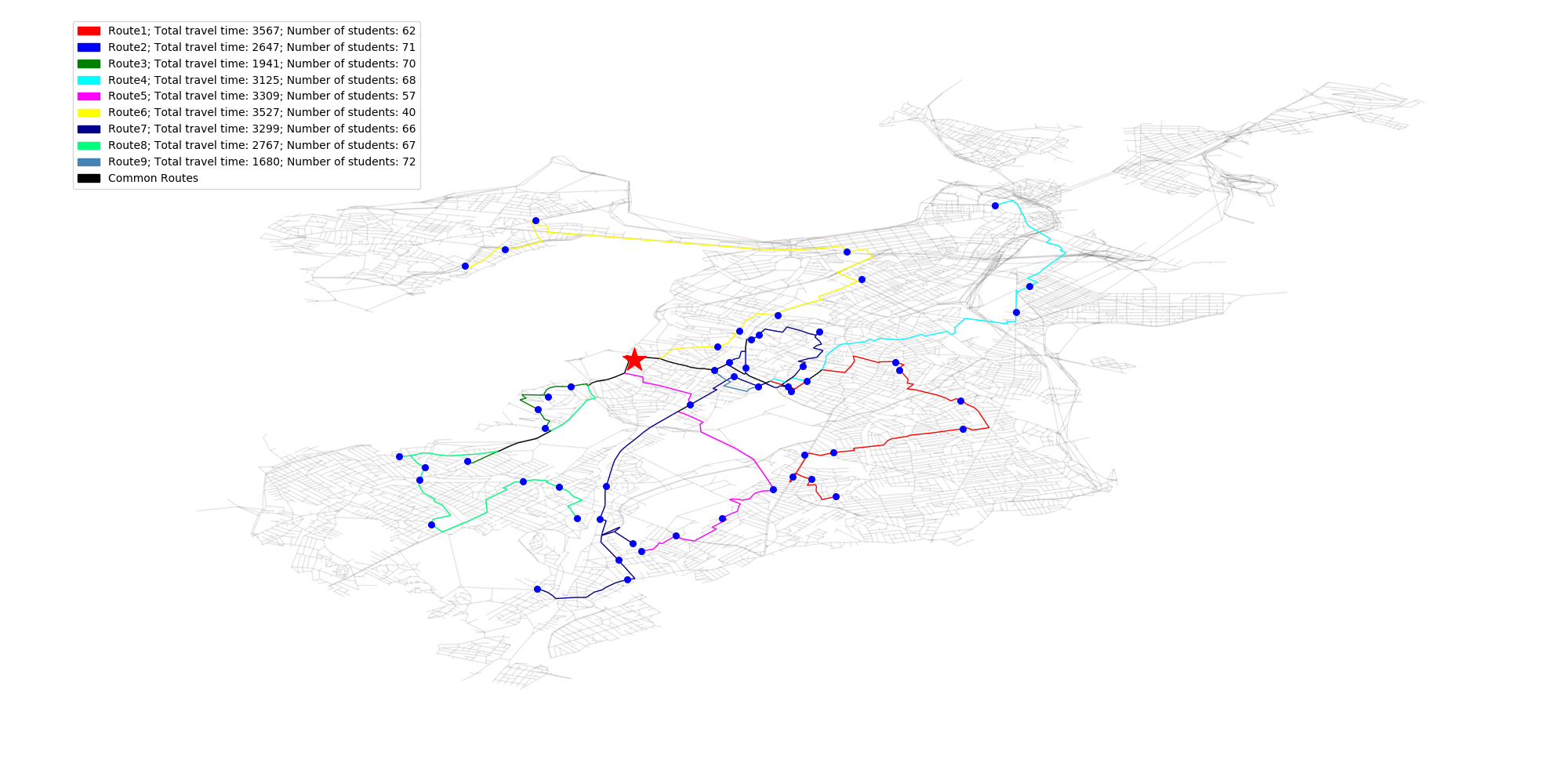}
\caption{School bus schedules for the Rick Ferrell with BiRD algorithm}
\end{figure*}

\FloatBarrier

\section{Additional sensitivity analyses}
\label{apend:addition_sensitivity}

In this section, we conduct the additional sensitivity analyses for both network compression techniques and the cost parameter for alternate modes using the large-scale BPS instance of the Dick Williams School with 183 students, where 28 students need door-to-door pickup.

\begin{figure}[h!]%
 \centering
 \subfloat[Number of buses]{\includegraphics[width=.49\linewidth]{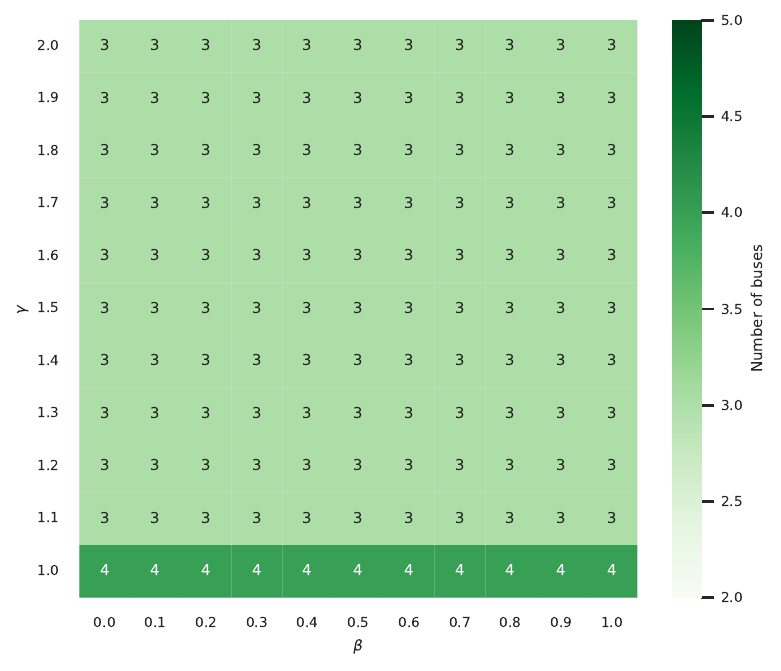}\label{fig:beta_gamma_sensitivity_a_addtion}}%
 \subfloat[Objective value]{\includegraphics[width=.49\linewidth]{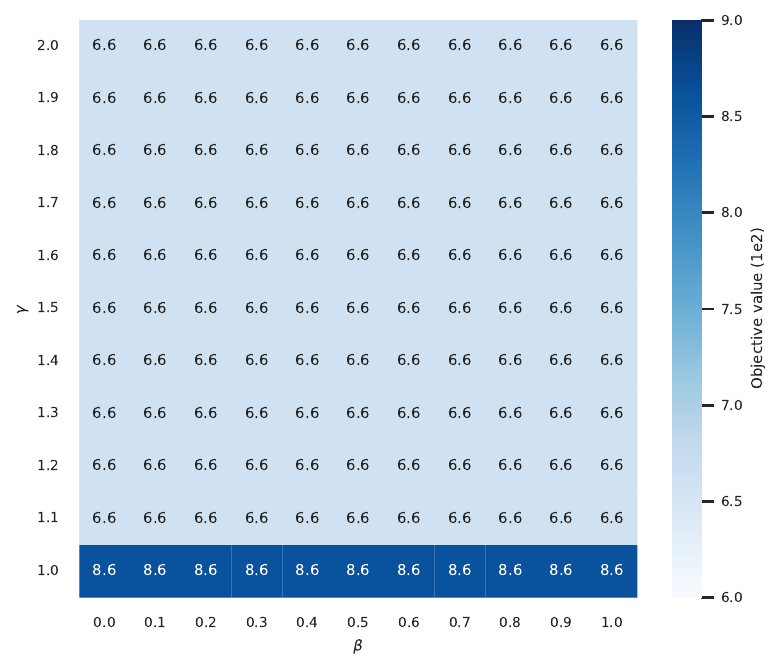}\label{fig:beta_gamma_sensitivity_b_addtion}}\\
 \subfloat[Size of trip configuration list]{\includegraphics[width=.49\linewidth]{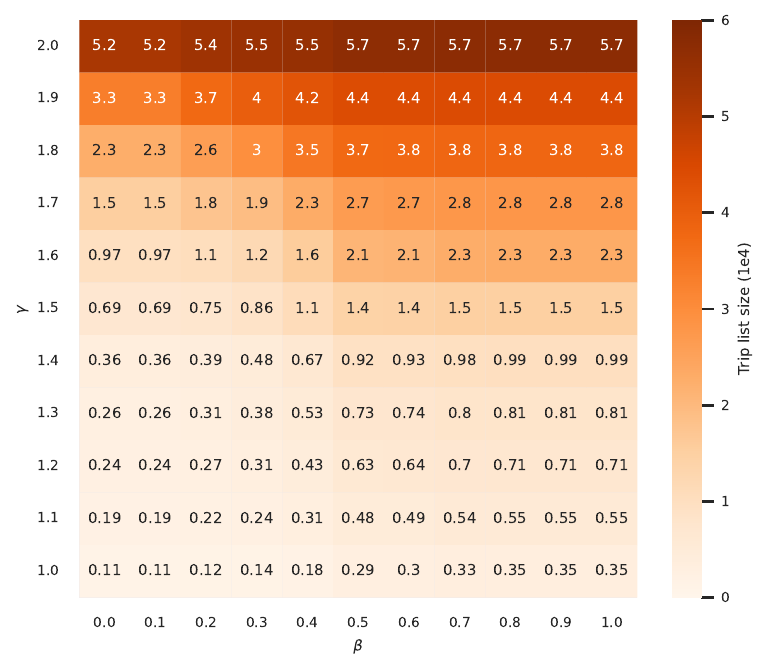}\label{fig:beta_gamma_sensitivity_c_addtion}}%
 \subfloat[Computation time]{\includegraphics[width=.49\linewidth]{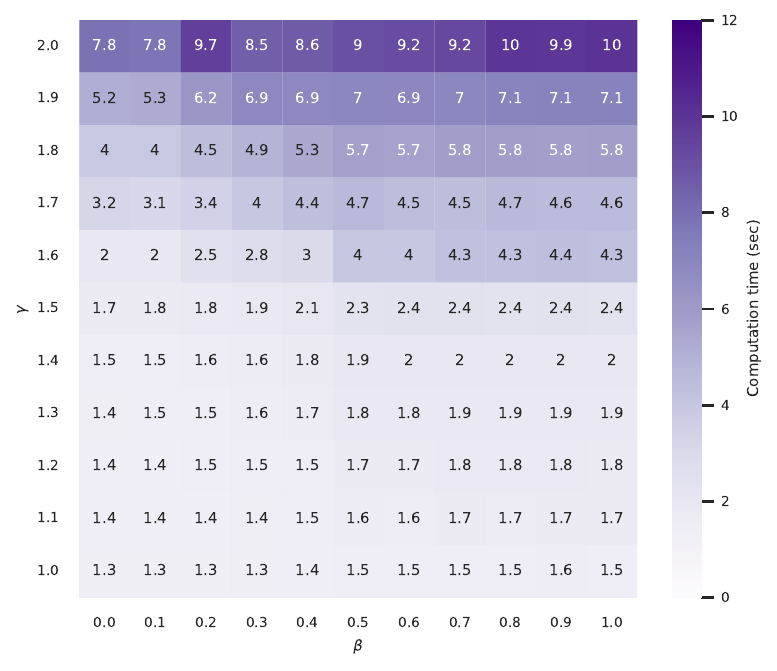}\label{fig:beta_gamma_sensitivity_d_addtion}}%
 \caption{Sensitivity analysis for parameters $\beta$ and $\gamma$ in network compression techniques.}%
 \label{fig:beta_gamma_sensitivity_addtion}%
\end{figure}

To measure the sensitivity of control parameters $\beta$ and $\gamma$ with respect the edge pruning technique, we choose $\beta$ values ranging from 1 to 2 with a step size of 0.1 and $\gamma$ values ranging from 0 to 1.0 with a step size of 0.1.
The sensitivity analyses results are shown in Figure \ref{fig:beta_gamma_sensitivity_addtion}.
In general, the computation time increases when the size of trip configuration list increases, which is induced by the increase of $\beta$ and $\gamma$.  
There are only two solutions over all combinations of $\beta$ and $\gamma$: 4 buses with objective value 860.54, 3 buses with objective value 659.98 (optimal bus schedules). 
We get the optimal school schedules when $\beta \geq 1.1$ regardless of $\gamma$.  

For the Dick Williams School instance without alternate modes, the size of trip configuration list without applying heuristics is 71672 and the optimal objective value is 659.98.
The minimum length of trip configuration list achieving the optimal school bus schedule is 1859 when $\beta = 1.1$ and $\gamma = 0 \text{ or } 0.1$.
With a proper choice of $\beta$ and $\gamma$, the optimal school bus schedule can be found by exploring only 2.6\% feasible trips.

To measure the sensitivity of the alternate modes cost parameter $\alpha_C^{direct}$, we consider a range from 0 (replaced by 0.1 since zero cost is unrealistic) to 5 with a step size of 0.5. 
The sensitivity analyses results are shown in Figure \ref{fig:alternate_sensitivity_addition}. When $\alpha_C^{direct}$ increases, the number of buses needed and the objective value increase while the number of students using alternate modes decreases.

\begin{figure}[h!]%
 \centering
 \subfloat[Number of buses]{\includegraphics[scale=.33]{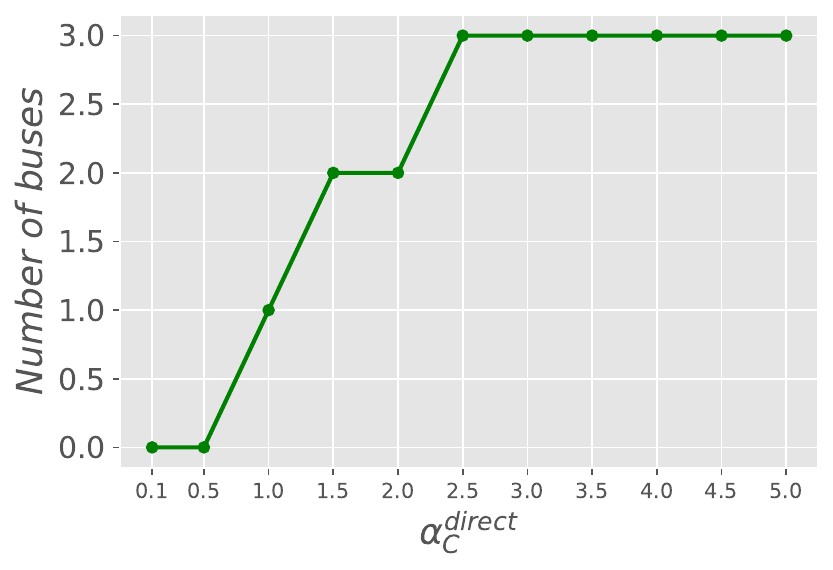}\label{fig:alternate_sensitivity_a_addition}}%
 \subfloat[Objective value]{\includegraphics[scale=.33]{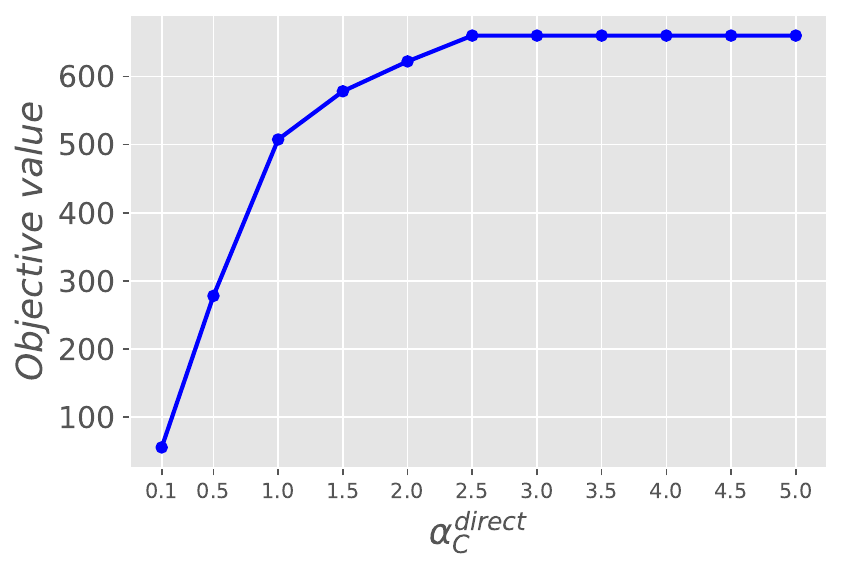}\label{fig:alternate_sensitivity_b_addition}}
 \subfloat[Number of students using alternate modes]{\includegraphics[scale=.33]{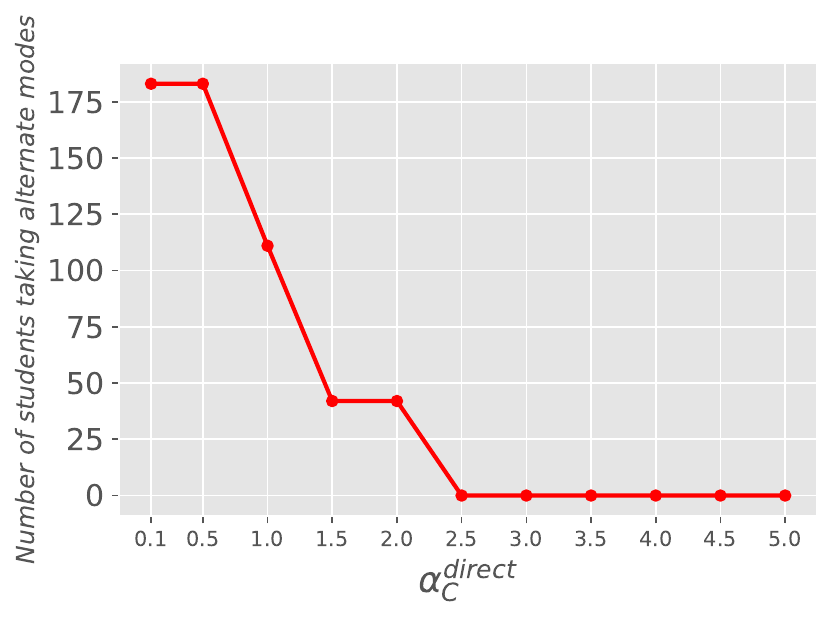}\label{fig:alternate_sensitivity_c_addition}}%
 \caption{Sensitivity analysis for the cost parameter $\alpha_C^{direct}$ of alternate modes.}%
 \label{fig:alternate_sensitivity_addition}%
\end{figure}



%% References







% \end{thebibliography}

\end{document}